\documentclass[12pt,oneside,english]{amsart}
\usepackage[T1]{fontenc}
\usepackage{geometry}
\usepackage{mathrsfs}
\usepackage{bm}
\usepackage{amstext}
\usepackage{amsthm}
\usepackage{amssymb}
\usepackage{cancel}
\usepackage{graphicx}
\usepackage{float}

\makeatletter
\numberwithin{equation}{section}
\numberwithin{figure}{section}
\theoremstyle{plain}
\newtheorem{thm}{\protect\theoremname}[section]
\theoremstyle{definition}
\newtheorem{defn}[thm]{\protect\definitionname}
\theoremstyle{definition}
\newtheorem{example}[thm]{\protect\examplename}
\theoremstyle{definition}
\newtheorem{rem}[thm]{\protect\remarkname}
\theoremstyle{plain}
\newtheorem{lem}[thm]{\protect\lemmaname}
\theoremstyle{plain}
\newtheorem{cor}[thm]{\protect\corollaryname}
\theoremstyle{plain}
\newtheorem{prop}[thm]{\protect\propositionname}
\theoremstyle{plain}

\usepackage[svgnames]{xcolor}
\usepackage[bookmarksnumbered=true]{hyperref} 
\hypersetup{
	colorlinks = true,
	linkcolor = Blue,
	anchorcolor = blue,
	citecolor = Green,
	filecolor = blue,
	urlcolor = FireBrick
}
\usepackage{bbm}

\MakeRobust{\ref}
\newcommand{\labeltext}[2]{
\@bsphack
\csname phantomsection\endcsname
\def\@currentlabel{#1}{\label{#2}}
\@esphack
}

\usepackage{comment}

\@ifundefined{showcaptionsetup}{}{%
 \PassOptionsToPackage{caption=false}{subfig}}
\usepackage{subfig}
\makeatother

\usepackage{babel}
\providecommand{\corollaryname}{Corollary}
\providecommand{\definitionname}{Definition}
\providecommand{\examplename}{Example}
\providecommand{\lemmaname}{Lemma}
\providecommand{\propositionname}{Proposition}
\providecommand{\remarkname}{Remark}
\providecommand{\theoremname}{Theorem}
\providecommand{\conjecturename}{Conjecture}

\DeclareRobustCommand{\SkipTocEntry}[5]{}

\newcommand{\mR}{\mathbb{R}}   
\newcommand{\abs}[1]{\lvert #1 \rvert}  
\newcommand{\norm}[1]{\lVert #1 \rVert}  

\newcommand{\ol}[1]{\overline{#1}}

\newcommand{\eps}{\epsilon}

\newcommand{\dbar}{\overline{\partial}}

\newcommand{\p}{\partial}

\newcommand{\dd}{\mathbb{D}}

\DeclareMathOperator{\Bal}{Bal_{\it k}}

\DeclareMathOperator{\supp}{supp}
\DeclareMathOperator{\dist}{dist}
\newcommand{\dm}{\mathsf{m}}

\begin{document}
\begin{sloppypar}
\title[Quadrature domains for the Helmholtz equation]{Quadrature domains for the Helmholtz equation with applications to non-scattering phenomena}

\author{Pu-Zhao Kow}
\address{Department of Mathematics and Statistics, P.O. Box 35 (MaD), FI-40014 University of Jyv\"{a}skyl\"{a}, Finland.}
\email{\href{mailto:pu-zhao.pz.kow@jyu.fi}{pu-zhao.pz.kow@jyu.fi}}

\author{Simon Larson}
\address{Mathematical Sciences, Chalmers University of Technology and the University of Gothenburg, SE-412 96 G\"{o}teborg, Sweden. }
\email{\href{mailto:larsons@chalmers.se}{larsons@chalmers.se}}

\author{Mikko Salo}
\address{Department of Mathematics and Statistics, P.O. Box 35 (MaD), FI-40014 University of Jyv\"{a}skyl\"{a}, Finland.}
\email{\href{mailto:mikko.j.salo@jyu.fi}{mikko.j.salo@jyu.fi}}

\author{Henrik Shahgholian}
\address{Department of Mathematics, KTH Royal Institute of Technology, SE-100 44 Stockholm, Sweden.}
\email{\href{mailto:henriksh@kth.se}{henriksh@kth.se}}

\begin{abstract}
In this paper, we introduce quadrature domains for the Helmholtz equation. We show existence results for such domains and implement the so-called partial balayage procedure. We also give an application to inverse scattering problems, and show that there are non-scattering domains for the Helmholtz equation at any positive frequency that have inward cusps.
\end{abstract}

\keywords{quadrature domain; non-scattering phenomena; mean value theorem; Helmholtz equation; acoustic equation; metaharmonic functions; partial balayage}
\subjclass[2020]{35J05; 35J15; 35J20; 35R30; 35R35}
\maketitle

\tableofcontents

\section{Introduction and main results}

\addtocontents{toc}{\SkipTocEntry}
\subsection{Background}

This work is motivated by a problem in inverse scattering theory, but it raises questions of independent interest in the context of quadrature domains and free boundary problems. We recall that a bounded domain $D \subset \mathbb{R}^n$ is called a \emph{quadrature domain} (for harmonic functions), corresponding to a measure $\mu$ with $\mathrm{supp}(\mu) \subset D$, if 
\begin{equation} \label{quadrature_harmonic}
\int_D h(x) \,dx = \int h(x) \,d\mu(x)
\end{equation}
for every harmonic function $h \in L^1(D)$. More generally, one can consider distributions $\mu \in \mathscr{E}'(D)$. In the most classical case one is interested in domains $D$ for which $\mu$ is supported at finitely many points, so that \eqref{quadrature_harmonic} reduces to a quadrature identity for computing integrals of harmonic functions.

Quadrature domains can be viewed as a generalization of the mean value theorem (MVT) for harmonic functions. Indeed, we can rephrase the MVT for harmonic functions as follows: 
\[
B_{r}(a) \text{ is a quadrature domain with }\mu = \mathsf{m}(B_{r}(a))\delta_{a},
\]
where $\delta_{a}$ is the Dirac measure at $a$, $\mathsf{m}$ denotes the Lebesgue measure in $\mathbb{R}^{n}$ (i.e. $d\mathsf{m}=dx$) and $B_{r}(a)$ is the ball of radius $r$ centered at $a$. 
In general, the boundary of a quadrature domain is a free boundary in an obstacle-type  problem (see \cite{PSU12FreeBoundary}), and hence near any given point $z \in \partial D$ the domain $D$ is either smooth or $D^c$ has zero density at $z$. Various examples can be constructed via complex analysis, for example, the cardioid domain in Example~\ref{exa:cardioid} below. We refer to \cite{Dav74SchwartzFunction}, \cite{Sak83BalayageQuadrature}, and \cite{GS05QuadratureDomain} for further background.
\labeltext{$\mathsf{m}$ Lebesgue measure in $\mathbb{R}^{n}$}{index:LebesgueMeasure} \labeltext{$B_{r}(a)$ ball of radius $r$ centered at $a$}{index:Ball} 
\labeltext{$B_{r}$ ball of radius $r$ centered at origin}{index:BallOrigin}

The inverse scattering problems studied in \cite{SS21NonscatteringFreeBoundary} lead to a related concept, for solutions of the Helmholtz equation $(\Delta + k^2) u = 0$, where $k \geq 0$ is a frequency. This setting gives rise to various interesting questions. We are not aware of earlier work on quadrature domains for $k > 0$, and in this article we only give some first steps. In addition, we show that any quadrature domain is a \emph{non-scattering domain} (cf.\ Definition \ref{def_ns_domain}) if it admits an incident wave that is positive on its boundary. In \cite{SS21NonscatteringFreeBoundary} it was observed that in the case $k=0$ quadrature domains are non-scattering domains, and hence there are non-scattering domains having inward cusps. Corollary \ref{cor:k-quad-nonscattering} below provides a similar result valid for all $k > 0$.

\addtocontents{toc}{\SkipTocEntry}
\subsection{Notation} Here we gather recurring notation and definitions, with reference to relevant pages. We also mention here that all functions and measures will be real-valued unless stated  otherwise.

\bigskip

\noindent \ref{index:LebesgueMeasure}, \pageref{index:LebesgueMeasure}\\
\noindent \ref{index:Ball}, \pageref{index:Ball} \\
\noindent \ref{index:BallOrigin}, \pageref{index:BallOrigin} \\
\noindent \ref{index:quadratureDefinition}, \pageref{index:quadratureDefinition} \\
\noindent \ref{index:UnitDisk}, \pageref{index:UnitDisk} \\
\noindent \ref{index:Bessel}, \pageref{index:Bessel} \\
\noindent \ref{index:ZeroBessel}, \pageref{index:ZeroBessel} \\
\noindent \ref{index:BesselSecond}, \pageref{index:BesselSecond} \\
\noindent \ref{index:ConstantRnk}, \pageref{index:ConstantRnk} \\
\noindent \ref{index:RealFundamentalSolution}, \pageref{index:RealFundamentalSolution} \\
\noindent \ref{index:potential-Uk}, \pageref{index:potential-Uk} \\
\noindent \ref{index:FamilyFunctionsMu}, \pageref{index:FamilyFunctionsMu} \\
\noindent \ref{index:ConstantMVTRepeat}, \pageref{index:ConstantMVTRepeat} \\
\noindent \ref{index:SaturatedSetDmu}, \pageref{index:SaturatedSetDmu} \\
\noindent \ref{index:Setomegamu}, \pageref{index:Setomegamu}\\

\addtocontents{toc}{\SkipTocEntry}
\subsection{Main results}

We begin with a definition generalizing \eqref{quadrature_harmonic}.

\begin{defn} \label{def:k-quadrature-domain}
Let $k > 0$. A bounded open set $D \subset \mathbb{R}^n$ {\rm (}not necessarily connected{\rm )} is called a \emph{quadrature domain for $(\Delta+k^2)$}, or a \emph{$k$-quadrature domain}, corresponding to a distribution $\mu \in \mathscr{E}'(D)$, if \labeltext{Definition of $k$-quadrature domain}{index:quadratureDefinition}
\[
\int_D w(x) \,dx = \langle \mu, w \rangle
\]
for all $w \in L^1(D)$ satisfying $(\Delta + k^2) w = 0$ in $D$. 
\end{defn}

We remark that solutions of $(\Delta + k^2) w = 0$ are sometimes called \emph{metaharmonic functions}, see e.g.\  \cite[Section~4]{Kuz19MeanValueProperty} or \cite{Fri57metaharmonic} for a discussion. It is important that $\supp(\mu)$ has to be a subset of $D$ (see however \cite[Lemma 2.8]{KM96NewtonianPotential} for a discussion that weakens this assumption for harmonic functions). Indeed, that $\supp(\mu) \subset D$ implies that the distributional pairing $\langle \mu, w \rangle$ is well defined, because solutions of $(\Delta+k^2)w=0$ are smooth in $D$. Furthermore, without this requirement the existence of a distribution satisfying the definition would be trivial, indeed one could choose $\mu = \chi_D$. 

The first question is whether $k$-quadrature domains even exist for $k > 0$. This is indeed the case. In fact, balls are always $k$-quadrature domains. This is a consequence of a MVT for the Helmholtz equation which goes back to H.\ Weber \cite{Web1868,Web1869} (see also \cite{Kuz19MeanValueProperty},
\cite{Kuz21Metaharmonic}, or \cite[p.~289]{CH89MethodsMathematicsPhysicsII}). 
The MVT takes the form 
\[
\int_{B_r(a)} w(x) \,dx = c_{n,k,r}^{{\rm MVT}} w(a)
\]
whenever $w \in L^1(B_r(a))$ and $(\Delta + k^2) w = 0$ in $B_r(a)$. However, unlike for harmonic functions, the constant $c_{n,k,r}^{{\rm MVT}}$ has varying sign depending on $k, r$. In particular, the constant vanishes when $J_{n/2}(kr) = 0$ where $J_{\alpha}$ denotes the Bessel function of the first kind. More details are given in Appendix~\ref{appen:GreenFunction}. It follows that unions of disjoint balls are also $k$-quadrature domains corresponding to linear combinations of delta functions. Choosing two balls whose closures intersect at one point furnishes an example of a $k$-quadrature domain whose boundary is not smooth.
\labeltext{$J_{\alpha}$ Bessel function of first kind}{index:Bessel}

In order to make further progress we consider a PDE characterization of $k$-quadrature domains. One can show (see Proposition \ref{prop_quadrature_pde_equivalence}) that $D$ is a $k$-quadrature domain corresponding to $\mu \in \mathscr{E}'(D)$ if and only if there is a distribution $u \in \mathscr{D}'(\mR^n)$ satisfying 
\begin{equation} \label{oqd_eq}
\begin{cases}
\hfil (\Delta + k^2) u = \chi_D - \mu &\text{ in $\mR^n$}, \\
\hfil u = |\nabla u| = 0 &\text{ in $\mR^n \setminus D$}.
\end{cases}
\end{equation}
Note that by elliptic regularity the distribution $u$ solving $(\Delta + k^2) u = \chi_D$ near $\p D$ must be $C^1$ near $\p D$, and thus the condition that $u$ and $\nabla u$ vanish in $\mR^n \setminus D$ (instead of $\mR^n \setminus \ol{D}$) makes sense. The following result is a local version of the above fact, characterizing domains $D$ that are $k$-quadrature domains for some distribution $\mu$. However, there is no reason to expect that $\mu$ could be chosen to have support at finitely many points.

\begin{thm} \label{thm_pde_quadrature_local}
Let $k > 0$, and let $D$ be a bounded open set in $\mR^n$. Then $D$ is a $k$-quadrature domain for some $\mu \in \mathscr{E}'(D)$ if and only if there is a neighborhood $U$ of $\p D$ in $\mR^n$ and a distribution $u \in \mathscr{D}'(U)$ satisfying 
\begin{equation}
\begin{cases}
\hfil (\Delta + k^2) u = \chi_D &\text{ in $U$}, \\
\hfil u = |\nabla u| = 0 &\text{ in $U \setminus D$}.
\end{cases} \label{eq:local-PDE-characterization}
\end{equation}
Moreover, if $D$ is a $k$-quadrature domain for some $\mu \in \mathscr{E}'(D)$, then $D$ is also a $k$-quadrature domain for some measure $\tilde{\mu}$ having smooth density with respect to Lebesgue measure.
\end{thm}

\begin{rem}
If $u$ is as in Theorem \ref{thm_pde_quadrature_local}, then clearly 
\begin{equation}
\begin{cases}
\hfil \Delta u = f \chi_D &\text{ in $U$}, \\
\hfil u = |\nabla u| = 0 &\text{ in $U \setminus D$},
\end{cases} \label{eq:harmonic-continuation}
\end{equation}
with $f = 1 - k^2 u$. Extending $u$ from a neighborhood of $\p D$ into some distribution in $\mR^n$ with $u = |\nabla u| = 0$ in $\mR^n \setminus D$ shows that we have an analogue of \eqref{oqd_eq} with $k=0$ and with $\chi_D$ replaced by $f \chi_D$. Thus any  $k$-quadrature domain is a weighted $0$-quadrature domain. Since the weight $f$ is positive on $\p D$, free boundary regularity results for weighted $0$-quadrature domains apply also to $k$-quadrature domains. In particular, such a domain has locally either smooth boundary or its complement is thin in the sense of minimal diameter (see \cite[page 109]{PSU12FreeBoundary}). We also remark that when $k=0$ the equation \eqref{oqd_eq} is related to harmonic continuation of potentials, see \cite{Isa90InverseSourceProblems} for further information.
\end{rem}

Theorem \ref{thm_pde_quadrature_local} has an immediate consequence showing that domains with real-analytic boundary are $k$-quadrature domains.

\begin{cor} \label{thm_quadrature_analytic}
If $k > 0$, then any bounded open set $D \subset \mR^n$ with real-analytic boundary is a $k$-quadrature domain.
\end{cor}
\begin{proof}
Since $\p D$ is real-analytic, we can use the Cauchy--Kowalevski theorem to find a real-analytic function $u$ near $\p D$ satisfying 
\[
\begin{cases}
\hfil (\Delta + k^2) u = 1 \text{ near $\p D$}, \\
\hfil u|_{\p D} = \p_{\nu} u|_{\p D} = 0,
\end{cases}
\]
where $\partial_\nu$ denotes the derivative in the normal direction to $\p D$. 
We redefine $u$ to be zero outside $D$. One can directly check that $u$ and $\nabla u$ are Lipschitz continuous across $\p D$. Hence $u$ will be $C^{1,1}$ near $\p D$ and will satisfy the condition in Theorem~\ref{thm_pde_quadrature_local}. 
This proves that $D$ is a $k$-quadrature domain.
\end{proof}

The next result gives further examples of $k$-quadrature domains in two dimensions. 

\begin{thm} \label{thm_quadrature_complex}
Let $k > 0$, and let $\dd$ be the unit disc in $\mathbb{R}^{2} \cong \mathbb{C}$. Suppose that $D = \varphi(\dd)$ where $\varphi$ is a complex analytic function in a neighbourhood of $\overline{\dd}$ such that $\varphi\colon \dd \to D$ is bijective. Then $D$ is a $k$-quadrature domain. \labeltext{$\mathbb{D}$ unit disk}{index:UnitDisk}
\end{thm}

Domains $D$ as in Theorem \ref{thm_quadrature_complex} include cardioid type domains and domains with double points. Examples and further properties of these domains are given in the end of Section \ref{sec:complex-analysis}.

We also study $k$-quadrature domains from the potential theoretic point of view. More precisely, we construct some $k$-quadrature domains by using \emph{partial balayage}, that is, given a non-negative compactly supported Radon measure $\mu$, we construct a measure $\nu$ by distributing the mass of $\mu$ more uniformly. By investigating the structure of $\nu$, we then construct a $k$-quadrature domain $D$ with respect to $\mu$. For the case when $k=0$ this procedure is classical, see e.g.\ \cite{GR18PartialBalayageManifold,Gus90QuadratureDomains,Gus04LecturesBalayage,Sak83BalayageQuadrature}. In this paper, we give similar results for $k>0$ and many of our results and proofs follow those in the case of $k=0$ as presented in~\cite{Gus90QuadratureDomains,Gus04LecturesBalayage}. In this direction, our main goal is to prove the following theorem:

\begin{thm}[see also Theorem~{\rm \ref{thm:k-quadrature-balayage}}]
\label{thm:main2} 
Let $\mu$ be a positive measure supported in a ball of radius $\epsilon>0$. There exists a constant $c_{n}>0$ depending only on the dimension such that if 
\begin{equation}
0 < k < \frac{c_{n}}{\mu(\mathbb{R}^n)^{1/n}} \quad \mbox{and}\quad \epsilon < c_n \mu(\mathbb{R}^n)^{1/n}, \label{eq:k-range-first}
\end{equation}
then there exists an open connected set $D$ with real-analytic boundary which is a $k$-quadrature domain for $\mu$. 
Moreover, for each $w\in L^1(D)\cap L^1(\mu)$ satisfying $(\Delta+k^2)w \ge 0$ in $D$ we have
\begin{equation}\label{eq:SL-analogue-first}
\int_{D} w(x)\,dx \ge \int w(x)\,d\mu(x). 
\end{equation}
\end{thm}

\begin{rem}
The assumption $w \in L^1(\mu)$ is in order to ensure that the right-hand side of \eqref{eq:SL-analogue-first} is well defined. 
\end{rem}

Finally we consider the relation of $k$-quadrature domains to the inverse problem of determining the shape of a penetrable obstacle from a single measurement, as discussed in \cite{SS21NonscatteringFreeBoundary}. See \cite{CCH16InverseScatteringTransmission,CK19scattering,Yaf10ScatteringAnalyticTheory} for more details about scattering problems. Let $D \subset \mR^n$ be a bounded open set, and let $h \in L^{\infty}(D)$ satisfy $|h| \ge c > 0$ a.e.\ near $\partial D$ (such a function $h$ is called a \emph{contrast} for $D$). The pair $(D,h)$ describes a penetrable obstacle $D$ with contrast $h$.

We now probe the penetrable obstacle
$(D,h)$ by some incident field $u_{0}$ at frequency $k > 0$. The incident field is a solution of 
\[
(\Delta + k^{2})u_{0} = 0 \quad \text{in}\;\;\mathbb{R}^{n}.
\]
Let $u_{\rm sc}$ be the corresponding scattered field. That is, the unique function $u_{\rm sc}$ so that the total field $u_{\rm tot} = u_{0} + u_{\rm sc}$ satisfies 
\begin{equation}
\begin{cases}
	(\Delta + k^{2} + h \chi_D)u_{\rm tot}=0 & \text{in}\;\;\mathbb{R}^{n},\\
	u_{\rm sc} \text{ satisfies the Sommerfeld radiation condition} & \text{at }|x| \rightarrow \infty.
\end{cases} \label{eq:scattering-problem}
\end{equation}
Here we recall that a solution $u$ of $(\Delta+k^{2})u=0$ in $\mathbb{R}^{n}\setminus\overline{B_{R}}$ (for some $R>0$) satisfies the Sommerfeld radiation condition if 
\[
\lim_{|x|\rightarrow\infty}|x|^{\frac{n-1}{2}}(\partial_{r}u-iku)=0, \quad \text{uniformly in all directions }\hat{x} = \frac{x}{|x|} \in \mathcal{S}^{n-1},
\]
where $\partial_{r}$ denotes the radial derivative. Solutions satisfying the Sommerfeld radiation condition are also called outgoing. The functions $u_0$, $u_{\rm sc}$ and $u_{\rm tot}$ are allowed to be complex. 

The single measurement inverse problem is to determine some properties of the obstacle $D$ from knowledge of the scattered wave $u_{\rm sc}(x)$ when $\abs{x}$ is large. If $D = \emptyset$, then $u_{\rm sc} \equiv 0$, and a related question is to ask whether some nontrivial domain $D$ admits some $h$ and $u_0$ so that $u_{\rm sc} = 0$ for large $x$. Such a penetrable obstacle $(D,h)$ would be invisible when probed by the incident wave $u_0$ and would look like empty space. Domains $D$ having this property for some $h$ and $u_0$ will be called non-scattering domains.

\begin{defn} \label{def_ns_domain}
We say that a bounded open set $D \subset \mR^n$ is a \emph{non-scattering domain} if there is some $h \in L^{\infty}(D)$ with $\abs{h} \geq c > 0$ a.e.\ near $\p D$ and some solution $u_0$ of $(\Delta + k^2) u_0 = 0$ in $\mR^n$ such that the corresponding scattered wave $u_{\rm sc}$ satisfies $u_{\rm sc}|_{\mR^n \setminus \ol{B}_R} = 0$ for some $R > 0$. 
\end{defn}

The following result states that $k$-quadrature domains are also non-scattering domains, at least if there is some incident wave $u_0$ that is positive on $\p D$. By the results in \cite{SS21NonscatteringFreeBoundary} such an incident wave $u_0$ exists at least when 
\begin{itemize}
\item 
$D$ is a $C^1$ domain (Lipschitz if $n=2,3$) so that $\mR^n \setminus \ol{D}$ is connected and $k^2$ is not a Dirichlet eigenvalue of $-\Delta$ in $D$; or 
\item
$D$ is contained in a ball of radius $< k^{-1} j_{\frac{n-2}{2},1}$ where $j_{\frac{n-2}{2},1}$ is the first positive zero of the Bessel function $J_{\frac{n-2}{2}}$. 
\labeltext{$j_{\alpha,1}$ the first positive zero of the Bessel function $J_{\alpha}$}{index:ZeroBessel}
\end{itemize}

By combining Theorem~\ref{thm_pde_quadrature_local} and \cite[Remark~2.4]{SS21NonscatteringFreeBoundary} we deduce the following corollary. 
\begin{cor}\label{cor:k-quad-nonscattering}
Let $D \subset \mR^n$ be a $k$-quadrature domain, and assume that there exists $u_0$ solving $(\Delta + k^2) u_0 = 0$ in $\mR^n$ with $u_0|_{\p D} > 0$. Then $D$ is a non-scattering domain {\rm (}for the incident wave $u_0$ and for some contrast $h${\rm )}.
\end{cor}

From Theorem~\ref{thm_quadrature_complex} and Corollary~\ref{cor:k-quad-nonscattering} we see that there exist non-scattering domains with inward cusps for any $k > 0$, extending the corresponding result for $k=0$ in \cite{SS21NonscatteringFreeBoundary}. In contrast, domains having suitable corner points cannot be non-scattering domains for any $k > 0$, i.e.\ ``corners always scatter''. This line of research was initiated in \cite{BPS14CornerScattering} and various further results were obtained in \cite{Bla18CornerScattering,BL17CornerScattering,BL21CornerScatteringSingleFarField,CakoniVogelius,CX21CornerScatteringEllipticOperator,EH15CornersEdgesScatter,PSV17CornerScattering}.

\addtocontents{toc}{\SkipTocEntry}
\subsection{Organization}

We prove Theorems~\ref{thm_pde_quadrature_local} and \ref{thm_quadrature_complex} in \S\ref{sec:PDE-characterization}, respectively \S\ref{sec:complex-analysis}. 
In \S \ref{sec:Partial-balayage2}, we introduce an obstacle problem, and define the partial balayage in terms of the maximizer of such an obstacle problem. We then study the structure of partial balayage in \S \ref{sec:Structure-of-partial} and \S \ref{sec:balayage-iterative}. Using these properties, we prove Theorem~\ref{thm:main2} in~\S\ref{sec:Construction-partial-balayage}. Finally, we provide some details about a real-valued fundamental solution relevant to our construction, some results related to maximum principles, the mean value theorem (MVT), and conformal images of $\mathbb{D}$ in Appendix~\ref{appen:GreenFunction}.

\addtocontents{toc}{\SkipTocEntry}
\subsection*{Acknowledgments}

This project was finalized while the authors stayed at Institute Mittag Leffler (Sweden), during the program Geometric aspects of nonlinear PDE.
The authors would like to express their gratitude to Lavi Karp and Aron Wennman for helpful comments. We would also like to give special thanks to Bj\"orn Gustafsson and the anonymous referee for a careful reading of the manuscript and several detailed suggestions that have improved the presentation. 
In particular, the comments from the anonymous referee have led to improvements in our results. Kow and Salo were partly supported by the Academy of Finland (Centre of Excellence in Inverse Modelling and Imaging, 312121) and by the European Research Council under Horizon 2020 (ERC CoG 770924). Larson was supported by Knut and Alice Wallenberg Foundation grant KAW~2021.0193. Shahgholian was supported by Swedish Research Council.

\section{\label{sec:PDE-characterization}PDE characterization of quadrature domains}

In this section we will prove Theorem \ref{thm_pde_quadrature_local} from the introduction. We begin with a global PDE characterization of $k$-quadrature domains.

\begin{prop} \label{prop_quadrature_pde_equivalence}
Let $k > 0$, and let $D \subset \mR^n$ be a bounded open set. Then $D$ is a $k$-quadrature domain corresponding to $\mu \in \mathscr{E}'(D)$ if and only if there is a distribution $u \in \mathscr{D}'(\mR^n)$ satisfying 
\begin{equation}\label{eq:PDE0}
\begin{cases}
\hfil (\Delta + k^2) u = \chi_D - \mu &\text{ in $\mR^n$}, \\
\hfil u = |\nabla u| = 0 &\text{ in $\mR^n \setminus D$}.
\end{cases}
\end{equation}
\end{prop}

Note that even though $u$ is only assumed to be in $\mathscr{D}'(\mR^n)$, the equation $(\Delta+k^2) u = \chi_D$ near $\p D$ and elliptic regularity imply that $u$ is $C^1$ near $\p D$ and hence the condition that $u = |\nabla u| = 0$ in $\mR^n \setminus D$ is meaningful.

\begin{example} (When $\mu$ is a Dirac mass.)
For the case when $D = B_{R}$ with $R>0$, and the measure is a constant multiple of the Dirac mass, we can find an explicit solution $u = u_{k,R}$ of~\eqref{eq:PDE0} in terms of Bessel functions. The general radially symmetric solution of $(\Delta+k^2)u=1$ in $\mathbb{R}^n\setminus\{0\}$ is \labeltext{$Y_{\alpha}$ Bessel function of second kind}{index:BesselSecond}
\[
 u_k(x)= \frac{1}{k^2}+ c_1 |x|^{1-\frac{n}{2}}J_{\frac{n}{2}-1}(k|x|) + c_2 |x|^{1-\frac{n}{2}}Y_{\frac{n}{2}-1}(k|x|).
\]
Given a radius $R>0$ there is a unique choice of the constants $c_1, c_2$ so that $u_{k,R}:=u_k \chi_{B_R} \in C^{1,1}(\mathbb{R}^n\setminus\{0\})$, namely
\begin{align*}
 c_1 = \frac{\pi R^{\frac{n}{2}} Y_{\frac{n}{2}}(k R)}{2k} \quad \mbox{and}\quad c_2 = -\frac{\pi R^{\frac{n}{2}}J_{\frac{n}{2}}(kR)}{2k}. 
\end{align*}
With these choices of coefficients $u_{k,R}$ satisfies 
\begin{align*}
(\Delta+k^2)u_{k,R} & = \chi_{B_{R}} - k^{-\frac{n}{2}}(2\pi R)^{\frac{n}{2}}J_{\frac{n}{2}}(kR)\,\delta \quad \text{in } \mathbb{R}^{n}, \\
u_{k,R}|_{\mathbb{R}^{n} \setminus \overline{B_{R}}} & = 0,
\end{align*}
which gives an example of Proposition~\ref{prop_quadrature_pde_equivalence} with $\mu = k^{-\frac{n}{2}}(2\pi R)^{\frac{n}{2}}J_{\frac{n}{2}}(kR)\, \delta$. 
\end{example}

\begin{example} (When $\mu \equiv 0$.)
Let $D$ be a bounded domain in $\mathbb{R}^{n}$ such that $\partial D$ is homeomorphic to a sphere. The well-known \href{https://www.scilag.net/problem/G-180522.1}{Pompeiu problem} \cite{Pom29PompeiuProblem, Wil76PompeiuProblem, Zalcman1992} asks whether the existence of a nonzero continuous function on $\mathbb{R}^{n}$ whose integral vanishes on all congruent copies of $D$ implies that $D$ is a ball. 
 
The problem can be reformulated in terms of free boundary problems, or in the context of this paper, in terms of null $k$-quadrature domains (i.e., $\mu \equiv 0$). Indeed the assumption in Pompeiu problem is equivalent to the existence of a function $v$ solving the free boundary problem
\begin{equation} \label{pompeiu_lambda}
\Delta v + \lambda v = \chi_{D} \ \hbox{in }\mathbb R^n, \qquad v=0 \ \hbox{outside } D,
\end{equation}
for some $\lambda >0$, see~\cite[Theorem~1]{Wil81PompeiuProblem} and \cite{Wil76PompeiuProblem}. 
If the bounded open set $D$ satisfies the assumptions in the Pompeiu problem and its boundary $\partial D$ is additionally Lipschitz regular, then $\partial D$ is analytic \cite{Wil81PompeiuProblem}. See also \cite{BK82Pompeiu,BST73Pompeiu,Den12SchifferConjecture} for some related results. The so-far unanswered question is: whether $D$ has to be a ball?

The fact that balls (with appropriate radii depending on $k > 0$) solve this problem is evident from the following simple procedure: take the function $u(x) = \abs{x}^{\frac{2-n}{2}} J_{\frac{n-2}{2}}(k \abs{x})$ that solves $\Delta u + k^2 u = 0$ in $\mR^n$, add a constant to $u$ so that one of the local minima (say $|x| = R$) of $u$ reaches the level zero, and then redefine the function to be zero outside $B_R$. After multiplying by a suitable constant, this function obviously solves the free boundary formulation of the Pompeiu problem.
 
An interesting observation is that the solution to the free boundary formulation of the Pompeiu problem thus constructed may change sign. The construction leads to a non-negative solution only if we choose $R$ to be the smallest radius for which $u$ takes a minimum. 

The above discussion also gives an indication of the failure of the application of the classical moving-plane technique for this problem.
\end{example}

We will require the following Runge approximation type result, see e.g.\ \cite[Chapter 11]{AdamsHedberg} for related results. We will follow the argument in  \cite[Lemma~5.1]{Sak84ObstacleGreenObstacles}. 
\begin{prop} \label{prop_runge_lone}
Let $k \geq 0$, and let $D \subset \mR^n$ be a bounded open set. Let $\Psi_k$ be any fundamental solution of $-(\Delta + k^{2})$ and let $\Omega \supset \overline{D}$ be any open set in $\mathbb{R}^{n}$. Then the linear span of
\[
F = \{ \p^{\alpha} \Psi_k(z-\,\cdot\,)|_D \,:\, z \in \Omega \setminus D, \ \abs{\alpha} \leq 1 \}
\]
is dense in 
\[
H_k L^1(D) = \{w \in L^1(D) \,:\, (\Delta+k^2) w = 0 \text{ in }D \}
\]
with respect to the $L^1(D)$ topology.
\end{prop}
\begin{rem}
If the domain $D$ has sufficiently regular boundary it suffices to take $\alpha=0$ in $F$. However, for domains like the slit disk one needs to consider also $\partial^\alpha \Psi_k(z-\,\cdot\,)|_D$ for all $|\alpha|=1$ and $z \in \Omega \setminus D$ below (note that these functions are all in $L^1(D)$). We shall later require a version of this result for sub-solutions (see Proposition~\ref{prop:prop_runge_lone-subsolution}).
\end{rem}

\begin{proof} [Proof of Proposition~{\rm \ref{prop_runge_lone}}]
By the Hahn--Banach theorem, it is enough to show that any bounded linear functional $\ell$ on $L^1(D)$ that satisfies $\ell|_F = 0$ also satisfies $\ell|_{H_k L^1(D)} = 0$. Since the dual of $L^1(D)$ is $L^{\infty}(D)$, there is a function $f \in L^{\infty}(D)$ with 
\[
\ell(w) = \int_D f w \,dx, \qquad w \in L^1(D).
\]
We extend $f$ by zero to $\mR^n$ and consider the function 
\[
u(z) = -(\Psi_k \ast f)(z) \quad \text{for all } z \in\Omega.
\]
By the assumption $\ell|_F = 0$, the function $u$ satisfies 
\[
\begin{cases}
\hfil (\Delta + k^2) u = f &\text{ in } \Omega, \\
\hfil u = |\nabla u| = 0 &\text{ in $\Omega \setminus D$}. 
\end{cases}
\]
We now consider the zero extension of $u$, still denoted by $u$, which satisfies 
\[
\begin{cases}
\hfil (\Delta + k^2) u = f &\text{ in } \mR^n, \\
\hfil u = |\nabla u| = 0 &\text{ in $\mR^n \setminus D$}. 
\end{cases}
\]
Note that since $f \in L^{\infty}$, we have $u \in C^{1,\alpha}$ for any $\alpha < 1$. In order to show that $\ell|_{H_k L^1(D)} = 0$, we take some $w \in H_k L^1(D)$ and compute 
\[
\ell(w) = \int_D f w \,dx = \int_D ((\Delta+k^2)u) w \,dx.
\]
We claim that one can integrate by parts and use the condition $(\Delta+k^2)w=0$ to conclude that 
\begin{equation} \label{runge_ibp}
    \int_D ((\Delta+k^2)u) w \,dx = 0.
\end{equation}
This implies that $\ell|_{H_k L^1(D)} = 0$ and proves the result. However, the proof of \eqref{runge_ibp} is somewhat delicate due to the failure of Calder\'on--Zygmund estimates when $p=\infty$. In the case $k=0$, \eqref{runge_ibp} follows from \cite[Lemma 5.1]{Sak84ObstacleGreenObstacles}. We will verify that the same argument works for $k > 0$.

First observe that $u$ solves  
\[
\Delta u = f - k^2 u \text{ in $\mR^n$}.
\]
Since $f$ and $u$ are $L^{\infty}$, it follows from \cite[Theorem 3.9]{GT01Elliptic} that $\nabla u$ satisfies 
\[
\abs{\nabla u(x) - \nabla u(y)} \leq C \abs{x-y} \log (1/\abs{x-y}), \qquad x, y \in \ol{D}, \ \abs{x-y} < e^{-2}.
\]
Using the condition $u = |\nabla u| = 0$ in $\mR^n \setminus D$, this implies that uniformly for $x \in D$ near $\p D$ one has 
\begin{align*}
    u(x) &= O(\delta(x)^2 \log(1/\delta(x))), \\
    \nabla u(x) &= O(\delta(x) \log(1/\delta(x))),
\end{align*}
where $\delta(x) = \mathrm{dist}(x, \p D)$.

As in \cite[Lemma 5.1]{Sak84ObstacleGreenObstacles} we introduce the sequence $(\omega_j)_{j=1}^{\infty}$ of Ahlfors--Bers mollifiers \cite{Ahl64KleinianGroups,Ber65Approximation} that satisfy $\omega_j \in C^{\infty}(\mR^n)$, $0 \leq \omega_j \leq 1$, $\omega_j = 0$ near $\p D$, $\omega_j = 1$ outside a neighborhood of $\p D$, $\omega_j(x) \to 1$ for $x \notin \p D$, and 
\[
\abs{\p^{\alpha} \omega_j(x)} \leq C_{\alpha} j^{-1} \delta(x)^{-\abs{\alpha}} (\log 1/\delta(x))^{-1} \text{ for $x \notin \p D$},
\]
see \cite[Lemma~4]{Hed73ApproximationFunctions}. We now begin the proof of \eqref{runge_ibp}. One has 
\begin{align*}
\int_D ((\Delta+k^2)u) w \,dx &= \lim_{j \rightarrow \infty} \int_D ((\Delta+k^2)u) \omega_j w \,dx \\
 &= \lim_{j \rightarrow \infty} \int_D \left[ (\Delta+k^2)(\omega_j u) - 2 \nabla \omega_j \cdot \nabla u - (\Delta \omega_j) u \right] w \,dx.
\end{align*}
Using the estimates for $u$ and $\omega_j$, the limits corresponding to the last two terms inside the brackets are zero. Moreover, since $w$ is smooth near $\supp(\omega_j)$, we have 
\[
\int_D (\Delta+k^2)(\omega_j u) w \,dx = \int_D \omega_j u (\Delta + k^2)w \,dx = 0.
\]
This concludes the proof of \eqref{runge_ibp}.
\end{proof}

\begin{proof}[Proof of Proposition~{\rm \ref{prop_quadrature_pde_equivalence}}]
Let $\Psi_k$ be any fundamental solution of $-(\Delta+k^2)$, i.e.\ $\Psi_k \in \mathscr{D}'(\mR^n)$ solves $-(\Delta+k^2) \Psi_k = \delta_0$ in $\mR^n$. In particular, $\Psi_k$ is smooth away from the origin.
If $D$ is a $k$-quadrature domain corresponding to $\mu$, then 
\begin{equation}\label{eq:QI0}
\int_D \p^{\alpha} \Psi_k(z-x) \,dx = \langle \mu, \p^{\alpha} \Psi_k(z-\,\cdot\,) \rangle
\end{equation}
whenever $z \in \mR^n \setminus D$ and $\abs{\alpha} \leq 1$. Let $u = -\Psi_k \ast (\chi_D - \mu)$, which is well defined since $\chi_D - \mu$ is a compactly supported distribution. We see that $(\Delta + k^2)u = \chi_D - \mu$  and $u = |\nabla u| = 0$ in $\mR^n \setminus D$ as required.

Conversely, suppose that $u \in \mathscr{D}'(\mR^n)$ is as in the statement. We easily obtain the quadrature identity for functions $w$ that solve $(\Delta +k^2)w=0$ near $\overline{D}$, since by taking a cutoff $\psi \in C^{\infty}_c(\mR^n)$ with $\psi=1$ near $\overline{D}$ we have 
\[
\int_D w \,dx - \langle w, \mu \rangle = \langle \chi_D-\mu, \psi w \rangle = \langle -(\Delta+k^2)u, \psi w \rangle = \langle u, -(\Delta+k^2)(\psi w) \rangle = 0,
\]
using that the derivatives of $\psi$ vanish near $\supp(u)$.

For general solutions $w \in L^1(D)$ we need another argument. Since $u$ is compactly supported, by the properties of convolution for distributions we have 
\[
u = \delta_0 * u = -(\Delta+k^2) \Psi_k * u= -\Psi_k * (\Delta+k^2)u = -\Psi_k * (\chi_D - \mu).
\]
Using that $u = |\nabla u| = 0$ in $\mR^n \setminus D$, we have 
\begin{equation}
\int_D \p^{\alpha} \Psi_k(z-x) \,dx = \langle \mu, \p^{\alpha} \Psi_k(z-\,\cdot\,) \rangle \label{eq:QI0-patch}
\end{equation}
for all $z \in \mR^n \setminus D$ and $\abs{\alpha} \leq 1$. Now let $w \in L^1(D)$ solve $(\Delta+k^2)w = 0$ in $D$, and use Proposition~\ref{prop_runge_lone} to find a sequence $w_j \in \mathrm{span} \{ \p^{\alpha} \Psi_k(z-\,\cdot\,)|_{D} \,:\, z \in \mR^n \setminus D, \,\abs{\alpha} \leq 1 \}$ with $w_j \to w \in L^1(D)$. In particular, for any $j \geq 1$ we have 
\begin{equation} \label{wj_identity}
\int_D w_j \,dx = \langle \mu, w_j \rangle.
\end{equation}
Since $\mu \in \mathscr{E}'(D)$, there is a compact set $K \subset D$ and an integer $m \geq 0$ such that 
\[
\abs{\langle \mu, \varphi \rangle} \leq C \norm{\varphi}_{C^m(K)}, \qquad \varphi \in C^{\infty}(D).
\]
By elliptic regularity and Sobolev embedding, any $v \in L^1(D)$ with $(\Delta+k^2)v \in H^{s-2}(D)$ satisfies $v \in C^m(K)$ when $s > m + n/2$. By the closed graph theorem this yields the estimate 
\[
\norm{v}_{C^m(K)} \leq C(\norm{v}_{L^1(D)} + \norm{(\Delta+k^2)v}_{H^{s-2}(D)}).
\]
Applying this estimate to $v = w_j - w$ gives 
\begin{equation}
\norm{w_j-w}_{C^m(K)} \leq C \norm{w_j-w}_{L^1(D)}. \label{eq:patch1-approximation}
\end{equation}
Thus we may take limits as $j \to \infty$ in \eqref{wj_identity} and obtain that 
\[
\int_D w \,dx = \langle \mu, w \rangle.
\]
This shows that $D$ is a $k$-quadrature domain for $\mu$.
\end{proof}

\begin{proof}[Proof of Theorem~{\rm \ref{thm_pde_quadrature_local}}]
If $D$ is a $k$-quadrature domain corresponding to $\mu$, then taking a neighborhood $U$ of $\p D$ that is disjoint from $\supp(\mu)$ and restricting the distribution from Proposition~\ref{prop_quadrature_pde_equivalence} to $U$ gives the required $u \in \mathscr{D}'(U)$ satisfying \eqref{eq:local-PDE-characterization}.

Conversely, assume that $u \in \mathscr{D}'(U)$ satisfies \eqref{eq:local-PDE-characterization}. Let $\psi \in C^{\infty}_c(U)$ satisfy $0 \leq \psi \leq 1$ and $\psi = 1$ near $\p D$, and define $\tilde{u} = \psi u \in \mathscr{D}'(\mR^n)$. Also define 
\begin{equation} \label{mu_def_quadrature}
\tilde{\mu} := \chi_D - (\Delta + k^2)\tilde{u}.
\end{equation}
Then $\tilde{u}$ satisfies 
\[
\begin{cases}
\hfil (\Delta + k^2)\tilde{u} = \chi_D - \tilde{\mu} &\text{ in }\mR^n, \\
\hfil \tilde{u} = |\nabla \tilde{u}| = 0 & \text{ in $\mR^n \setminus D$}.
\end{cases}
\]
Moreover, $\tilde{\mu} \in \mathscr{D}'(\mR^n)$ satisfies $\supp(\tilde{\mu}) \subset D$ by the assumption on $u$. Then $D$ is a $k$-quadrature domain by Proposition~\ref{prop_quadrature_pde_equivalence}.  By elliptic regularity $u$ is smooth in $U \cap D$,  and thus $\tilde{\mu}$ coincides with a smooth function in $D$. Since one also has $\supp(\tilde{\mu}) \subset D$‚ it follows that $\tilde{\mu}$ has a smooth density with respect to Lebesgue measure.
\end{proof}

\section{\label{sec:complex-analysis}Quadrature domains with cusps}

This section contains the proof of Theorem~\ref{thm_quadrature_complex}. The proof will employ the following simple fact regarding the vanishing order of solutions. In this section all functions are allowed to be complex valued.

\begin{lem} \label{lemma_vo}
Let $v$ be a $C^{\infty}$ function near some $x_0 \in \mathbb{R}^n$ satisfying 
\[
\begin{cases}
\hfil \Delta v = O(\abs{x-x_0}^m) \text{ near }x_0, \\
\hfil v|_S = \partial_{\nu} v|_S = 0,
\end{cases}
\]
where $m \geq 0$ is an integer, $S$ is a smooth hypersurface through $x_0$, and $\partial_\nu$ denotes the derivative in the normal direction to $S$. Then one has $v = O(\abs{x-x_0}^{m+2})$, and more precisely 
\[
v = \sum_{\abs{\alpha}=m+2} v_{\alpha}(x)(x-x_0)^{\alpha}
\]
where $v_{\alpha}$ are smooth near $x_0$.
\end{lem}

\begin{proof}
After a rigid motion, we may assume that $x_0 = 0$ and the normal of $S$ satisfies $\nu(0) = e_n$. We use the Taylor formula and write $v$ as 
\[
v = \sum_{j=0}^{m+1} P_j + R, \qquad R = \sum_{\abs{\alpha}=m+2} v_{\alpha}(x)(x-x_0)^{\alpha} ,
\]
where each $P_j$ is a homogeneous polynomial of degree $j$ and each $v_{\alpha}$ is smooth. Using the assumption $v|_S = \partial_{\nu} v|_S = 0$ we have $P_0 = P_1 = 0$. Moreover, the assumption for $\Delta v$ implies that 
\[
\sum_{j=2}^{m+1} \Delta P_j = O(\abs{x}^m).
\]
Since the left hand side is a polynomial of degree $m-1$, it follows that we must have $\Delta P_j = 0$ for $2 \leq j \leq m+1$.

Suppose that $\gamma$ is a smooth curve on $S$ with $\gamma(0) = 0$ and $\dot{\gamma}(0) = \omega$ where $\omega \perp e_n$ and $\abs{\omega} = 1$. Since $v|_S = \p_{\nu} v|_S = 0$, we have 
\begin{align}
0 &= v(\gamma(t)) = \sum_{j=2}^{m+1} \abs{\gamma(t)}^j P_j(\gamma(t)/\abs{\gamma(t)}) + O(\abs{\gamma(t)}^{m+2}), \label{vs1} \\
0 &= \p_{\nu} v(\gamma(t)) = \sum_{j=2}^{m+1} \abs{\gamma(t)}^{j-1} \nu(\gamma(t)) \cdot \nabla P_j(\gamma(t)/\abs{\gamma(t)}) + O(\abs{\gamma(t)}^{m+1}). \label{vs2}
\end{align}
Since $\gamma(t) = t \omega + O(t^2)$, we have $\gamma(t)/\abs{\gamma(t)} \to \omega$ as $t \to 0$. If one would have $P_2(\omega) \neq 0$, then multiplying \eqref{vs1} by $t^{-2}$ would lead to a contradiction as $t \to 0$. Similarly $\p_n P_2(\omega) \neq 0$ would lead to a contradiction with \eqref{vs2}. Thus $P_2(\omega) = \p_n P_2(\omega) = 0$. Varying $\omega$ implies that
\[
P_2|_{x_n=0} = \p_n P_2|_{x_n=0} = 0.
\]
But since $\Delta P_2 = 0$, unique continuation implies that $P_2 \equiv 0$. Iterating this argument shows that $P_2 \equiv \ldots \equiv P_{m+1} = 0$ as required.\footnote{ Lemma~\ref{lemma_vo} can also be proved by a simple blow-up argument. Starting with a quadratic blow-up one obtains $P_2$ in the limit and $S$ becomes a hyperplane $\{x_n=0\}$, along with the zero Cauchy-data for $P_2$. This implies $P_2\equiv 0$. Repeating this argument, by a cubic scaling we obtain $P_3 \equiv 0$. Iterating this argument we have $P_j \equiv 0$ for all $j \leq m+1$.}
\end{proof}

We are now ready to prove Theorem~\ref{thm_quadrature_complex}. As we shall illustrate in Examples \ref{exa:cardioid}--\ref{exa:curve-cusps} the domain $D$ may have inward cusps and the map $\varphi$ is not necessarily injective on $\p \dd$, which introduces some technicalities in the argument.

\begin{proof} [Proof of Theorem~{\rm \ref{thm_quadrature_complex}}]
Let $D = \varphi(\dd)$ where $\varphi$ is analytic near $\ol{\dd}$ and injective in $\dd$. Note that $D$ is an open set by the open mapping theorem for analytic functions \cite[Theorem~10.32]{Rud87RealComplexAnalysis}. We claim:
\begin{equation} \label{varphi_boundary}
 \text{if $z_j \in \dd$ and $d(\varphi(z_j), \p D) \to 0$, then $d(z_j, \p \dd) \to 0$.}
\end{equation}
To see \eqref{varphi_boundary}, we argue by contradiction and assume that $d(\varphi(z_j), \p D) \to 0$ but there is $\eps > 0$ and a subsequence $(z_{j_k})$ with $d(z_{j_k}, \p \dd) \geq \eps$. After passing to another subsequence also denoted by $(z_{j_k})$, we have $z_{j_k} \to z \in \dd$ with $d(z, \p \dd) \geq \eps$. However, since $d(\varphi(z_j), \p D) \to 0$ we must have $d(\varphi(z), \p D) = 0$. This contradicts the fact that $\varphi(\dd) = D$, proving \eqref{varphi_boundary}.

Next we show that 
\begin{equation} \label{varphi_boundary1}
 \varphi(\p \dd) = \p D.
\end{equation}
We begin the proof of \eqref{varphi_boundary1} by taking $z \in \p \dd$ and showing that $\varphi(z) \in \p D$. By continuity $\varphi(z) \in \ol{D}$. If one had $\varphi(z) \in D$, then since $\varphi$ is bijective $\dd \to D$ there would be some $z' \in \dd$ with $\varphi(z') = \varphi(z)$. For any $\eps <|z-z'|/2$ we consider the open sets $\varphi(B_\eps(z'))$ and $\varphi(B_{\eps}(z)\cap \dd)$. The point $\varphi(z)$ is contained in the interior of the first set and in the closure of the second, in particular the two sets are not disjoint. This contradicts the assumption that $\varphi$ is injective, and thus proves that $\varphi(\partial \dd) \subset \partial D$. For the converse inclusion, if $x \in \p D$ and $x_j \to x$ where $x_j \in D$, then $x_j = \varphi(z_j)$ for some $z_j \in \dd$. After passing to a subsequence we may assume that $z_{j_k} \to z \in \ol{\dd}$, and by \eqref{varphi_boundary} one must have $z \in \p \dd$. Thus $x = \lim \varphi(z_j) = \varphi(z)$, proving \eqref{varphi_boundary1}.

By Theorem \ref{thm_pde_quadrature_local}, the result will follow if we can find a distribution $u$ near $\p D$ solving 
\begin{equation} \label{u_complex_required}
\begin{cases}
(\Delta + k^2) u = \chi_D & \text{near}\;\;\p D, \\
u = |\nabla u| = 0 & \text{outside}\;\;D.
\end{cases}
\end{equation}
By the chain rule
(see Lemma \ref{lemma_varphi}),
the equation $(\Delta+k^2) u = 1$ in some set $\varphi(U_1)$, with $U_1 \subset \dd$ open, is equivalent to the equation 
\[
(\Delta + k^2 \abs{\varphi'}^2)(u \circ \varphi) = \abs{\varphi'}^2 \text{ in $U_1$.}
\]
Since $|\varphi'|^{2}=\p \varphi \ol{\p \varphi}$ is real-analytic near $\p \dd$, the Cauchy--Kowalevski theorem implies that there exists a neighborhood $U$ of $\partial \dd$ and a function $\hat{u}$ which is real-analytic in $U$ such that 
\begin{equation}
\begin{cases}
	(\Delta+k^{2}|\varphi'|^{2})\hat{u}=|\varphi'|^{2} & \text{in}\;\;U,\\
	\hat{u}=\partial_{\nu}\hat{u}=0 & \text{on}\;\;\partial \dd.
\end{cases}\label{eq:aux-tildeu0}
\end{equation}
By \eqref{varphi_boundary1} and the open mapping theorem we know that $V = \varphi(U)$ is an open neighborhood of $\partial D$. We define 
\[
u(x):=\begin{cases}
	\hat{u}(\varphi^{-1}(x)) & \text{for all}\;\;x\in V \cap D,\\
	0 & \text{for all}\;\;x \in V \setminus D.
\end{cases}
\]

The function $u$ is defined piecewise and it satisfies \eqref{u_complex_required} away from $\p D$. If we can prove that $u \in C^{1,1}(V)$, then $u$ will satisfy \eqref{u_complex_required} also near $\p D$ and the proof of the theorem will be concluded. Note that by the inverse function theorem, $u$ is smooth in $V \cap D$. We would like to show that $u$ is continuous up to $\p D$. If $x \in \p D$ and $x_j \in V \cap D$ satisfy $x_j \to x$, then $x_j = \varphi(z_j)$ for some $z_j \in \dd$. Then $d(\varphi(z_j), \p D) \to 0$, and \eqref{varphi_boundary} ensures that $d(z_j, \p \dd) \to 0$. It follows that 
\[
u(x_j) = \hat{u}(z_j) \to 0
\]
since $\hat{u}$ is Lipschitz near $\p \dd$ and $\hat{u}|_{\p \dd} = 0$. This shows that $u \in C^0(V)$.

Next we show that $u$ is $C^1$ up to $\p D$. Let $x \in \p D$ and $x_j \in D$ with $x_j \to x$. It is enough to show that for any $\eps > 0$ there is $j_0$ such that $\abs{\nabla u(x_j)} \leq \eps$ for $j \geq j_0$. Now $x_j = \varphi(z_j)$ where $z_j \in \dd$, and by the chain rule one has 
\[\p u(\varphi(z)) = \frac{\p \hat{u}(z)}{\varphi'(z)}, \qquad \dbar u(\varphi(z)) = \frac{\dbar \hat{u}(z)}{\ol{\varphi'(z)}}.
\]
Thus 
\begin{equation} \label{u_chain_rule}
\abs{\nabla u(\varphi(z))} = \frac{\abs{\nabla \hat{u}(z)}}{\abs{\varphi'(z)}}.
\end{equation}
Using \eqref{varphi_boundary} we know that $d(z_j, \p \dd) \to 0$, and thus $\nabla \hat{u}(z_j) \to 0$ since $\nabla \hat{u}|_{\p \dd} = 0$. However, $\varphi'(z_j)$ may also converge to zero and this requires some care. We start by observing that there are only finitely many points $z_0 \in \p \dd$ with $\varphi'(z_0) = 0$, and near any such $z_0$ one can write 
\[
\varphi(z) = \varphi(z_0) + (z-z_0)^2 g(z) ,
\] 
for some analytic function $g$. Since $\varphi\colon \mathbb{D}\to D$ is bijective it follows that $\varphi''(z_0)\neq 0$ and hence $g(z_0)\neq 0$ (see Remark~\ref{rem:cusp_discussion}). Thus $\abs{\varphi'(z)}^2 = O(\abs{z-z_0}^{2})$. Using Lemma~\ref{lemma_vo} we know that 
\begin{equation} \label{palphau_estimate}
\p^{\alpha} \hat{u}(z) = O(|z - z_{0}|^{4-\abs{\alpha}}) \text{ for $\abs{\alpha} \leq 4$ and for all } z \text{ near } z_{0}.
\end{equation}
By \eqref{u_chain_rule}, for $z \in \dd$ near $z_0$ we have 
\[
|\nabla u(\varphi(z))| \leq C\frac{|z-z_0|^{3}}{|z-z_{0}|} \leq C |z-z_0|^2.
\]
Thus there is $\delta > 0$ such that 
\[
|\nabla u(\varphi(z))| \leq \eps \text{ when } z \in W := \bigcup_{z_0 \in \p \dd, \varphi'(z_0) = 0} (B(z_0,\delta) \cap \dd).
\]
We have $\p \dd \subset W \cup W'$ where $W'$ is some open set with $\abs{\varphi'(z)} \geq c > 0$ for $z \in W'$. We already know that $\abs{\nabla u(\varphi(z_j))} \leq \eps$ when $z_j \in W$, and for $z_j \in W'$ the expression \eqref{u_chain_rule} gives that 
\[
\abs{\nabla u(\varphi(z_j))} \leq \frac{1}{c} |\nabla\hat{u}(z_j)|
\]
which becomes $\leq \eps$ when $j \geq j_0$ for some sufficiently large $j_0$ by \eqref{varphi_boundary}. This concludes the proof that $u \in C^1(V)$. 

Finally, we use the chain rule again and observe that for $z \in \dd$ one has 
\[
|\nabla^{2}u(\varphi(z))| \le C\bigg(\frac{|\nabla^{2}\hat{u}(z)|}{|\varphi'(z)|^{2}}+\frac{|\nabla\hat{u}(z)||\varphi''(z)|}{|\varphi'(z)|^{3}}\bigg).
\] 
As before, the worst case is when $z$ is close to some $z_0 \in \p \dd$ with $\varphi'(z_0) = 0$. By \eqref{palphau_estimate}, for $z$ near $z_0$ one has 
\[
|\nabla^{2}u(\varphi(z))| \le C\bigg(\frac{|\nabla^{2}\hat{u}(z)|}{|z-z_{0}|^{2}}+\frac{|\nabla\hat{u}(z)|}{|z-z_{0}|^{3}}\bigg) \leq C.
\]
It follows that $\nabla u$ is Lipschitz continuous in $V$. In fact it is Lipschitz in $V \cap D$ and $V \setminus D$, and if $x \in V \cap D$ and $y \in V \setminus D$ we let $y_1$ be a closest point to $x$ in $\mR^n \setminus D$ (so that $y_1 \in \p D$) and observe that 
\[
\abs{\nabla u(x) - \nabla u(y)} = \abs{\nabla u(x) - \nabla u(y_1)} \leq C \abs{x-y_1} \leq C \abs{x-y}.
\]
This proves that $u \in C^{1,1}(V)$, and therefore concludes the proof of Theorem~\ref{thm_quadrature_complex}.
\end{proof}

\begin{rem} \label{rem:cusp_discussion}

Let $D = \varphi(\dd)$ where $\varphi$ is an analytic function near $\overline{\dd}$ which is injective in $\dd$. In this remark we clarify what $\p D$ looks like. Recall from \eqref{varphi_boundary1} that $\varphi(\p \dd) = \p D$. We may divide the boundary points in three categories.
\begin{enumerate}
\renewcommand{\labelenumi}{\theenumi}
\renewcommand{\theenumi}{(\roman{enumi})}
\item \label{itm:smooth-points}
(Smooth points) If $x_0 \in \p D$ is of the form $x_0 = \varphi(z_0)$ for a unique $z_0 \in \p \dd$ and $\varphi'(z_0) \neq 0$, then by the inverse function theorem $D$ near $x_0$ is given by the region above the graph of a real-analytic function.
 
\item \label{itm_cusppoints_varphi}
(Inward cusp points) If $x_0 \in \p D$ is of the form $x_0 = \varphi(z_0)$ for some $z_0 \in \p \dd$ with $\varphi'(z_0) = 0$, then $\varphi''(z_0) \neq 0$ since if $\varphi$ vanished to higher order the bijectivity in $\dd$ would fail in the same way that it does for $z \mapsto z^m$, $m > 2$, around $z= 0$ (the image of an arbitrary half-plane covers $\mathbb{C} \setminus \{0\}$ more than once). Thus $\varphi$ behaves near $z_0$ like $z \mapsto z^2$ which produces an inward cusp.
 
\item
(Double points) If $x_0 \in \p D$ satisfies $x_0 = \varphi(z_1) = \varphi(z_2)$ for two distinct $z_1, z_2 \in \p \dd$, then by the bijectivity $\varphi'(z_1)\neq 0$ and $\varphi'(z_2)\neq 0$ and there exists an $r>0$ small enough so that $\partial D \cap B_r(x_0)$ is the union of two analytic arcs whose intersection is $\{x_0\}$ where the arcs touch (by injectivity they do not cross).
\end{enumerate}
Moreover, there are only finitely many points which fail to be in category \ref{itm:smooth-points}.
\end{rem}

This classification of the points on the boundary of $D$ is rather classical. Remark~\ref{rem:cusp_discussion} is also related to Sakai's regularity theorem, see \cite[Theorem~5.2]{Sak91RegularityTheorem} as well as \cite[Section~3.2]{LM16QuadratureDomain}. 
However, since we were unable to find a direct proof of the fact that the set of non-smooth points is finite under our assumptions we provide a sketch of the argument in Appendix~\ref{appen: Riemann mapping regularity}. 

\begin{example} [Figure~\ref{fig:cardioid}] \label{exa:cardioid}
Let $\varphi(z) = z + \frac{1}{2} z^2$ and $D = \varphi(\dd)$. Then $D$ is a cardioid whose boundary is smooth except at the point $\varphi(-1) = -1/2$ where it has an inward cusp. It is clear that $\varphi$ satsifies the conditions of Theorem~\ref{thm_quadrature_complex}. Similarly, if $\varphi(z) = z + \frac{1}{m} z^{m}$ for integer $m \ge 2$ then $D$ has $m-1$ inward cusps. 
\end{example}

\begin{example} [Figure~\ref{fig:double-and-inward-cusps}] \label{exa:double-and-inward-cusps}
Let $\varphi(z) = z - \frac{2\sqrt{2}}{3}z^{2} + \frac{1}{3}z^{3}$ and $D = \varphi(\dd)$ (see e.g.~\cite[equation~(1.9)]{LM16QuadratureDomain}). Then the corresponding domain $D$ is not a Jordan domain and furthermore its boundary has inward cusps. By Theorem~\ref{thm_quadrature_complex}, the domain $D$ is a $k$-quadrature domain. 
\end{example}

\begin{example} [Figure~\ref{fig:curve-cusps}] \label{exa:curve-cusps}
Let $\varphi(z) = (z-1)^{2} - (1 - \frac{i}{2} ) (z-1)^{3}$ and $D = \varphi(\dd)$. 
The domain $D$ looks similar to a cardioid, but with an inward cusp which is curved in such a manner that the $\partial D$ cannot locally be represented as the graph of a function. It is also a $k$-quadrature domain by Theorem~\ref{thm_quadrature_complex}.
\end{example}

\begin{figure}[H]
\centering
\subfloat[$m=2$]{
		\begin{centering}
		\includegraphics[scale=0.15]{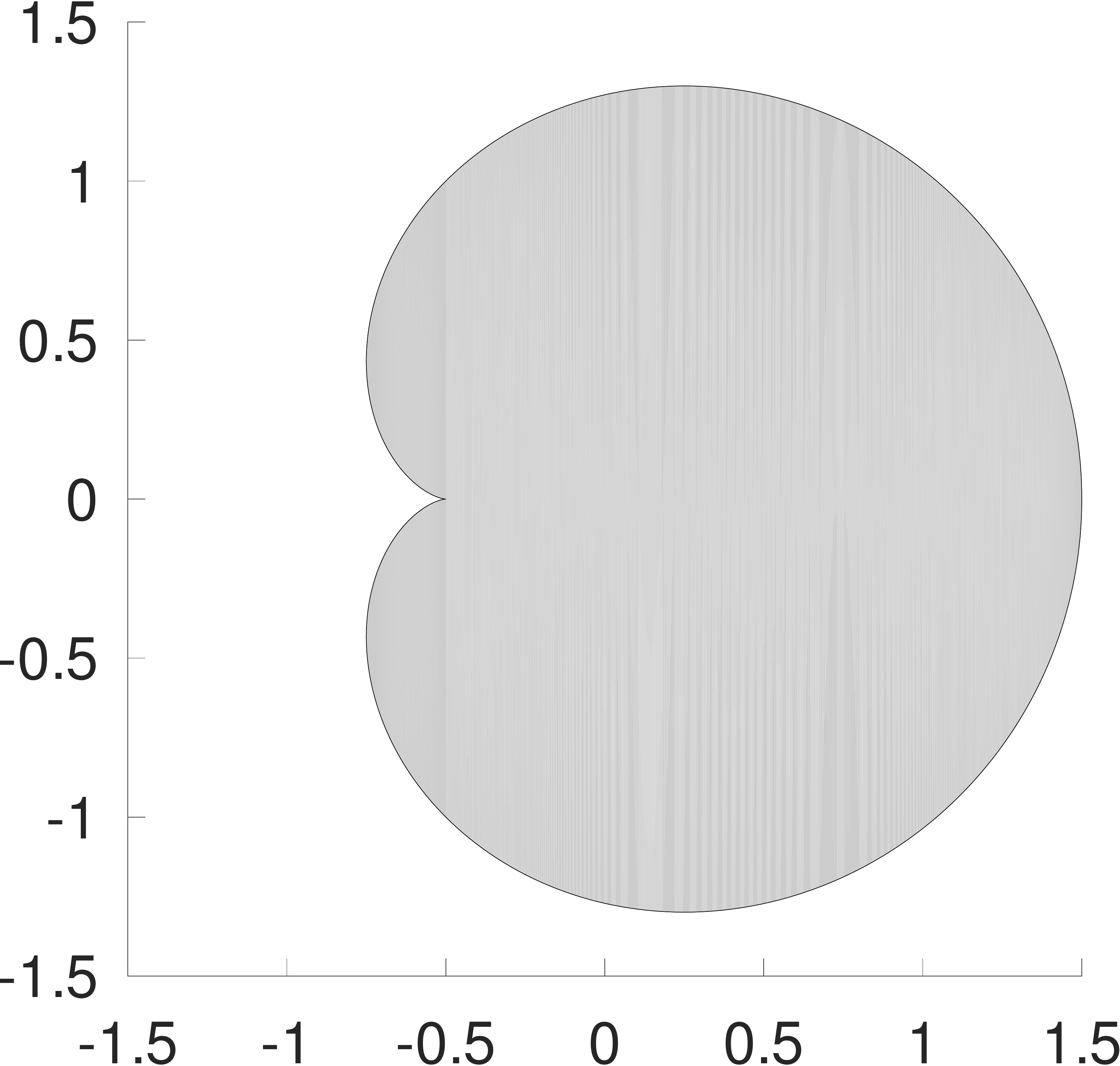}
		\par\end{centering}
}
\subfloat[$m=4$]{
		\begin{centering}
		\includegraphics[scale=0.15]{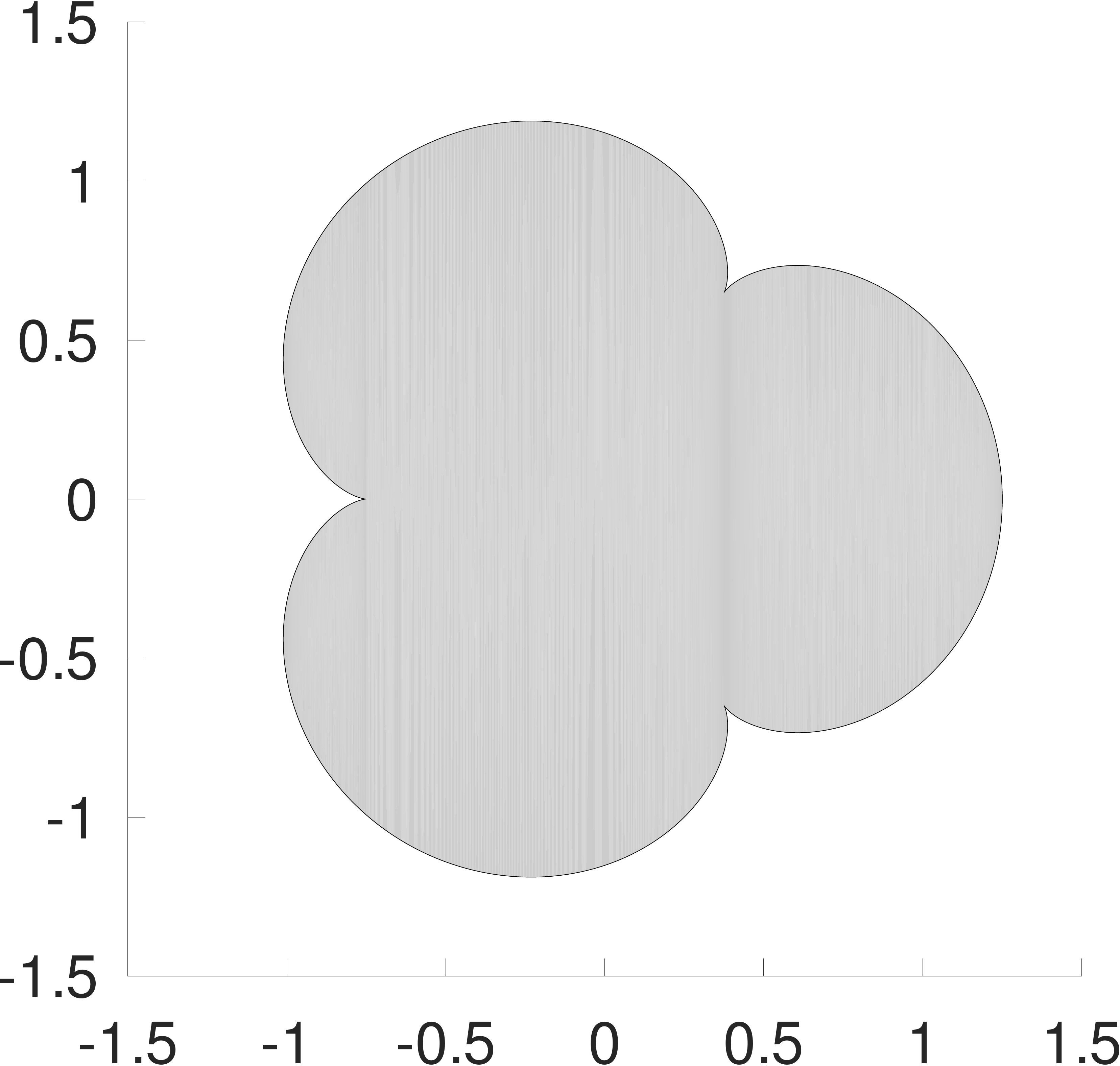}
		\par\end{centering}
}
\caption{\label{fig:cardioid} Plot of Example~\ref{exa:cardioid} (GNU Octave)}
\end{figure}		

\begin{figure}[H]
\subfloat[Global view]{
 \begin{centering}
		\includegraphics[scale=0.15]{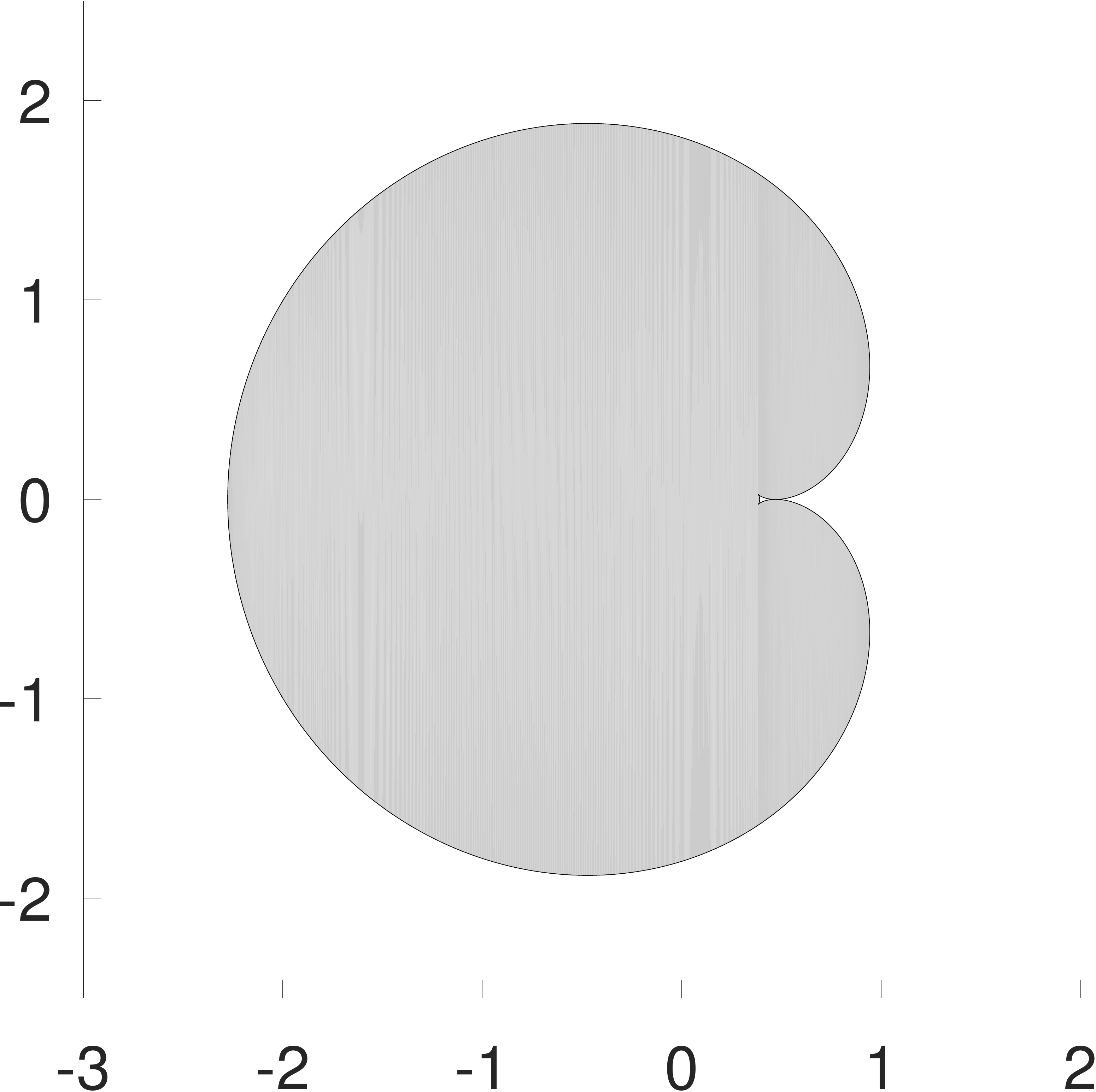}
		\par\end{centering}
}
\subfloat[Near points of interest]{
 \begin{centering}
		\includegraphics[scale=0.15]{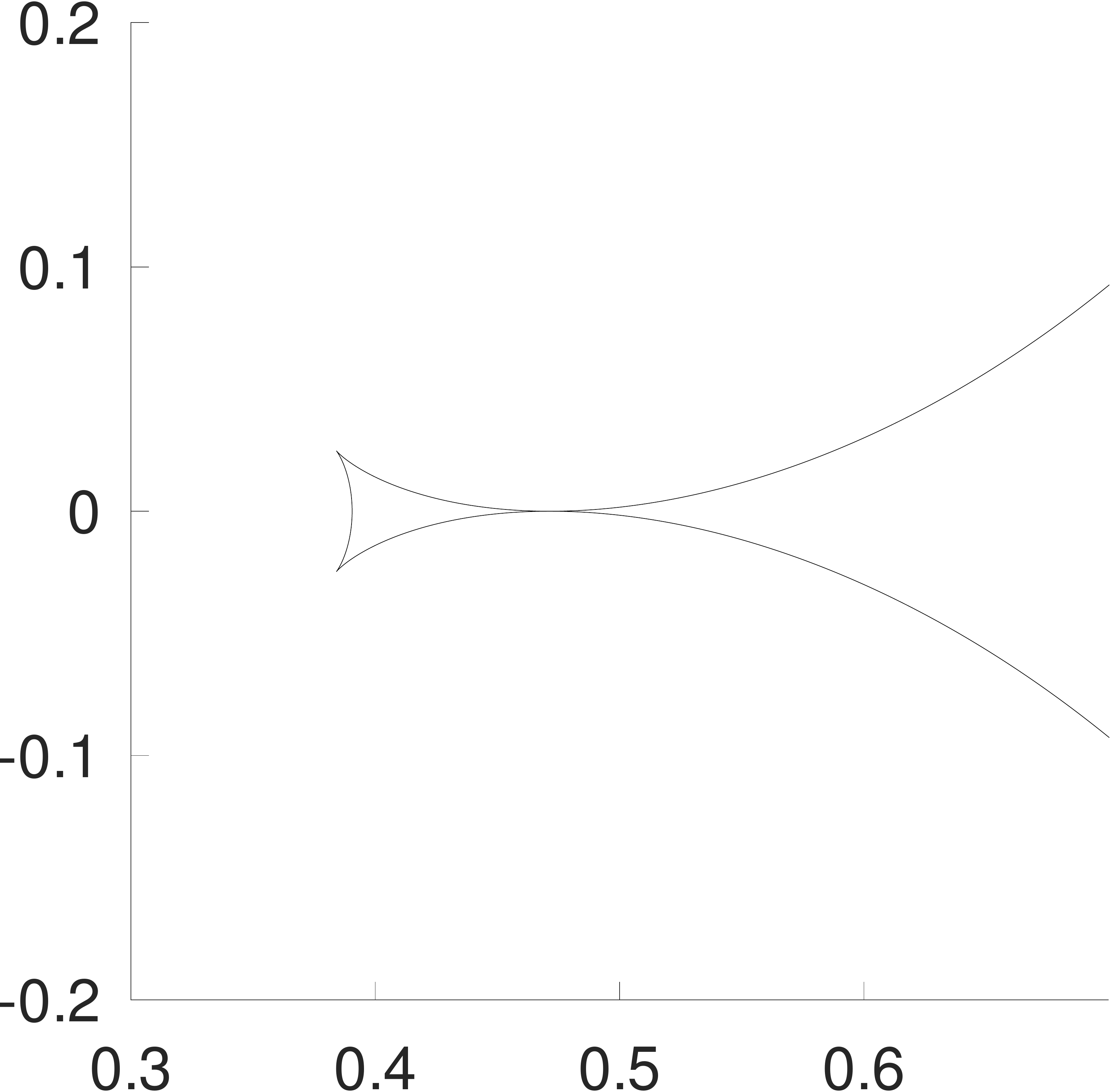}
		\par\end{centering}
}
\caption{\label{fig:double-and-inward-cusps} Plot of Example~\ref{exa:double-and-inward-cusps} (GNU Octave)}
\end{figure}

\begin{figure}[H]
\subfloat[Global view]{
 \begin{centering}
		\includegraphics[scale=0.15]{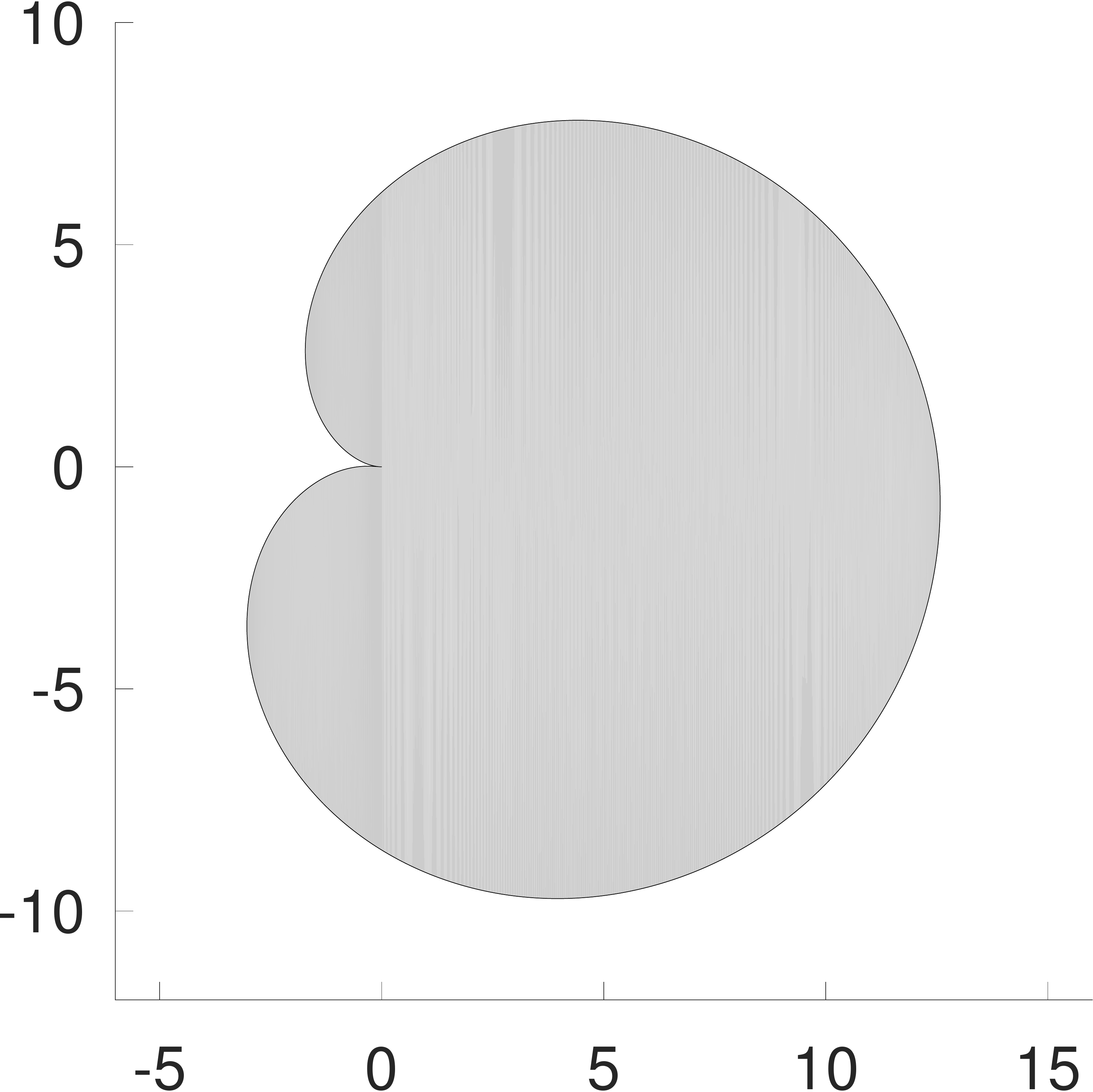}
		\par\end{centering}
}
\subfloat[Near the cusp]{
 \begin{centering}
		\includegraphics[scale=0.15]{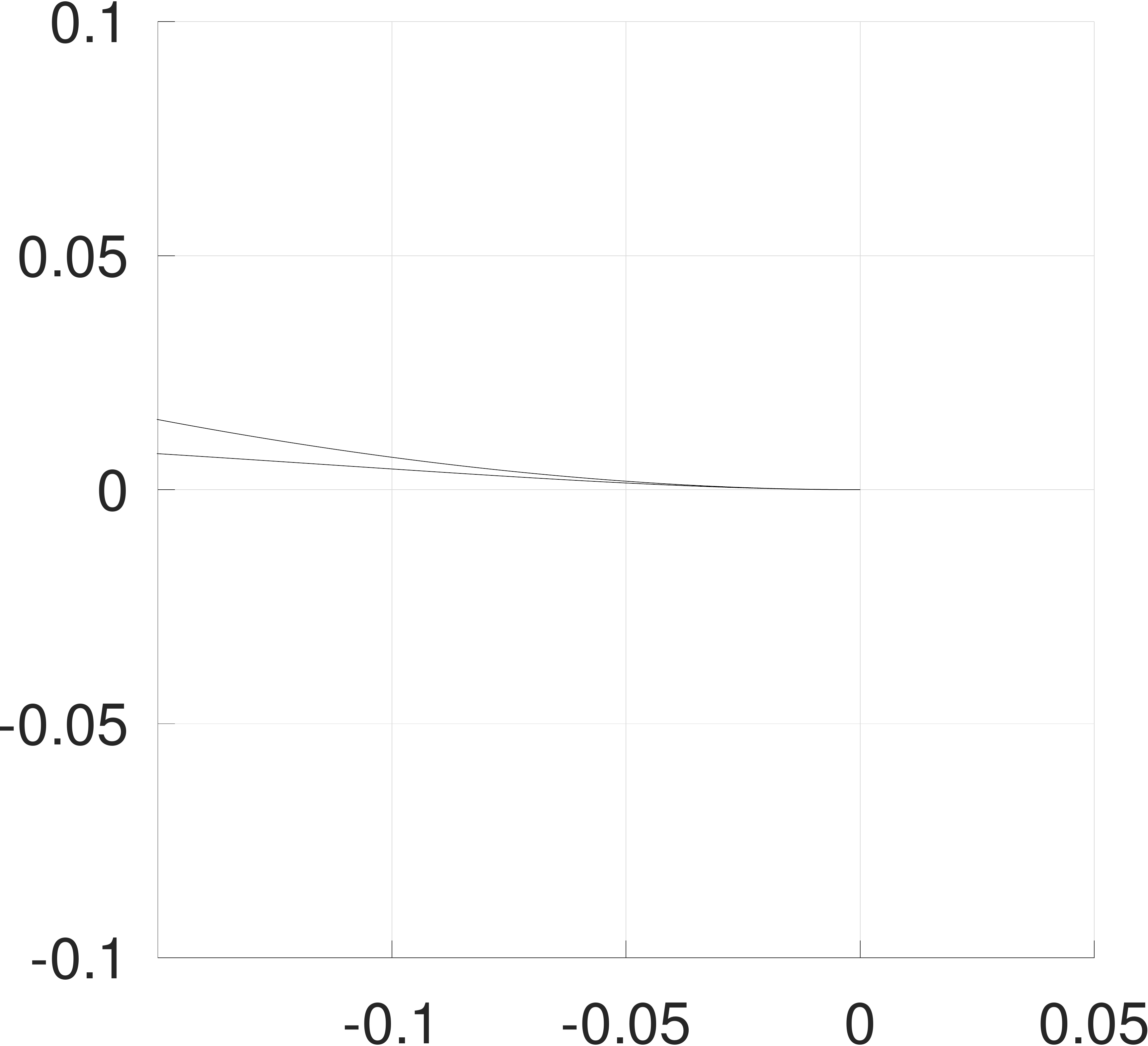}
		\par\end{centering}
}
\caption{\label{fig:curve-cusps} Plot of Example~\ref{exa:curve-cusps} (GNU Octave)}
\end{figure}

\section{\label{sec:Partial-balayage2}Partial balayage via an obstacle problem }

In this section we define the partial balayage measure $\Bal(\mu)$ with respect to $\Delta + k^2$ for certain sufficiently concentrated measures $\mu$ when $k > 0$ is small. For simplicity we will assume that $\mu$ is concentrated near the origin, but by translation invariance any other point would do. There are several ways of defining partial balayage, and we will proceed via an obstacle problem (see e.g.\ \cite[Definition 3.2]{Gus04LecturesBalayage} for $k=0$). 

Let $R(n,k)= \frac{1}{2}j_{\frac{n-2}{2},1} k^{-1}$ and let $\tilde{\Phi}_{k} = \tilde{\Phi}_{k,R(n,k)}$ be the fundamental solution given in Proposition~\ref{prop:GreenFunction}. Here and in what follows we write $U_{k}^{\mu} := \tilde{\Phi}_{k} \ast \mu$ for any Radon measure $\mu$. In the special case where $\mu = \chi_{\Omega}\mathsf{m}$ for some open set $\Omega$ we simply write $U_{k}^{\Omega} := \tilde{\Phi}_{k} \ast \chi_{\Omega}$. 
\labeltext{$R(n,k) = c_{n}^{\rm ref} k^{-1}$ maximal length scale}{index:ConstantRnk} 
\labeltext{$\tilde{\Phi}_{k} = \tilde{\Phi}_{k,R(n,k)}$ a particular fundamental solution of the Helmholtz operator $-(\Delta+k^2)$}{index:RealFundamentalSolution}
\labeltext{$U_{k}^{\mu} = \tilde{\Phi}_{k} \ast \mu$ potential of a measure $\mu$}{index:potential-Uk}

We now restrict ourselves to measures $\mu$ having a bounded density with respect to Lebesgue measure. Slightly abusing notation we write $\mu\in L^{\infty}(\Omega)$  to mean that $\mu$ has the form
\[
\mu = f \mathsf{m},
\]
where $f \in L^{\infty}(\mR^n)$ satisfies $f = 0$ outside $\Omega$. Under this assumption, by elliptic regularity (see e.g.\ \cite[Theorem~9.11]{GT01Elliptic}) 
\[
U_{k}^{\mu}\in\bigcap_{1<p<\infty}W_{{\rm loc}}^{2,p}(\mathbb{R}^{n})\subset\bigcap_{0<\alpha<1} C^{1,\alpha}(\overline{B_{R(n,k)}}).
\]
For $k>0$ and $\mu \in L^\infty(B_{R(n, k)})$ define \labeltext{$\mathscr{F}_{k}(\mu)$ the class of admissible functions in an obstacle problem}{index:FamilyFunctionsMu} 
\begin{equation}\label{eq: scrF}
\mathscr{F}_{k}(\mu)=\begin{Bmatrix}\begin{array}{l|l}
v\in H^{1}(B_{R(n,k)}) & \begin{array}{l}
(\Delta+k^{2})v \ge -1 \text{ in }B_{R(n,k)}\\
v\le U_{k}^{\mu}\text{ in }B_{R(n,k)}\\
v=U_{k}^{\mu}\text{ on }\partial B_{R(n,k)}
\end{array}\end{array}\end{Bmatrix}.
\end{equation}

\begin{lem}
\label{lem:non-empty-patch} 
Let $0 < \gamma < R(n,k)$ and $\mu\in L^{\infty}(B_\gamma)$.
Assume that there exists $r>0$ such that
\begin{equation}
r< R(n,k) - \gamma \quad \text{and} \quad c_{n,k,r}^{\rm MVT} \ge \mu_{+}(\mathbb{R}^{n}), \label{eq:non-empty-assumption-patch}
\end{equation}
where $c_{n,k,r}^{{\rm MVT}}$ is the constant appearing in the mean value theorem for the Helmholtz equation {\rm (}see~\eqref{eq:MVT-Helmholtz}{\rm )}. Then $\mathscr{F}_{k}(\mu)$ contains an element $\tilde{u}_{k}$ which equals $U_{k}^{\mu}$ in $B_{R(n,k)} \setminus B_{\gamma + r}$ {\rm (}note that $\gamma + r < R(n,k)${\rm )}. 
\end{lem}

\begin{proof}
Let 
\begin{equation}
\tilde{u}_{k} := U_{k}^{\mu_{+}}*h_r-U_{k}^{\mu_{-}}\quad\text{where}\quad h_r :=\frac{1}{c_{n,k,r}^{{\rm MVT}}}\chi_{B_{r}}.\label{eq:construction-example}
\end{equation}
Using the mean value theorem in Proposition~\ref{prop:MVT-metaharmonic}, we have 
\[
U_{k}^{\mu_{+}*h_{r}}(x) \le U_{k}^{\mu_{+}}(x) \text{ for all } x \in \mathbb{R}^{n} \text{ with equality if } \mu_{+}(B_{r}(x))=0,
\]
which implies 
\[
\tilde{u}_{k} \le U_{k}^{\mu} \text{ in } \mathbb{R}^{n} \quad \text{and} \quad \tilde{u}_{k} = U_{k}^{\mu} \text{ in } \mathbb{R}^{n} \setminus B_{\gamma + r}.
\]
Finally we note that 
\[
(\Delta + k^{2})\tilde{u}_{k}(x) = (-\mu_{+}*h_{r} + \mu_{-})(x) \ge -\mu_{+}*h_{r}(x) = -\frac{\mu_{+}(B_{r}(x))}{c_{n,k,r}^{{\rm MVT}}} \ge - \frac{\mu_{+}(\mathbb{R}^{n})}{c_{n,k,r}^{{\rm MVT}}} \ge -1,
\]
which shows that $\tilde{u}_{k} \in \mathscr{F}_{k}(\mu)$.
\end{proof}

For fixed $\mu$ we now choose the parameter $r$ in Lemma~\ref{lem:non-empty-patch} in order to find an explicit range of $k>0$ for which the lemma applies.

By the definition of $c_{n,k,r}^{\rm MVT}$ the second inequality in \eqref{eq:non-empty-assumption-patch} is equivalent to 
\begin{equation}
k \le \frac{(2\pi kr)^{\frac{1}{2}} J_{\frac{n}{2}}(kr)^{\frac{1}{n}}}{\mu_{+}(\mathbb{R})^{\frac{1}{n}}}\,. \label{eq:non-empty-assumption-patch1}
\end{equation}
Since $t \mapsto t^{\frac{1}{2}}J_{\frac{n}{2}}(t)^{\frac{1}{n}}$ is strictly increasing on $[0, j_{\frac{n-2}{2},1}]$, we see that in order to maximize the range of $k$ we here want to choose $k r$ as large as possible.

By the definition of $R(n, k)$ we see that the range of $r$ we can consider is given by
\[
0< rk < \frac{1}{2}j_{\frac{n-2}2,1}-\gamma k\,.
\]
Therefore, if we assume that
\[
k \le \frac{j_{\frac{n-2}{2},1}}{4\gamma} 
\]
we can choose
\[
rk =\frac{\frac{1}{2} j_{\frac{n-2}{2},1} - \gamma k}{2} \ge \frac{j_{\frac{n-2}{2},1}}{8}.
\]
By the monotonicity of $t \mapsto t^{\frac{1}{2}}J_{\frac{n}{2}}(t)^{\frac{1}{n}}$ we then know that \eqref{eq:non-empty-assumption-patch1} is satisfied for all
\[
k \le \frac{( \frac{\pi j_{\frac{n-2}{2},1}}{4})^{\frac{1}{2}} J_{\frac{n}{2}}(\frac{j_{\frac{n-2}{2},1}}{8})^{\frac{1}{n}}}{\mu_{+}(\mathbb{R})^{\frac{1}{n}}}\,.
\]
Consequently, for any 
\[
c_{n} \leq \min \Bigl\{ \bigg( \frac{\pi j_{\frac{n-2}{2},1}}{4} \bigg)^{\frac{1}{2}} J_{\frac{n}{2}} \bigg( \frac{j_{\frac{n-2}{2},1}}{8} \bigg)^{\frac{1}{n}} , \frac{j_{\frac{n-2}{2},1}}{4} \Bigr\}
\]
we conclude the following lemma:

\begin{lem}
\label{lem:non-empty} 
Fix any $\gamma > 0$ and $\mu\in L^{\infty}(B_{\gamma})$.
There exists a positive constant $c_{n}$ {\rm (}depending only on the dimension $n${\rm )} such that if
\begin{equation}\label{eq:k bound existence lemma}
0 < k \le  c_n \min  \{\gamma^{-1},\mu_{+}(\mathbb{R}^{n})^{-\frac{1}{n}}\}
\end{equation}
then $\mathscr{F}_{k}(\mu)$ contains an element $\tilde{u}_{k}$ with 
\begin{equation}
\tilde{u}_{k}=U_{k}^{\mu} \quad \text{near } \mathbb{R}^{n} \setminus B_{R(n,k)}.\label{eq:near-boundary}
\end{equation}
\end{lem}

The following proposition will be used to define partial balayage in terms of the solution of our obstacle problem.
\begin{prop}\label{prop:partial-balayage-obstacle} 
Let $\mu$ and $k>0$ be as in Lemma~{\rm \ref{lem:non-empty}}.
Then there exists a largest element $V_{k}^{\mu}$ in $\mathscr{F}_{k}(\mu)$. In addition, the element $V_{k}^{\mu}$ satisfies 
\begin{equation}
\langle 1 + (\Delta + k^{2})V_{k}^{\mu} , V_{k}^{\mu} - U_{k}^{\mu} \rangle = 0,\label{eq:balayage-characterization2}
\end{equation}
where $\langle \cdot ,\cdot \rangle$ is the $H^{-1}(B_{R(n,k)}) \times H_{0}^{1}(B_{R(n,k)})$ duality pairing. 
\end{prop}

\begin{rem} \label{rem:identical-near-boundary}
Note that Lemma~\ref{lem:non-empty} implies that there exists $\tilde{u}_{k}\in \mathscr{F}_{k}(\mu)$ satisfying  $\tilde u_k = U_k^\mu$ near $\partial B_{R(n,k)}$. Therefore, if $V_{k}^{\mu}$ is that largest element in $\mathscr{F}_{k}(\mu)$ then
\begin{equation}
V_{k}^{\mu}=U_{k}^{\mu}\quad\text{near }\partial B_{R(n,k)}. \label{eq:identical-near-boundary-rem}
\end{equation}
Therefore, we can extend $V_{k}^{\mu}$ to the whole $\mathbb{R}^{n}$, by defining $V_{k}^{\mu} := U_{k}^{\mu}$ outside $B_{R(n,k)}$.
\end{rem}

The proof of the proposition is based on variational arguments. In particular, we shall need the following elementary lemma several times in the proof.
\begin{lem}\label{lem: properties of bilinear form}
Fix $k> 0$ and $0<R<j_{\frac{n-2}{2},1}k^{-1}$. Let $a \colon H_0^1(B_R)\times H_0^1(B_R)\to \mathbb{R}$ be the symmetric bilinear form defined by 
\begin{equation}
a(u_{1},u_{2}) :=\int_{B_R}(\nabla u_{1}\cdot\nabla u_{2}-k^{2}u_{1}u_{2})\,d\mathsf{m}. \label{eq:bilinear-form-Appendix}
\end{equation}
Then $a$ is continuous, positive, and coercive.
\end{lem}

\begin{proof}
That $a$ is a continuous is clear from the definition. To prove that the form is coercive and positive we observe that by assumption $k^2$ is strictly smaller than the first eigenvalue of Dirichlet Laplacian on $B_R$ (which is exactly $j_{\frac{n-2}{2},1}^2R^{-2}$). Therefore,
\begin{equation*}
 a(u, u) = \int_{B_R}(|\nabla u|^2 -k^{2}|u|^2)\,d\mathsf{m}\geq (1- k^2R^2j_{\frac{n-2}{2},1}^{-2})\int_{B_R}|\nabla u|^2\,d\mathsf{m}.
\end{equation*}
This concludes the proof.
\end{proof}

\medskip

\begin{proof}
[Proof of Proposition~{\rm \ref{prop:partial-balayage-obstacle}}]
Let $\varphi\in H^{1}(B_{R(n,k)})$ be the unique solution to 
\begin{equation}
\begin{cases}
(\Delta+k^{2})\varphi = -1 & \text{in}\;\;B_{R(n,k)},\\
\varphi=U_{k}^{\mu} & \text{on}\;\;\partial B_{R(n,k)},
\end{cases} \label{eq:varphi-auxiliary}
\end{equation}
and define 
\begin{align*}
\tilde{\mathscr{F}}_{k}(\mu) & =\begin{Bmatrix}\begin{array}{l|l}
w=\varphi-v & v\in\mathscr{F}_{k}(\mu)\end{array}\end{Bmatrix}
 =\begin{Bmatrix}\begin{array}{l|l}
w\in H_{0}^{1}(B_{R(n,k)}) & \begin{array}{l}
(\Delta + k^{2}) w \le 0\text{ in }B_{R(n,k)}\\
w\ge\varphi -U_{k}^{\mu}\text{ in }B_{R(n,k)}
\end{array}\end{array}\end{Bmatrix}.
\end{align*} 
We claim that there exists a smallest element $u_{*}$ of $\tilde{\mathscr{F}}_{k}(\mu)$. If this is the case then
\begin{equation}
V_{k}^{\mu} := \varphi-u_{*}\quad\text{in}\;\;B_{R(n,k)} \label{eq:u-star-minimizer}
\end{equation}
is the largest element of $\mathscr{F}_{k}(\mu)$. 

To see that there exists a smallest element in $\tilde{\mathscr{F}}_{k}(\mu)$ we argue as follows.
Let $a$ be the bilinear form defined in~\eqref{eq:bilinear-form-Appendix} with $R = R(n,k)$.
By Lemma~\ref{lem: properties of bilinear form} $a$ is symmetric, continuous, and coercive. Define the constraint set 
\begin{equation}\label{eq:constraint set step 2 obstacle problem}
\tilde{\mathcal{K}}_{k}:=\begin{Bmatrix}\begin{array}{l|l}
u\in H_{0}^{1}(B_{R(n,k)}) & u\ge\varphi-U_{k}^{\mu}\end{array}\end{Bmatrix}.
\end{equation}

Note that $\varphi-U_k^\mu \in H_0^1(B_{R(n, k)})$ by definition of $\varphi$, thus $\tilde{\mathcal{K}}_k$ is nonempty. 
Since $\tilde{\mathcal{K}}_k$ is a nonempty closed convex subset of $H_{0}^{1}(B_{R(n,k)})$, Stampacchia's theorem~\cite[Theorem~5.6]{Bre11PDE} implies that there exists a unique $u_{*} \in \tilde{\mathcal{K}}_k$ that minimizes the functional
\begin{equation}
u \mapsto a(u, u) \label{eq:u-star-minimizer-general}
\end{equation}
and $u_{*}\in\tilde{\mathcal{K}}_k$ can also be characterized by 
\begin{equation}
a(u_{*},u-u_{*}) 
\ge 0 \quad \text{for all } u \in \tilde{\mathcal{K}}_{k}. \label{eq:minimizer2}
\end{equation}
Plugging in $u = u_{*} + \phi$ with non-negative $\phi \in C_{c}^{\infty}(B_{R(n,k)})$ into \eqref{eq:minimizer2},
the definition of $a$ in~\eqref{eq:bilinear-form-Appendix} implies that 
\begin{equation}
(\Delta + k^{2})u_{*} \le 0 \text{ in $B_{R(n, k)}$}.
\label{eq:complementarity1}
\end{equation}
In particular, we conclude that $u_{*} \in \tilde{\mathscr{F}}_{k}(\mu)$. Finally, by arguing as in the proof of \cite[Theorem~II.6.4]{KS00IntroductionVariationalInequalities}, one can prove that $u_{*} \le v$ in $B_{R(n, k)}$ for all $v \in \tilde{\mathscr{F}}_{k}(\mu)$. Consequently, we have found the desired smallest element in $\tilde{\mathscr{F}}_{k}(\mu)$.

Choosing $u=\varphi-U_{k}^{\mu}\in H_{0}^{1}(B_{R(n,k)})$ in \eqref{eq:minimizer2}, we have 
\begin{equation}
\langle (\Delta + k^{2})u_{*}, \varphi - U_{k}^{\mu} - u_{*} \rangle \le 0.\label{eq:conclusion-maximizer1-1}
\end{equation} 
Since $u_{*}\in\tilde{\mathcal{K}}_{k}$ we know that $u_{*}\ge\varphi-U_{k}^{\mu}$. Along with \eqref{eq:complementarity1} and \eqref{eq:conclusion-maximizer1-1} this inequality implies 
\begin{equation}
\langle (\Delta + k^{2})u_{*}, \varphi - U_{k}^{\mu} - u_{*} \rangle = 0.\label{eq:conclusion-maximizer1}
\end{equation}
Combining \eqref{eq:conclusion-maximizer1} with \eqref{eq:u-star-minimizer}, as well as \eqref{eq:varphi-auxiliary}, we obtain 
\begin{align*}
0 = \langle (\Delta + k^{2})(\varphi - V_{k}^{\mu}), V_{k}^{\mu} - U_{k}^{\mu} \rangle 
 = - \langle 1 + (\Delta + k^{2})V_{k}^{\mu} , V_{k}^{\mu} - U_{k}^{\mu} \rangle,
\end{align*}
which shows that $V_{k}^{\mu}$ satisfies \eqref{eq:balayage-characterization2}.
\end{proof}

We are now ready to define partial balayage for the Helmholtz operator.
\begin{defn} \label{def:obstacle-balayage}
Let $\mu$ and $k>0$ be as in Lemma~{\rm \ref{lem:non-empty}}. The \emph{partial balayage} of $\mu$ is defined by 
\begin{equation}
\Bal (\mu) := -(\Delta + k^{2})V_{k}^{\mu} \quad \text{in distribution sense}, \label{eq:balayage-definition-equivalent}
\end{equation}
where $V_{k}^{\mu}$ is given by Proposition~\ref{prop:partial-balayage-obstacle}. 
\end{defn}

We have the following basic properties of the partial balayage measure and the corresponding potential.

\begin{lem} \label{lem:basic-properties}
\begin{subequations}
Let $\mu$ and $k>0$ be as in Lemma~{\rm \ref{lem:non-empty}}. Then \begin{align}
\Bal(\mu) \leq 1 \quad \text{in }\mathbb{R}^{n}, \label{eq:basic1}\\
U_k^{\Bal(\mu)} \equiv V_{k}^{\mu} \quad \text{in }\mathbb{R}^{n}.
\label{eq_uknu_vkmu}
\end{align}
We also have 
\begin{equation}
U_{k}^{\Bal(\mu)} \le U_{k}^{\mu} \quad \text{in }\mathbb{R}^{n} \label{eq:basic2}
\end{equation}
and 
\begin{equation}
U_k^{\Bal(\mu)} = U_k^{\mu} \quad \text{in a neighborhood of }\mathbb{R}^{n}\setminus B_{R(n,k)}. \label{eq:basic3}
\end{equation}
\end{subequations}
\end{lem}

\begin{proof} 
If we can show \eqref{eq_uknu_vkmu}, then \eqref{eq:basic1} and \eqref{eq:basic2} are immediate consequence of Proposition~\ref{prop:partial-balayage-obstacle} and the definition of $\mathscr{F}_{k}(\mu)$, while \eqref{eq:basic3} is an immediate consequence of Remark~\ref{rem:identical-near-boundary}. It remains to prove \eqref{eq_uknu_vkmu}. 
Write $u = U_{k}^{\mu}-V_{k}^{\mu}$, so $-(\Delta+k^2)u = \mu - \Bal(\mu)$. Note that $u$ has compact support by Remark~\ref{rem:identical-near-boundary}. Thus 
\[
u = \tilde{\Phi}_k * (-(\Delta+k^2)u) = \tilde{\Phi}_k * (\mu - \Bal(\mu)) = U_k^{\mu}-U_k^{\Bal(\mu)}.
\]
This proves that $U_{k}^{\Bal(\mu)}=V_{k}^{\mu}$.
\end{proof}

We also make the following observation which will be very useful in our construction of $k$-quadrature domains.
\begin{lem} \label{lem:equivalent-saturated-set1}
Let $\mu$ and $k>0$ be as in Lemma~{\rm \ref{lem:non-empty}}. If 
\begin{equation}
\Bal(\mu)=\chi_{D}\mathsf{m}\quad\text{for some open set }D\label{eq:structure-goal}
\end{equation}
then 
\begin{subequations}
\begin{align}
U_{k}^{D} & \le U_{k}^{\mu}\quad\text{in }\mathbb{R}^{n},\label{eq:characterization1a}\\
U_{k}^{D} & =U_{k}^{\mu}\quad\text{in }\mathbb{R}^{n}\setminus D.\label{eq:characterization1b}
\end{align}
\end{subequations}
\end{lem}

\begin{proof}
We already proved \eqref{eq:characterization1a} in Lemma~\ref{lem:basic-properties}. Set $\nu=\Bal(\mu)$ and rewrite \eqref{eq:balayage-characterization2} as 
\[
0 = \int_{B_{R(n,k)}} (U_{k}^{\mu} - U_{k}^{D})(1 - \chi_{D}) \,dx = \int_{B_{R(n,k)\setminus D}} (U_{k}^{\mu} - U_{k}^{D}) \,dx.
\]
Combining this equality with \eqref{eq:characterization1a} and \eqref{eq:basic3}, we conclude \eqref{eq:characterization1b}. 
\end{proof}

We end this section by quickly relating our definition of partial balayage through an obstacle problem to a formulation in terms of energy minimization. Such a formulation is classical in the setting of $k=0$ (see for instance~\cite{Gus04LecturesBalayage}).

\begin{rem}[Partial balayage and energy minimization]
Let $\mu$ and $k>0$ be as in Lemma~{\rm \ref{lem:non-empty}}. Using Proposition~\ref{prop:structure-balayage}, we know that $\nu := \Bal (\mu) \in L^{\infty}(B_{R(n,k)})$. We define the following bilinear form: 
\begin{equation*}
(\mu_{1},\mu_{2})_{e,k}  := \iint_{B_{R(n,k)}\times B_{R(n,k)}}\tilde{\Phi}_{k}(x-y)\,d\mu_{1}(y)\,d\mu_{2}(x) = \int_{B_{R(n,k)}}U_{k}^{\mu_{1}}(x)\,d\mu_{2}(x)
\end{equation*}
for all $\mu_{1},\mu_{2} \in L^{\infty}(B_{R(n,k)})$. Using Lemma~\ref{lem:basic-properties}, we can write \eqref{eq:balayage-characterization2} as $(\nu - \mu, \mathsf{m} - \nu)_{e,k} = 0$.
Accordingly, for each $\sigma \in L^{\infty}(B_{R(n,k)})$ with $\sigma \le \mathsf{m}$, we see that 
\begin{equation}
(\nu - \mu,  \sigma - \nu)_{e,k} = (\nu - \mu, \sigma - \mathsf{m})_{e,k} \ge 0, \label{eq:balayage-Stampacchia1}
\end{equation}
where the inequality follows from Lemma~\ref{lem:basic-properties}. By defining the ``energy'' $E_{k}(\lambda) := (\lambda,\lambda)_{e,k}$, we see that 
\[
(\nu - \mu, \sigma - \nu)_{e,k} = - E_{k}( \nu - \mu ) + (\nu - \mu, \sigma - \mu)_{e,k},
\]
thus from \eqref{eq:balayage-Stampacchia1}, we have 
\begin{equation}
E_{k}( \nu - \mu ) \le (\nu - \mu, \sigma - \mu)_{e,k} \quad \text{for all } \sigma \in L^{\infty}(B_{R(n,k)}) \text{ with }\sigma \le \mathsf{m}. \label{eq:energy1}
\end{equation}
When $U_{k}^{\mu_{1}},U_{k}^{\mu_{2}} \in H_{0}^{1}(B_{R(n,k)})$ we can compute 
\begin{align*}
(\mu_{1},\mu_{2})_{e,k} & = - \int_{B_{R(n,k)}} U_{k}^{\mu_{1}}(\Delta + k^{2}) U_{k}^{\mu_{2}} \,dx = a(U_{k}^{\mu_{1}},U_{k}^{\mu_{2}}), \\
E_{k}(\mu_{1}) & = a(U_{k}^{\mu_{1}},U_{k}^{\mu_{1}}) \ge 0, 
\end{align*}
where $a(\cdot,\cdot)$ is the (real) inner product given by Lemma~\ref{lem: properties of bilinear form}. Thus the notion of $E_k$ as an energy functional makes sense.
Using this observation and the Cauchy-Schwarz inequality, if we restrict $\sigma$ in \eqref{eq:energy1} to those functions satisfying $U_{k}^{\sigma} - U_{k}^{\mu} \in H_{0}^{1}(B_{R(n,k)})$, then we have 
\begin{align*}
E_{k}(\nu-\mu) & \le a(U_{k}^{\nu-\mu},U_{k}^{\sigma-\mu}) \\
& \le (a(U_{k}^{\nu}-U_{k}^{\mu},U_{k}^{\nu}-U_{k}^{\mu}))^{\frac{1}{2}}(a(U_{k}^{\sigma}-U_{k}^{\mu},U_{k}^{\sigma}-U_{k}^{\mu}))^{\frac{1}{2}} \\
& \equiv (E_{k}(\nu-\mu))^{\frac{1}{2}}(E_{k}(\sigma-\mu))^{\frac{1}{2}}.
\end{align*}
Therefore, the partial balayage $\nu$ minimizes the energy in the following sense: 
\begin{equation}
E_{k}(\nu - \mu) \le E_{k}(\sigma - \mu)  \quad \text{for all } \sigma \in L^{\infty}(B_{R(n,k)}) \text{ with } U_{k}^{\sigma} \in \mathscr{F}_{k}(\mu),
\end{equation}
where $\mathscr{F}_{k}(\mu)$ is given by \eqref{eq: scrF}. Here we also refer to \cite[Section~30]{Tre75BasicLinearPDE} for a related discussion. \end{rem}

\section{\label{sec:Structure-of-partial}Structure of partial balayage}

In this section we prove the following proposition which provides information concerning the structure of $\Bal(\mu)$. This will in particular be useful when we later on wish to construct $k$-quadrature domains.

\labeltext{$D(\mu)$ saturated set for $\Bal(\mu)$}{index:SaturatedSetDmu} \labeltext{$\omega(\mu)$ non-contact set for an obstacle problem}{index:Setomegamu}
\begin{prop}[Structure of partial balayage] \label{prop:structure-balayage}
Let $\mu$ and $k>0$ be as in Lemma~{\rm \ref{lem:non-empty}} and let $\nu := \Bal(\mu)$. Then 
\begin{equation}
\min\{\mu,\dm\}\le\nu \le \dm \quad \text{in }\mathbb{R}^{n}. \label{eq:nu-explicit-bound}
\end{equation}
Furthermore, if we define the open sets
\begin{align}
	D(\mu)&:=\mathbb{R}^{n}\setminus\supp(\mathsf{m}-\nu) \mbox{, and}\\
	\omega(\mu)&:=\begin{Bmatrix}\begin{array}{l|l} 
x\in\mathbb{R}^{n} & U_{k}^{\mu}(x)>U_{k}^{\nu}(x)\end{array}\end{Bmatrix}\label{eq:set-omega}
\end{align} 
then $\omega(\mu) \subset D(\mu)$ and for each measurable set $D$ with $\omega(\mu)\subset D \subset D(\mu)$ we have 
\begin{equation}
\nu = \chi_{D}\dm +\chi_{\mathbb{R}^{n}\setminus D}\mu.
\label{eq:balayage-structure}
\end{equation}
\end{prop}

\begin{rem} \label{rem:sructure-balayage-explain}
The corresponding result for $k=0$ can be found in \cite[Theorem~2.3(c)]{Gus90QuadratureDomains}. 
We refer also to \cite{GS09PartialBalayage,Gus04LecturesBalayage,Sjo07PartialBalayage}, and in particular \cite[Figure~3]{Gus04LecturesBalayage} for a visualization.

The set $D(\mu)$ is called the saturated set for $\nu=\Bal(\mu)$, which is the largest open set $\mathcal{O}$ in $\mathbb{R}^{n}$ such that $\nu|_{\mathcal{O}}=\mathsf{m}|_{\mathcal{O}}$. By our assumptions and Lemma~\ref{lem:basic-properties} both $\supp(\nu)$ and $\supp(\mu)$ are contained in $B_{R(n,k)}$, and hence $\overline{D(\mu)} \subset B_{R(n,k)}$. We also note that if $\mathcal{O}\subset D(\mu) \setminus \omega(\mu)$ has positive Lebesgue measure then $\mu|_{\mathcal{O}}= \nu|_{\mathcal{O}}=\dm|_{\mathcal{O}}$. In particular, if the density of $\mu$ is greater than $1$ on $\supp(\mu)$ it holds that $\dm(D(\mu)\setminus \omega(\mu))=0$. See also \cite[Remark~2.4]{Gus90QuadratureDomains} for some discussions on the relation between $D(\mu)$ and $\omega(\mu)$ for the case when $k=0$.
\end{rem}

\begin{proof} [Proof of Proposition~{\rm \ref{prop:structure-balayage}}]

\textbf{Step 1: A minimization problem.} Let $\xi\in H_{0}^{1}(B_{R(n,k)})$ be the unique solution to $-(\Delta+k^{2})\xi=(1-\mu)_{+}$, and consider the constraint set 
\[
\hat{\mathcal{K}}_{k}=\begin{Bmatrix}\begin{array}{l|l}
w\in H_{0}^{1}(B_{R(n,k)}) & w\ge\xi-u_{*}\end{array}\end{Bmatrix},
\]
where $u_{*} \in H_{0}^{1}(B_{R(n,k)})$ is the function appearing in the proof of Proposition~\ref{prop:partial-balayage-obstacle}. We recall that $u_*$ minimizes the functional $a(u, u)$ among all functions in $\tilde{\mathcal{K}}_k$, where $a$ is the bilinear form defined in \eqref{eq:bilinear-form-Appendix} and $\tilde{\mathcal{K}}_k$ was defined in~\eqref{eq:constraint set step 2 obstacle problem}. Note that $\hat{\mathcal{K}}_{k}$ is nonempty since $\xi-u_* \in \hat{\mathcal{K}}_{k}$.

By Lemma~\ref{lem: properties of bilinear form} and Stampacchia's Theorem (see~\cite[Theorem~5.6]{Bre11PDE}) there exists a unique $w_{*}\in \hat{\mathcal{K}}_{k}$ which minimizes the functional $w \mapsto a(w, w)$ among $w \in \hat{\mathcal{K}}_{k}$. Moreover, the minimizer $w_*$ is characterized by the property
\begin{equation}
a(w_{*},w-w_{*}) 
= \langle -(\Delta + k^{2}) w_{*}, w - w_{*} \rangle 
\ge 0 \quad \text{for all } w\in \hat{\mathcal{K}}_{k}. \label{eq:minimizer3}
\end{equation}

\medskip

\textbf{Step 2: Complementarity formulation.} Since $w_{*}\in \hat{\mathcal{K}}_{k}$, we can in~\eqref{eq:minimizer3} restrict $w$ to those satisfying $w\ge w_{*}$. The definition of $a$ implies that
\begin{equation}
(\Delta+k^{2})w_{*} \le 0 \quad\text{in } B_{R(n,k)}.\label{eq:complementarity1-1}
\end{equation}
Choosing $w=\xi-u_{*}$ in \eqref{eq:minimizer3}, 
\[
\langle (\Delta + k^{2}) w_{*}, \xi - u_{*} - w_{*} \rangle \le 0,
\]
which along with \eqref{eq:complementarity1-1} and the fact that $w_{*}\ge\xi-u_{*}$ implies
\begin{equation}
\langle (\Delta + k^{2}) w_{*}, \xi - u_{*} - w_{*} \rangle = 0.\label{eq:complementarity2-1}
\end{equation}

In fact, if $w_{*}\in \hat{\mathcal{K}}_{k}$ satisfies \eqref{eq:complementarity1-1}
and \eqref{eq:complementarity2-1}, then
\begin{align*}
&\langle (\Delta + k^{2}) w_{*}, w - w_{*} \rangle \\
& \qquad = \langle (\Delta + k^{2}) w_{*}, w - (\xi - u_{*}) \rangle + \langle (\Delta + k^{2}) w_{*}, \xi - u_{*} - w_{*} \rangle \le 0 ,
\end{align*}
for all $w\in \hat{\mathcal{K}}_{k}$. Hence the minimizer $w_{*}\in \hat{\mathcal{K}}_{k}$ can also be characterized by the complementarity problem \eqref{eq:complementarity1-1} and \eqref{eq:complementarity2-1}. 

\medskip

\textbf{Step 3: An energy inequality.} 
We can rewrite \eqref{eq:complementarity2-1} as 
\begin{equation}
\langle (\Delta + k^{2}) w_{*}, \xi - w_{*} \rangle = \langle (\Delta + k^{2}) w_{*}, u_{*} \rangle. \label{eq:bounded-proof1}
\end{equation}
The inequalities $(\Delta + k^{2})\xi = -(1-\mu)_{+} \le 0$ and $w_{*}\ge\xi-u_{*}$ (i.e. $u_{*}\ge\xi-w_{*}$) thus imply that
\begin{equation}
\langle (\Delta + k^{2})\xi , \xi - w_{*} \rangle \ge \langle (\Delta + k^{2})\xi , u_{*} \rangle. \label{eq:bounded-proof2}
\end{equation}
Combining \eqref{eq:bounded-proof1} and \eqref{eq:bounded-proof2}, one finds 
\begin{align*}
 & \quad a(\xi-w_{*},\xi-w_{*}) = \langle -(\Delta + k^{2})(\xi - w_{*}), \xi - w_{*} \rangle \\
 & \le \langle -(\Delta + k^{2})(\xi - w_{*}), u_{*} \rangle =a(\xi-w_{*},u_{*})\\
 & \le a(\xi-w_{*},\xi-w_{*})^{\frac{1}{2}}a(u_{*},u_{*})^{\frac{1}{2}}.
\end{align*}
By Lemma~\ref{lem: properties of bilinear form} the bilinear form $a$ is positive, and thus we obtain the energy inequality 
\begin{equation}
a(\xi-w_{*},\xi-w_{*})\le a(u_{*},u_{*}).\label{eq:bounded-proof3}
\end{equation}

\medskip

\labeltext{Step~4}{Step4:prop:structure-balayage}\textbf{Step 4: Verifying that $w_{*}=\xi-u_{*}$.}
If we can show that $\xi-w_{*}\in\tilde{\mathcal{K}}_{k}$, i.e.\ that it satisfies 
\begin{equation}
\xi-w_{*}\ge\varphi-U_{k}^{\mu}\quad\text{in}\;\;B_{R(n,k)}\label{eq:WTS-goal}
\end{equation}
where $\varphi$ is the function in \eqref{eq:varphi-auxiliary}, then since $u_*$ minimizes $a(u, u)$ among all $u\in \tilde{\mathcal{K}}_k$ the inequality in \eqref{eq:bounded-proof3} implies that $\xi-w_{*}=u_{*}$ in $B_{R(n,k)}$, in other words $w_{*}=\xi-u_{*}$ in $B_{R(n,k)}$. 

To prove~\eqref{eq:WTS-goal} we argue as follows. Let 
\begin{equation}
\phi:=\min\{w_{*},\xi-(\varphi-U_{k}^{\mu})\}\quad\text{in}\;\;B_{R(n,k)}.\label{eq:auxiliary-phi}
\end{equation}
By the definition of $\phi$ and Proposition~\ref{prop: Min principle},
\begin{equation}
\phi \le w_{*}, \qquad \phi \in \hat{\mathcal{K}}_{k},\quad
\text{and} \quad -(\Delta + k^{2})\phi \ge 0 \quad \text{in } B_{R(n,k)} .
\label{eq:bounded-proof7}
\end{equation}
Using \eqref{eq:bounded-proof7} and the facts that $w_* \in H^1_0(B_{R(n,k)})$ and $-(\Delta+k^2) w_* \geq 0$ in $B_{R(n,k)}$, we have in terms of distributional pairings in $B_{R(n,k)}$ that 
\begin{align*}
    a(\phi,\phi) &= \langle -(\Delta+k^2) \phi, \phi \rangle \\
     &\leq \langle -(\Delta+k^2) \phi, w_* \rangle = \langle \phi, -(\Delta+k^2)w_* \rangle \\
     &\leq \langle w_*, -(\Delta+k^2)w_* \rangle = a(w_*,w_*).
\end{align*}
Since $w_{*}$ was defined to be the unique minimizer of $a(w, w)$ among $w\in\hat{\mathcal{K}}_{k}$, we obtain
\[
\phi=w_{*}\quad\text{in}\;\;B_{R(n,k)}.
\]
By the definition of $\phi$ this can equivalently be stated as
\[
w_{*}\le\xi-(\varphi-U_{k}^{\mu})\quad\text{in}\;\;B_{R(n,k)}.
\]
After rearranging we deduce the desired inequality \eqref{eq:WTS-goal}. 
By the discussion following \eqref{eq:WTS-goal} it holds that 
\[
w_{*}=\xi-u_{*}\quad\text{in}\;\;B_{R(n,k)}.
\]

\medskip

\textbf{Step 5: Proving \eqref{eq:nu-explicit-bound}.}
By \ref{Step4:prop:structure-balayage}, \eqref{eq:complementarity1-1}, and the definition of $\xi$,
\begin{equation}
(\Delta+k^{2})u_{*} \ge (\Delta+k^{2})\xi = -(1-\mu)_{+}.
\label{eq:conclusion-bounded1}
\end{equation}
From \eqref{eq:u-star-minimizer} and \eqref{eq_uknu_vkmu} we deduce that
\begin{equation}
u_{*}=\varphi-U_{k}^{\nu}. \label{eq:u-star-balayage}
\end{equation}
Combining \eqref{eq:conclusion-bounded1} and \eqref{eq:u-star-balayage}, we obtain that in $B_{R(n,k)}$ 
\begin{align*}
 & \nu - 1 \ge -(1-\mu)_{+} = -\max\{1-\mu,0\} = \min\{\mu-1,0\}\\
\iff & \min\{1,\mu\}\le\nu.
\end{align*}
By the definition of partial balayage $\nu\le1$ in $B_{R(n,k)}$, and we have arrived at~\eqref{eq:nu-explicit-bound}. 

\medskip

\textbf{Step 6: Proving \eqref{eq:balayage-structure}.} 
By the Calder\'{o}n--Zygmund inequality $U_{k}^{\mu},U_{k}^{\nu} \in \bigcap_{p < \infty} W_{\rm loc}^{2,p}(\mathbb{R}^{n})$ and hence $\omega(\mu)$ is well defined as an open set. 
From \eqref{eq:balayage-characterization2} and Proposition~\ref{prop:partial-balayage-obstacle}, it follows that
\[
0\le\int_{\omega(\mu)}(U_{k}^{\mu}-U_{k}^{\nu})\,d(\mathsf{m}-\nu)\le\int_{B_{R(n,k)}}(U_{k}^{\mu}-U_{k}^{\nu})\,d(\mathsf{m}-\nu)=0,
\]
and hence 
\begin{equation}
\int_{\omega(\mu)}(U_{k}^{\mu}-U_{k}^{\nu})\,d(\mathsf{m}-\nu)=0.\label{eq:geometry0}
\end{equation}
Consequently, $\nu|_{\omega(\mu)}= \dm|_{\omega(\mu)}$. Since $\omega(\mu)^c = \{x\in \mathbb{R}^n: U^\nu_k(x) = U^\mu_k(x)\}$ and $U^\mu_k, U^\nu_k \in W^{2,p}_{loc}(\mathbb{R}^n)$ it holds that $(\Delta+k^2)(U^\nu_k-U^\mu_k) =0$ almost everywhere on this set. Therefore, $\nu|_{\omega(\mu)^c}= \mu|_{\omega(\mu)^c}$. 

Consequently, for any $D$ as in the proposition we by the definition of $D(\mu)$ have $\nu|_{D \setminus \omega(\mu)} = \dm|_{D \setminus \omega(\mu)}$ and thus the claimed decomposition
\begin{equation*}
	\nu = \chi_D \dm + \chi_{\mathbb{R}^n\setminus D}\mu
\end{equation*}
follows. This completes the proof of Proposition~\ref{prop:structure-balayage}.
\end{proof}

We next deduce the following lemma.

\begin{lem}
	\label{lem:observation-new-patch}
Let $\mu$ and $k>0$ be as in Lemma~{\rm \ref{lem:non-empty}}. Suppose there is an open set $D$ such that $\overline{D} \subset B_{R(n,k)}$ and $\supp (\mu) \subset D$ and a distribution $u$ satisfying
\begin{equation}
\begin{cases}
(\Delta + k^{2})u = \chi_{D} - \mu & \text{in } B_{R(n,k)}, \\
u > 0 & \text{in } D, \\
u = 0 & \text{in } B_{R(n,k)} \setminus {D}.
\end{cases} \label{eq:structure-patch}
\end{equation}
Then $\Bal (\mu) = \chi_{D} \mathsf{m}$, $D = \omega(\mu)$ and $D$ is a $k$-quadrature domain for $\mu$.
\end{lem}

\begin{proof}[Proof of Lemma~{\rm \ref{lem:observation-new-patch}}]
Since $u$ (extended by zero outside $B_{R(n,k)}$) is a compactly supported distribution, we have 
\[
u = \tilde{\Phi}_{k}*(-(\Delta + k^{2})u) = U_{k}^{\mu} - U_{k}^{D}.
\]
Since $u$ is non-negative, then we know that $U_{k}^{D} \in \mathscr{F}_{k}(\mu)$, where $\mathscr{F}_{k}(\mu)$ is the collection of functions given in \eqref{eq: scrF}. For each $v \in \mathscr{F}_{k}(\mu)$, since $u = 0$ in $B_{R(n,k)} \setminus D$, we see that 
\[
w := U_{k}^{D} - v = U_{k}^{D} - U_{k}^{\mu} + U_{k}^{\mu} - v \ge 0 \quad \text{in } B_{R(n,k)} \setminus D.
\]
On the other hand, we have $(\Delta + k^{2})w = -1 - (\Delta + k^{2})v \le 0$ in $D$. Therefore the maximum principle in Proposition~\ref{prop: Max principle} implies that $w \ge 0$ in $D$ as well. This shows that $U_{k}^{D}$ is the largest element in $\mathscr{F}_{k}(\mu)$, so by the definition of partial balayage~\eqref{eq:balayage-definition-equivalent} we have 
\[
\Bal(\mu) = -(\Delta + k^{2})U_{k}^{D} = \chi_{D} \mathsf{m}.
\]
By the above we see that $D = \{u>0\}=\{U^\mu_k>U^{\Bal(\mu)}_k\}=\omega(\mu)$. 

Since $u\in C^{1}(\mathbb{R}^n)$ attains its minimum in $D^c$ it holds that $|\nabla u| =0$ in $D^c$. Therefore, since by assumption $\supp(\mu) \subset D$, Proposition~\ref{prop_quadrature_pde_equivalence} implies that $D$ is a $k$-quadrature domain for~$\mu$.
\end{proof}

\section{\label{sec:balayage-iterative}Performing balayage in smaller steps }

Fix $\gamma>0$ and assume that $\mu_1, \mu_2 \in L^\infty(B_\gamma)$ are non-negative. By Proposition~\ref{prop:partial-balayage-obstacle}, there exists a positive constant $c_{n}$ such that if 
\begin{equation}\label{eq: k bound mu1}
0 <k <  c_n \min  \{\gamma^{-1},\mu_{1}(\mathbb{R}^n)^{-\frac{1}{n}}\},
\end{equation}
then $U_{k}^{\Bal(\mu_{1})}$ is the largest element in $\mathscr{F}_{k}(\mu_{1})$ (defined as in~\eqref{eq: scrF})
and $U_{k}^{\Bal(\mu_{1})}=U_{k}^{\mu_{1}}$ near $\partial B_{R(n,k)}$. 
Again Proposition~\ref{prop:partial-balayage-obstacle} also implies that if
\begin{equation}\label{eq: k bound mu1 + mu2}
0 <k <  c_n \min  \{\gamma^{-1},(\mu_{1}+\mu_{2})(\mathbb{R}^n)^{-\frac{1}{n}}\},
\end{equation}
then $U_{k}^{\Bal(\mu_{1} + \mu_{2})}$ is the largest element of $\mathscr{F}_{k}(\mu_{1} + \mu_{2})$
and $U_{k}^{\Bal(\mu_{1}+\mu_{2})} = U_{k}^{\mu_{1}+\mu_{2}}$ near $\partial B_{R(n,k)}$. 

Finally, if we additionally assume that $\supp (\nu_{1}) \subset B_{\gamma}$ with $\nu_1=\Bal(\mu_{1})$, Proposition~\ref{prop:partial-balayage-obstacle} implies that if 
\begin{equation}\label{eq: k bound nu1 + mu2}
0 <k <  c_n \min  \{\gamma^{-1},(\nu_{1}+\mu_{2})(\mathbb{R}^n)^{-\frac{1}{n}}\},
\end{equation}
then $U_{k}^{\Bal(\nu_{1}+\mu_{2})}$ is the largest element of $\mathscr{F}_{k}(\nu_{1} + \mu_{2})$
and $U_{k}^{\Bal(\nu_{1}+\mu_{2})} = U_{k}^{\nu_{1}+\mu_{2}}$ near $\partial B_{R(n,k)}$. Using Proposition~{\rm \ref{prop:structure-balayage}} we observe that 
\[
\supp(\mu_{1}) \subset \supp(\mu_{1} + \mu_{2}) \subset \supp(\nu_{1} + \mu_{2}) \subset B_{\gamma}.
\]

We are now ready to prove the following proposition: 
\begin{prop}
\label{prop:iteration-balayage-operator}Let $\gamma>0$ and $\mu_{1},\mu_{2}\in L^{\infty}(B_\gamma)$ be non-negative and such that $\supp (\Bal(\mu_{1})) \subset B_{\gamma}$. If
\[
0 < k < c_n \min \{\gamma^{-1},(\mu_{1}+\mu_{2})(\mathbb{R}^n)^{-\frac{1}{n}},(\Bal(\mu_{1})+\mu_{2})(\mathbb{R}^n)^{-\frac{1}{n}}\} 
\]
with $c_n$ {\rm (}depending only on the dimension $n${\rm )}, then 
\[
\Bal(\mu_{1}+\mu_{2})=\Bal(\Bal(\mu_{1})+\mu_{2})
\]
and
\begin{equation*}
	\omega(\mu_1+\mu_2) = \omega(\mu_1) \cup \omega(\Bal(\mu_1)+\mu_2).
\end{equation*}
\end{prop}

\begin{proof}
Note that if $k$ satisfies the inequality in the proposition then $k$ satisfies the inequalities~\eqref{eq: k bound mu1},~\eqref{eq: k bound mu1 + mu2}, and~\eqref{eq: k bound nu1 + mu2}. 
We begin by showing the equality $\Bal(\mu_1+\mu_2)= \Bal(\Bal(\mu_1)+\mu_2)$.
Since $U^{\nu_1+\mu_2}_k = U_k^{\nu_1}+U_k^{\mu_2}$ and $U_k^{\nu_1}=U_k^{\mu_1}$ near $\partial B_{R(n, k)}$ it suffices to show that 
\begin{equation}
U_{k}^{\Bal(\nu_1+\mu_{2})} = U_{k}^{\Bal(\mu_{1}+\mu_{2})} \quad \text{in } B_{R(n,k)}.\label{eq:projection1}
\end{equation}

\medskip

\textbf{Step 1: The implication ``$\le$'' of \eqref{eq:projection1}.} Using Lemma~\ref{lem:basic-properties} we observe that 
\begin{align*}
U_{k}^{\Bal(\nu_1+\mu_{2})} & \le U_{k}^{\nu_1+\mu_{2}}=U_{k}^{\nu_1}+U_{k}^{\mu_{2}}\\
 & \le U_{k}^{\mu_{1}}+U_{k}^{\mu_{2}}=U_{k}^{\mu_{1}+\mu_{2}} \quad \text{in } B_{R(n,k)}
\end{align*}
and 
\[
(\Delta+k^{2})U_{k}^{\Bal(\nu_1+\mu_{2})} \ge -1 \quad \text{in } B_{R(n,k)}.
\]
Thus $U_{k}^{\Bal(\nu_1+\mu_{2})}\in \mathscr{F}(\mu_{1} + \mu_{2})$.
Since $U_{k}^{\Bal(\mu_{1} + \mu_{2})}$ is the largest element in $\mathscr{F}(\mu_{1} + \mu_{2})$, we arrive at 
\begin{equation}
U_{k}^{\Bal(\nu_1+\mu_{2})} \le U_{k}^{\Bal(\mu_{1} + \mu_{2})} \quad \text{in } B_{R(n,k)}.\label{eq:projection2}
\end{equation}

\medskip

\textbf{Step 2: The implication ``$\ge$'' of \eqref{eq:projection1}.} Observe that
\[
U_{k}^{\Bal(\mu_{1}+\mu_{2})}-U_{k}^{\mu_{2}}\le U_{k}^{\mu_{1}+\mu_{2}}-U_{k}^{\mu_{2}}=U_{k}^{\mu_{1}} \quad \text{in } B_{R(n,k)}
\]
and 
\[
(\Delta+k^{2})(U_{k}^{\Bal(\mu_{1}+\mu_{2})} - U_{k}^{\mu_{2}}) \ge - 1 + \mu_{2} \ge -1 \quad\text{in }B_{R(n,k)}.
\]
Thus $U_{k}^{\Bal(\mu_1+\mu_{2})}-U_{k}^{\mu_{2}}\in \mathscr{F}(\mu_{1})$.
Since $U_{k}^{\nu_{1}}$ is the largest element in $\mathscr{F}(\mu_{1})$, it holds that
\[
U_{k}^{\Bal(\mu_{1}+\mu_{2})} - U_{k}^{\mu_{2}} \le U_{k}^{\nu_1} \quad \text{in } B_{R(n,k)},
\]
and hence 
\[
U_{k}^{\Bal(\mu_{1}+\mu_{2})} \le U_{k}^{\nu_1} + U_{k}^{\mu_{2}} = U_{k}^{\nu_1+\mu_{2}} \quad \text{in } B_{R(n,k)}.
\]
Furthermore,
\[
(\Delta+k^{2})U_{k}^{\Bal(\mu_{1}+\mu_{2})} \ge -1 \quad \text{in } B_{R(n,k)}. 
\]
Thus $U_{k}^{\Bal(\mu_1+\mu_{2})}\in \mathscr{F}(\nu_{1}+\mu_2)$.
Since $U_{k}^{\Bal(\nu_1+\mu_{2})}$ is the largest element in $\mathscr{F}(\nu_1+\mu_{2})$, it follows that
\begin{equation}
U_{k}^{\Bal(\nu_1+\mu_{2})} \ge U_{k}^{\Bal(\mu_{1}+\mu_{2})} \quad \text{in } B_{R(n,k)}.\label{eq:projection3}
\end{equation}

\medskip

\textbf{Step 3: Conclusion.} Combining \eqref{eq:projection2} and \eqref{eq:projection3} implies \eqref{eq:projection1} and completes the proof that
\begin{equation*}
	\Bal(\mu_1+\mu_2) = \Bal(\nu_1+\mu_2).
\end{equation*}
In the proof of this equality we established that
\begin{align*}
	U^{\Bal(\mu_1+\mu_2)}_k = U^{\Bal(\nu_1+\mu_2)}_k \leq U^{\nu_1+\mu_2}_k = U^{\nu_1}_k+U^{\mu_2}_k \leq U^{\mu_1}_k + U_k^{\mu_2} = U_k^{\mu_1+\mu_2}\,.
\end{align*}
The first inequality is an equality only in $\omega(\Bal(\mu_1)+\mu_2)^c$ and the second is an equality only in $\omega(\mu_1)^c$. Therefore, the combined inequality, $U^{\Bal(\mu_1+\mu_2)}_k(x)\leq U_k^{\mu_1+\mu_2}(x)$ is an equality only for $x\in \omega(\nu_1+\mu_2)^c\cap \omega(\mu_1)^c = (\omega(\nu_1+\mu_2)\cup \omega(\mu_1))^c$. By definition, $\omega(\mu_1+\mu_2)$ is the set where this inequality is strict so the claim follows. This completes the proof of Proposition~\ref{prop:iteration-balayage-operator}.
\end{proof}

\section{Construction of \texorpdfstring{$k$}{k}-quadrature domains} \label{sec:Construction-partial-balayage}

In this section our aim is to prove the following theorem, which contains the statement of Theorem~{\rm \ref{thm:main2}}. 

\begin{thm}
\label{thm:k-quadrature-balayage} 
Let $\mu$ be a positive measure supported in $B_\epsilon$ for some $\epsilon>0$. There exists a constant $c_{n}>0$ depending only on the dimension such that if 
\begin{equation}
0 < k < \frac{c_{n}}{\mu(\mathbb{R}^n)^{1/n}} \quad \mbox{and}\quad \epsilon < c_n \mu(\mathbb{R}^n)^{1/n}, \label{eq:k-range}
\end{equation}
then there exists an open connected set $D$ with real-analytic boundary satisfying $\overline{D}\subset B_{R(n,k)}$ which is a $k$-quadrature domain for $\mu$. 
Moreover, for each $w\in L^1(D)\cap L^1(\mu)$ satisfying $(\Delta+k^2)w \ge 0$ in $D$ we have
\begin{equation}\label{eq:SL-analogue-repeat}
\int_{D} w(x)\,dx \ge \int w(x)\,d\mu(x). 
\end{equation}
\end{thm}

\begin{rem} 
As we shall see in the proof the $k$-quadrature domain is constructed as the non-contact set $\omega$ of the partial balayage of a measure obtained by averaging $\mu$ over a small ball. If $\mu$ satisfies the assumptions of Lemma~\ref{lem:non-empty} then $\Bal(\mu)$ is well-defined and the $k$-quadrature domain $D$ we construct is precisely $\omega(\mu)$.

However, our definitions and results concerning $\Bal(\mu)$ and $\omega(\mu)$ are not valid for every $\mu$ as in the statement of the theorem, and as such we need to take some care.
\end{rem}

Before we turn to the proof of Theorem~{\rm \ref{thm:k-quadrature-balayage}} we prove some preliminary results that we will need in our main argument.

The first is a simple lemma concerning the partial balayage of a multiple of Lebesgue measure restricted to a ball. 

\begin{lem} \label{lem:Ball}
Let $0 < r < r' < \frac{1}{2} j_{\frac{n-2}{2},1}k^{-1} = R(n,k)$. Then there exists a positive constant $c_{n}$ {\rm (}depending only on the dimension $n${\rm )} such that if
\begin{equation}
0 < k \le c_{n} \min \bigg\{ \frac{1}{r}, \frac{ J_{\frac{n}{2}}(kr)^{\frac{1}{n}} }{ J_{\frac{n}{2}}(kr')^{\frac{1}{n}} (rr')^{\frac{1}{2}}} \bigg\}, \label{eq:k-bound-patch-special}
\end{equation}
then
\begin{equation}
\Bal\bigg( \frac{c_{n,k,r'}^{\rm MVT}}{c_{n,k,r}^{\rm MVT}}\chi_{B_{r}}\mathsf{m} \bigg) = \chi_{B_{r'}}\mathsf{m} \label{eq:Ball-conclusion1}
\end{equation}
and 
\begin{equation}
\omega\bigg( \frac{c_{n,k,r'}^{\rm MVT}}{c_{n,k,r}^{\rm MVT}}\chi_{B_{r}}\mathsf{m} \bigg) =  
B_{r'}. \label{eq:000-refinement-omega-D-report}
\end{equation}
\end{lem}

\begin{rem}
Since $t \mapsto t^{\frac{n}{2}} J_{\frac{n}{2}}(t)$ is strictly increasing on $t \in [0,j_{\frac{n-2}{2},1}]$, then we see that 
\begin{equation}
\frac{c_{n,k,r'}^{\rm MVT}}{c_{n,k,r}^{\rm MVT}} = \frac{(r')^{n/2}J_{\frac{n}{2}}(kr')}{r^{n/2}J_{\frac{n}{2}}(kr)} > 1. \label{eq:000-constant-measure}
\end{equation}
Since $t \mapsto t^{-\frac{n}{2}}J_{\frac{n}{2}}(t)$ is a decreasing function on $[0, j_{\frac{n+2}{2},1}]$, we find that
\begin{equation}
\begin{aligned} 
& \quad \Bal\bigg( \frac{c_{n,k,r'}^{\rm MVT}}{c_{n,k,r}^{\rm MVT}}\chi_{B_{r}}\mathsf{m} \bigg) (\mathbb{R}^n) - \frac{c_{n,k,r'}^{\rm MVT}}{c_{n,k,r}^{\rm MVT}}\chi_{B_{r}}\mathsf{m} (\mathbb{R}^n) \\
&= \dm(B_1)((r')^n - \frac{(r')^{n/2}J_{\frac{n}{2}}(kr')}{r^{n/2}J_{\frac{n}{2}}(kr)} r^n) \\
	&= \dm(B_1)(r')^{\frac{n}{2}} k^{-\frac{n}{2}} J_{\frac{n}{2}}(kr')\Bigl(\frac{(kr')^{\frac{n}{2}}}{J_{\frac{n}{2}}(kr')}-\frac{(kr)^{\frac{n}{2}}}{J_{\frac{n}{2}}(kr)}\Bigr)\geq 0. 
\end{aligned}\label{eq: comparison measure Balayage of ball}
\end{equation}
\end{rem}

\begin{proof}[Proof of Lemma~{\rm \ref{lem:Ball}}]
For each $x \in \mathbb{R}^{n}$, we see that the distribution $y \mapsto \tilde{\Phi}_{k}(x-y)$ is in $L_{\rm loc}^{1}(\mathbb{R}^{n})$ and satisfies $(\Delta + k^{2}) \tilde{\Phi}_{k}(x - \cdot) = -\delta_{x} \le 0$ in $\mathbb{R}^{n}$. By applying the MVT in Proposition~{\rm \ref{prop:MVT-metaharmonic}}, we have 
\begin{equation}
\frac{1}{c_{n,k,r}^{\rm MVT}} U_{k}^{B_{r}}(x) = \frac{1}{c_{n,k,r}^{\rm MVT}} \int_{B_{r}} \tilde{\Phi}_{k}(x-y) \,dy \ge \frac{1}{c_{n,k,r'}^{\rm MVT}} \int_{B_{r'}} \tilde{\Phi}_{k}(x-y) \,dy = \frac{1}{c_{n,k,r'}^{\rm MVT}} U_{k}^{B_{r'}}(x) \label{eq:000-MVT-application-old-patch}
\end{equation}
for all $x \in \mathbb{R}^{n}$, and equality holds if and only if $x \in \mathbb{R}^{n} \setminus B_{r'}$. In other words
\[
u
= \frac{c_{n,k,r'}^{\rm MVT}}{c_{n,k,r}^{\rm MVT}}U_{k}^{B_{r}} - U_{k}^{B_{r'}} \in C^{1}(\mathbb{R}^{n})
\]
satisfies 
\begin{equation}
\begin{cases}
(\Delta + k^{2})u = \chi_{B_{r'}} - \frac{c_{n,k,r'}^{\rm MVT}}{c_{n,k,r}^{\rm MVT}} \chi_{B_{r}} &\text{in }\mathbb{R}^{n} \\
u > 0 &\text{in }B_{r'} \\
u = 0 &\text{in }\mathbb{R}^{n}\setminus B_{r'}.
\end{cases} \label{eq:patch1-condition-except-positivity}
\end{equation}
The conclusion of the lemma follows by applying Lemma~\ref{lem:observation-new-patch}.
\end{proof}

The second result we require is an analogue of Proposition~{\rm \ref{prop_runge_lone}} but for sub-solutions of the Helmholtz equation. We again follow the argument in \cite[Lemma 5.1]{Sak84ObstacleGreenObstacles} which considered the case $k=0$.

\begin{prop} \label{prop:prop_runge_lone-subsolution}
Let $k\ge0$, and let $D\subset\mathbb{R}^{n}$ be a bounded open set. Let $\Psi_{k}$ be any fundamental solution of $-(\Delta+k^{2})$ and let $\Omega\supset\overline{D}$ be any open set in $\mathbb{R}^{n}$. Then the linear span \emph{with positive coefficients} of
\[
F=\{\pm\partial^{\alpha}\Psi_{k}(z-\cdot)|_{D}:z\in\Omega\setminus D,|\alpha|\le1\} \cup \{ -\Psi_{k}(z-\cdot)|_{D}:z\in D\}
\]
is dense in 
\[
S_k L^1(D) =\{w\in L^{1}(D):(\Delta+k^{2})w \ge 0\text{ in }D\}
\]
with respect to the $L^{1}(D)$ topology. 
\end{prop}

\begin{proof}
We first show that if any bounded linear functional $\ell$ on $L^{1}(D)$ with $\ell|_{F}\ge 0$ also satisfies  
\begin{equation}
\ell|_{S_k L^1(D)}\ge0, \label{eq:HB-assump}
\end{equation}
then we have $\overline{G} = S_{k}L^{1}(D)$, where $G$ is the linear span with positive coefficients of $F$. Suppose to the contrary that there exists $f_{0} \in S_{k}L^{1}(D) \setminus \overline{G}$. Using the Hahn-Banach theorem (second geometric form, see e.g.\  \cite[Theorem~1.7]{Bre11PDE}), there exists a closed hyperplane $\{ \ell_{0} = \alpha \}$ that strictly separates the closed set $\overline{G}$ and the compact set $\{f_{0}\}$, thus we have 
\begin{equation}\label{eq:Hahn-Banach separation}
\ell_{0}(f_{0}) < \alpha < \ell_{0}(f) \quad \text{for all } f \in G.
\end{equation} 
Since $\lambda \ell_{0}(f) = \ell_{0}(\lambda f) > \alpha$ for all $\lambda > 0$ and each fixed $f \in G$, we deduce that $\ell_{0}(f) \ge 0$ for all $f \in G$ and that $\alpha \leq 0$. By \eqref{eq:HB-assump} and since $f_{0} \in S_{k}L^{1}(D)$ we know that $\ell_{0}(f_{0}) \ge 0$. Combining this with \eqref{eq:Hahn-Banach separation} and the fact that $\alpha \leq 0$ gives a contradiction.

\medskip 

Now let $\ell$ be a bounded linear functional on $L^1(D)$ with $\ell|_F \geq 0$. We need to prove that $\ell|_{S_k L^1(D)}\ge0$. Since the dual of $L^{1}(D)$ is $L^{\infty}(D)$, there is a function $f\in L^{\infty}(D)$ with 
\[
\ell(w)=\int_{D}fw\,dx,\quad w\in L^{1}(D).
\]
We extend $f$ by zero to $\mathbb{R}^{n}$ and consider the function
\[
u(z)=-(\Psi_{k}*f)(z)\quad\text{for all }z\in\Omega.
\]
By the assumption $\ell|_{F}\ge0$, the function $u$ satisfies 
\[
\begin{cases}
(\Delta+k^{2})u=f & \text{in }\Omega,\\
u=|\nabla u|=0 & \text{in }\Omega\setminus D,\\
u \ge 0 & \text{in }D.
\end{cases}
\]

\medskip 

Our aim is to employ the same argument as in the proof of \eqref{runge_ibp} to show that 
\[
\int_{D}((\Delta+k^{2})u)w\,dx\ge0, \quad \text{for all }w\in S_k L^1(D),
\]
which implies $\ell|_{S_k L^1(D)}\ge0$. However, in order to carry out the the integration by parts which concluded that argument we used that solutions of the Helmholtz equation are smooth in the interior of $D$, this is not necessarily the case for sub-solutions. To circumvent this issue we use a classical mollification argument. Fix a non-negative $\psi \in C_c^\infty(\mathbb{R}^n)$ with support in $B_1$ and $\|\psi\|_{L^1(\mathbb{R}^n)}=1$. For $\eps>0$ set $\psi_\eps(x) = \eps^{-n}\psi(x/\eps)$. For $w\in S_kL^1(D)$ set $w_\eps = w*\psi_\eps$, which is well defined and $C^{\infty}$ near any compact subset $K$ of $D$ if $\eps<\dist(K, D^c)$. Then $w_\eps \to w$ in $L^1(K)$ as $\eps \to 0^+$. We claim that $(\Delta+k^2)w_\eps(x)\geq 0$ in $K$ for $\eps<\dist(K, D^c)$. Indeed, since $(\Delta+k^2)w(x)\geq 0$ in $D$ this is in particular the case in $B_\eps(y)$ for any $y\in K$. The claim follows by differentiating under the integral sign and using the non-negativity of $\psi$.
With this approximation in hand the argument can be completed as in the proof of \eqref{runge_ibp} by appealing to the $L^1$ convergence of $w_\eps$ to $w$ in the support of the cutoff function $\omega_j$, and applying the integration by parts argument with $w$ replaced by $w_\eps$.
\end{proof}

Finally we need the following result which is in the spirit of Proposition~{\rm \ref{prop_quadrature_pde_equivalence}}. This result can be interpreted as saying that $D$ is a quadrature domain for sub-solutions $w$ satisfying $(\Delta + k^2) w \geq 0$ in $D$ (i.e.\ $D$ is a quadrature domain for metasubharmonic functions).

\begin{cor}
\label{cor:patch1} Let $k>0$, and let $D, \Omega \subset\mathbb{R}^{n}$ be bounded open sets such that $\overline{D}\subset \Omega$, and let $\mu\in L^\infty(D)$ be a non-negative measure with $\supp(\mu)\subset D$. If 
\begin{subequations}
\begin{align}
U_{k}^{D} & =U_{k}^{\mu}\quad\text{in } \Omega\setminus D,\label{eq:patch1-equation1a}\\
U_{k}^{D} & \le U_{k}^{\mu}\quad\text{in }\Omega,\label{eq:patch1-equation1b}
\end{align}
\end{subequations}
then for each $w\in L^{1}(D)$ satisfying $(\Delta+k^{2})w\ge0$ in $D$ we know that 
\begin{equation}
\int_{D}w(x)\,dx \ge \int w(x) \,d\mu(x) .\label{eq:patch1-goal}
\end{equation}
\end{cor}

\begin{proof}
By equations~\eqref{eq:patch1-equation1a} and~\eqref{eq:patch1-equation1b} $U_k^\mu-U_k^{D}\geq 0$ with equality in $\Omega \setminus D$. By Calder\'{o}n--Zygmund estimates $U_k^\mu, U_k^{D}\in C^1(\Omega)$. Since $U_k^\mu-U_k^D$ attains its minimum in $\Omega \setminus D$ it holds that $\nabla U_k^\mu = \nabla U_k^D$ in $\Omega \setminus D$. When combined with~\eqref{eq:patch1-equation1a} and~\eqref{eq:patch1-equation1b} we conclude that 
\begin{subequations}
\begin{align}
\int_{D}\partial^{\alpha}\tilde{\Phi}_{k}(z-x)\,dx & =\int\partial^{\alpha}\tilde{\Phi}_{k}(z-x)\,d\mu(x)\quad\text{for all }z\in B_{R(n,k)}\setminus D,|\alpha|\le1,\label{eq:patch1-equation2a}\\
\int_{D}\tilde{\Phi}_{k}(z-x)\,dx & \le\int\tilde{\Phi}_{k}(z-x)\,d\mu(x)\quad\text{for all }z\in D.\label{eq:patch1-equation2b}
\end{align}
\end{subequations}

Let $w$ be the function as in the statement of the lemma, and use Proposition~{\rm \ref{prop:prop_runge_lone-subsolution}} to find a sequence
\[
w_{j}\in{\rm span}_{+}\begin{pmatrix}\{\pm\partial^{\alpha}\tilde{\Phi}_{k}(z-\cdot)|_{D}:z\in B_{R(n,k)}\setminus D,|\alpha|\le1\}\\
\cup\{-\tilde{\Phi}_{k}(z-\cdot)|_{D}: z\in D \}
\end{pmatrix}
\]
with $w_{j}\rightarrow w\in L^{1}(D)$. From \eqref{eq:patch1-equation2a} and \eqref{eq:patch1-equation2b}, we know that 
\begin{equation}
\int_D w_{j}(x)\,dx\ge \int w_{j}(x)\,d\mu(x) \quad\text{for all }j.\label{eq:patch1-equation3}
\end{equation}
Since $\mu \in L^{\infty}(D)$ H\"older's inequality implies that
\[
\Bigl|\int (w_{j}(x)-w(x))\,d\mu(x) \Bigr| \le \| \mu\|_{L^{\infty}(D)} \| w - w_{j} \|_{L^{1}(D)}.
\]
Taking the limit $j \rightarrow \infty$ in \eqref{eq:patch1-equation3} we therefore arrive at
\begin{equation*}
\int_{D}w(x)\,dx\ge \int w(x)\,d\mu(x).
\end{equation*}
This is the desired inequality~\eqref{eq:patch1-goal}, and thus completes the proof.
\end{proof}

We are now ready to prove Theorem~{\rm \ref{thm:k-quadrature-balayage}}.

\begin{proof} [Proof of Theorem~{\rm \ref{thm:k-quadrature-balayage}}]
\textbf{Step 1: Constructing the $k$-quadrature domain.} 
For $\epsilon<\delta < R(n, k)$ to be chosen set
\begin{equation}
h_\delta=\frac{1}{c_{n,k,\delta}^{{\rm MVT}}}\chi_{B_{\delta}}.\label{eq:auxiliary-h-convolution}
\end{equation}
Then
\begin{equation}
\mu*h_\delta(x)=\frac{\mu(B_{\delta}(x))}{c_{n,k,\delta}^{{\rm MVT}}}
\end{equation}
is non-negative and supported in $B_{\epsilon+\delta}$. Furthermore for all $x\in B_{\delta-\epsilon}$,
\begin{equation}
\mu*h_\delta(x)=\frac{\mu(B_{\delta}(x))}{c_{n,k,\delta}^{{\rm MVT}}}
= \frac{\mu(\mathbb{R}^n)}{c_{n,k,\delta}^{{\rm MVT}}}. \label{eq:concolution-restriction-in-small-ball}
\end{equation}

Let us for $\kappa \ge 1$ and $r>0$ define
\begin{equation*}
\mu_{\kappa, r} := \kappa \chi_{B_{r}}\mathsf{m}.
\end{equation*}
Set
\[
\mu_{1}:= \mu_{\kappa, r} \quad \text{and} \quad \mu_{2} := \mu*h_\delta - \mu_{\kappa, r}
\]
with 
\begin{equation}
1< \kappa \le \frac{\mu(\mathbb{R}^n)}{c_{n,k,\delta}^{{\rm MVT}}} \quad \mbox{and}\quad  0<r\leq \delta-\epsilon \label{eq:concentration-measure}
\end{equation}
to be chosen. 
Note that these choices of $r,\kappa$ imply that the measures $\mu_1, \mu_2$ are non-negative. Furthermore, both measures have bounded densities with respect to Lebesgue measure and
\begin{equation*}
\mu_1(\mathbb{R}^n) = \kappa \dm(B_r) \quad \mbox{and} \quad \mu_2(\mathbb{R}^n) = \frac{\mu(\mathbb{R}^n)\dm(B_\delta)}{c_{n,k,\delta}^{\rm MVT}} - \kappa \dm(B_r). 
\end{equation*}

Our aim is to appeal to Proposition~\ref{prop:iteration-balayage-operator} and perform an initial balayage of $\mu_1$ by utilizing Lemma~\ref{lem:Ball}. To this end we choose
\begin{equation*}
	\kappa = \frac{c_{n,k,r'}^{\rm MVT}}{c_{n,k,r}^{\rm MVT}} = \frac{(r')^{\frac{n}2}J_{\frac{n}2}(k r')}{r^{\frac{n}2}J_{\frac{n}2}(k r)}\,.
\end{equation*}
for some $r<r'< R(n,k)$. By Lemma~\ref{lem:Ball}, if
\begin{equation}\label{eq: initial balayage k bound}
	0<k <c_n \min\Bigl\{ \frac{1}{r}, \frac{J_{\frac{n}2}(k r)^{\frac{1}n}}{J_{\frac{n}2}(k r')^{\frac{1}n}(r r')^{\frac{1}2}}\Bigr\} = c_n \min\Bigl\{ \frac{1}{r}, \frac{1}{\kappa^{\frac{1}n}r}\Bigr\}  = \frac{c_n}{r}\,,
\end{equation}
then $\Bal(\mu_1)= \chi_{B_{r'}}\dm$ and by~\eqref{eq: comparison measure Balayage of ball} it holds that $\Bal(\mu_1)(\mathbb{R}^n) \geq \mu_1(\mathbb{R}^n) = \kappa |B_r|$.

Consequently, $(\mu_1+\mu_2)(\mathbb{R}^n)\leq (\Bal(\mu_1)+\mu_2)(\mathbb{R}^n)$ and thus if
\begin{equation}
0 <k < c_n \min\Bigl\{\frac{1}{r'}, \frac{1}{\delta +\epsilon}, 
\frac{1}{((c_{n,k,\delta}^{\rm MVT})^{-1}\mu(\mathbb{R}^n)\delta^n + 
(r')^n-\kappa r^n)^{1/n}}
\Bigr\},
\label{eq: mollification k bound}  
\end{equation}
then \eqref{eq: initial balayage k bound} is valid, since $r'>r$, and furthermore Proposition~{\rm \ref{prop:iteration-balayage-operator}} implies that
\begin{equation}\label{eq: balayage convolution}
\Bal(\mu*h_\delta)=\Bal(\mu_{1}+\mu_{2})=\Bal(\Bal(\mu_{1})+\mu_{2})=\Bal(\mu_{1,r'}+\mu*h_\delta-\mu_{\kappa, r})
\end{equation}
and
\begin{equation}\label{eq: ball in non-contact set}
	\omega(\mu*h_\delta) = \omega(\mu_1)\cup \omega(\Bal(\mu_1)+\mu_2) \supset B_{r'}\,,
\end{equation}
where we also used that $\omega(\mu_1)=B_{r'}$ by Lemma~\ref{lem:Ball}.

By construction
\begin{equation}
\mu_{1, r'} + \mu*h_\delta - \mu_{\kappa, r} = 0 \quad \text{outside } B_{r'}\cup B_{\epsilon+\delta} \label{eq:support-2-step-measure}
\end{equation}
and
\begin{equation}
\mu*h_\delta - \mu_{\kappa, r} \ge 0. \label{eq:to-verify1} 
\end{equation}
Therefore,
\begin{equation}
\mu_{1,r'}+\mu*h_\delta-\mu_{\kappa, r} \ge 1 \quad \text{in }B_{r'}. \label{eq:mass-2-step-measure}
\end{equation}
Consequently, if we choose our parameters to satisfy
\begin{equation}
r'> \epsilon + \delta, \label{eq:to-verify2} 
\end{equation}
then the fact that $\mu*h_{\delta} \in L^{\infty}(\mathbb{R}^{n})$ combined with Proposition~{\rm \ref{prop:structure-balayage}} implies
\begin{equation}
\Bal(\mu_{1,r'}+\mu*h_\delta-\mu_{\kappa,r}) = \chi_{D(\mu_{1,r'}+\mu*h_\delta-\mu_{\kappa,r})}\dm, \label{eq:convolution-balayage-special-form}
\end{equation}
and $B_{r'}\subset D(\mu_{1,r'}+\mu*h_\delta-\mu_{\kappa,r})$.

Combining \eqref{eq: balayage convolution} and \eqref{eq:convolution-balayage-special-form}, we have
\begin{equation*}
\Bal(\mu*h_\delta) = \Bal(\mu_{1,r'}+\mu*h_\delta-\mu_{\kappa, r}) = \chi_{D(\mu_{1,r'}+\mu*h_\delta-\mu_{\kappa, r})}\dm.
\end{equation*}
Using~\eqref{eq: ball in non-contact set} and that we shall choose our parameters so that $r'>\eps+\delta$ we find
$$\supp(\mu*h_\delta) \subset B_{r'} \subset \omega(\mu*h_\delta),$$ therefore Lemma~\ref{lem:observation-new-patch} implies that
\footnote{Note that $\mu*h_\delta$ does not necessarily have a density which is greater than $1$ on its support and so the structure of its partial balayage in \eqref{eq:balayage-domain-relation} does not follow directly from Proposition~{\rm \ref{prop:structure-balayage}}.}
\begin{equation}
\Bal(\mu*h_\delta) = \chi_{\omega(\mu*h_\delta)}\dm. 
\label{eq:balayage-domain-relation}
\end{equation}
Using that only one connected component of $\omega(\mu*h_\delta)$ intersects $\supp(\mu*h_\delta)$, we can argue as in \cite[Corollary 2.3]{Gus90QuadratureDomains} to prove that $\omega(\mu*h_\delta)$ is connected.

By Lemma~{\rm \ref{lem:equivalent-saturated-set1}} and the definition of $\omega(\mu*h_\delta)$, 
\begin{subequations}
\begin{align}
U_{k}^{\omega(\mu*h_\delta)} & =U_{k}^{\mu*h_\delta}=\tilde{\Phi}_{k}*\mu*h_\delta\quad\text{in}\;\;B_{R(n,k)}\setminus \omega(\mu*h_\delta),\mbox{ and} \label{eq:sub-k-quadrature-convolution-a}\\
U_{k}^{\omega(\mu*h_\delta)} & < U_{k}^{\mu*h_\delta}=\tilde{\Phi}_{k}*\mu*h_\delta\quad\text{in}\;\;\omega(\mu*h_\delta).\label{eq:sub-k-quadrature-convolution-b}
\end{align}
\end{subequations}
Under our assumptions, Corollary~{\rm \ref{cor:patch1}} implies that $\omega(\mu*h_\delta)$ is a $k$-quadrature domain for $\mu*h_{\delta}$ and furthermore we have the quadrature inequality for sub-solutions.

The MVT (Proposition~{\rm \ref{prop:MVT-metaharmonic}}) implies that $\tilde{\Phi}_{k}*h_\delta(y)\leq \tilde{\Phi}_{k}(y)$ for all $y \in \mathbb{R}^n$ and equality holds if $|y|\geq\delta$. Therefore, by the non-negativity of $\mu$,
\begin{align*}
   U_k^{\mu*h_\delta}(x) = \tilde \Phi_k*\mu * h_\delta(x)
    = \int (\tilde \Phi_k*h_\delta)(x-y)d\mu(y) \leq \tilde \Phi_k * \mu(x) = U^\mu_k(x)
\end{align*}
with equality if $\dist(x, \supp(\mu))\geq\delta$. In particular, since under our assumptions $$\{x\in \mathbb{R}^n: \dist(x, \supp(\mu))< \delta\} \subset B_{\eps+\delta} \subset B_{r'}\subset \omega(\mu*h_\delta)$$ we have equality for $x \in B_{R(n, k)}\setminus \omega(\mu*h_{\delta})$. We have thus arrived at
\footnote{
If $\mu \in L^{\infty}(B_{R(n,k)})$ then \eqref{eq:sub-k-quadrature-a} and \eqref{eq:sub-k-quadrature-b} combined with an application of Lemma~\ref{lem:observation-new-patch} implies that
\begin{equation*}
\omega(\mu)=\omega(\mu*h_\delta)  \quad \mbox{and} \quad \Bal(\mu) =\Bal(\mu*h_\delta) =\chi_{\omega(\mu)}\mathsf{m}. 
\end{equation*}}
\begin{subequations}
\begin{align}
U_{k}^{D(\mu*h_{\delta})} & =U_{k}^{\mu}\quad\text{in } B_{R(n,k)}\setminus \omega(\mu*h_\delta),\mbox{ and} \label{eq:sub-k-quadrature-a}\\
U_{k}^{D(\mu*h_{\delta})} & < U_{k}^{\mu}\quad\text{in } \omega(\mu*h_\delta). \label{eq:sub-k-quadrature-b}
\end{align}
\end{subequations}

We now show that we can choose the parameters $r, \delta, \kappa$ appropriately only depending on the measure $\mu$, specifically we shall choose them depending on $\epsilon, \mu(\mathbb{R}^n)$. We choose $\delta = 2\epsilon$ and $r=\epsilon$, and let $\gamma = k \epsilon$.

Since $t \mapsto t^{-\frac{n}{2}} J_{\frac{n}{2}}(t)$ is a decreasing function on $[0,j_{\frac{n-2}{2},1}]$ satisfying that $\lim_{y\rightarrow 0_{+}} y^{-\frac{n}{2}} J_{\frac{n}{2}}(y) = \frac{2^{-\frac{n}{2}}}{\Gamma(1 + \frac{n}{2})}$, we by using the explicit form of $c_{n,k,r}^{\rm MVT}$ find that 
\begin{equation}
(c_{n,k,2\epsilon}^{\rm MVT})^{-1}(2\epsilon)^{n} = \frac{\gamma^{\frac{n}{2}}}{\pi^{\frac{n}{2}} J_{\frac{n}{2}}(2\gamma)} \ge \frac{\Gamma(1 + \frac{n}{2})}{\pi^{\frac{n}{2}}}. \label{eq:cMVT-uniform-bound}
\end{equation}
The required bound on $k$~\eqref{eq: mollification k bound} is then valid if
\begin{equation}
	0 <k < c_n \min\Bigl\{\frac{1}{r'}, \frac{1}{\epsilon}, 
	\frac{1}{(\pi^{-\frac{n}{2}} \Gamma(1 + \frac{n}{2}) \mu(\mathbb{R}^n) + (r')^n)^{1/n}} \Bigr\}. \label{eq:mollification-k-bound-1}
\end{equation}

Assume that $c_{n} \le j_{\frac{n-2}{2},1}/4$ so that $0 \le \gamma \le j_{\frac{n-2}{2},1}/4$. Then, since $t \mapsto t^{\frac{n}{2}}J_{\frac{n}{2}}(t)$ is strictly increasing on $[0, j_{\frac{n-2}{2},1}]$ we can choose $r'=4r$ which ensures that \eqref{eq:to-verify2} is satisfied since
\begin{equation*}
	r'=4r=4\eps>3\eps=\eps+\delta,
\end{equation*}
and we have
\begin{equation*}
\kappa = \frac{c_{n,k,4r}^{\rm MVT}}{c_{n,k,r}^{\rm MVT}} = \frac{ (4r)^{\frac{n}{2}} J_{\frac{n}{2}}(4\gamma) }{ r^{\frac{n}{2}} J_{\frac{n}{2}}(\gamma) } = \frac{2^{n}J_{\frac{n}{2}}(4\gamma)}{J_{\frac{n}{2}}(\gamma)}>1.
\end{equation*}

We now want to find a sufficient condition so that \eqref{eq:concentration-measure} holds. By the choice of $\kappa$ and the definition of $\gamma$ what we need to verify is the inequality
\begin{equation}
\kappa = \frac{2^{n}J_{\frac{n}{2}}(4\gamma)}{J_{\frac{n}{2}}(\gamma)} \le  \frac{\mu(\mathbb{R}^n)}{c_{n,k,2\epsilon}^{\rm MVT}} = \frac{\mu(\mathbb{R}^n) \gamma^{\frac{n}{2}}}{(2\epsilon)^{n} \pi^{\frac{n}{2}} J_{\frac{n}{2}}(2\gamma)}, \label{eq:concentration-measure-1}
\end{equation}
or equivalently,
\begin{equation*}
4^{n} \pi^{\frac{n}{2}} \frac{J_{\frac{n}{2}}(4\gamma) J_{\frac{n}{2}}(2\gamma)}{J_{\frac{n}{2}}(\gamma) \gamma^{\frac{n}{2}}} \epsilon^{n} \le \mu(\mathbb{R}^n). 
\end{equation*}
Since the function
\[
\gamma \mapsto \frac{J_{\frac{n}{2}}(4\gamma) J_{\frac{n}{2}}(2\gamma)}{J_{\frac{n}{2}}(\gamma) \gamma^{\frac{n}{2}}}
\]
is continuous it is bounded from above for all $\gamma \in [0, j_{\frac{n-2}{2},1}/4]$ by a constant depending only on $n$. Therefore, the required bound \eqref{eq:concentration-measure} holds if we assume that
\begin{equation}
\epsilon \leq c_n \mu(\mathbb{R}^n)^{1/n} \label{eq:check}
\end{equation}
provided $c_n$ is chosen sufficiently small, specifically so that  
\[
c_{n} \le \min_{\gamma \in [0,j_{\frac{n-2}{2},1}/4] } \bigg( \frac{ J_{\frac{n}{2}}(\gamma) \gamma^{\frac{n}{2}} }{ 4^{n} \pi^{\frac{n}{2}} J_{\frac{n}{2}}(4\gamma) J_{\frac{n}{2}}(2\gamma) } \bigg)^{1/n}.
\]

Since we chose our parameters so that $r'=4\eps$, the require bound on $k$~\eqref{eq:mollification-k-bound-1} (and thus also \eqref{eq: mollification k bound} and \eqref{eq: initial balayage k bound}) is valid if  
\begin{equation}
	0 <k < \frac{c_n}{\mu(\mathbb{R}^n)^{1/n}}. \label{eq:mollification-k-bound-2}
\end{equation}
Consequently, all our requirements are met provided
\begin{equation}
	0 <k <  
	\frac{c_n}{\mu(\mathbb{R}^n)^{1/n}} \quad \mbox{and} \quad \epsilon \leq c_n \mu(\mathbb{R}^n)^{1/n}
\end{equation}
for some constant $c_n$ depending only on $n$.

From here and on we let $h$ denote the function $h_\delta$ with the particular choice $\delta = 2\eps$ in the discussion above.
By the construction we have that 
\begin{equation}
\Bal(\mu*h) = \chi_{\omega(\mu*h)}\dm \label{eq:construction-QD-main}
\end{equation}
and
\begin{equation}\label{eq: support inclusion}
 \supp(\mu) \subset \overline{B_{\epsilon}} \subset \supp(\mu*h)\subset B_{4\epsilon}\subset \omega(\mu*h).
\end{equation}

\medskip

\textbf{Step 2: $\omega(\mu*h)$ is a $k$-quadrature domain with respect to $\mu$.} Equations \eqref{eq:sub-k-quadrature-a} and \eqref{eq:sub-k-quadrature-b} imply that $U_{k}^{\mu} - U_k^{\omega(\mu*h)}\geq 0$ with equality in $B_{R(n,k)}\setminus \omega(\mu*h)$. By elliptic regularity $U_{k}^{\omega(\mu*h)}$ is $C^1(B_{R(n, k)})$ and $U_{k}^{\mu}$ is smooth away from $\supp(\mu)$. Thus $U_{k}^{\mu} - U_k^{\omega(\mu*h)}$ attains its minimum in $B_{R(n,k)}\setminus \omega(\mu*h)$, and consequently
\begin{equation}
\nabla U_{k}^{\omega(\mu*h)} = \nabla U_{k}^{\mu} \quad \text{in } B_{R(n,k)} \setminus \omega(\mu*h). \label{eq:sub-k-quadrature-c}
\end{equation} 
Since $\supp(\mu) \subset \omega(\mu*h)\subset B_{R(n,k)}$ the extension by zero of $U_{k}^{\mu} - U_{k}^{\omega(\mu*h)}$ to all of $\mathbb{R}^n$ satisfies the assumptions of Proposition~\ref{prop_quadrature_pde_equivalence} and hence $\omega(\mu*h)$ is a $k$-quadrature domain with respect to $\mu$.

As noted above, \eqref{eq:sub-k-quadrature-convolution-a}, \eqref{eq:sub-k-quadrature-convolution-b} and Corollary~\ref{cor:patch1} imply that
\begin{equation*}
\int_{\omega(\mu*h)}w(x)\,dx \geq \int w(x)\,d(\mu*h)(x)
\end{equation*}
for all $w \in L^{1}(\omega(\mu*h))$ satisfying $(\Delta+k^2)w(x) \geq 0$ in $\omega(\mu*h)$. Assume further that $w \in L^1(\mu)$. By Fubini and since by construction $B_\delta(y) \subset \omega(\mu*h)$ for all $y\in \supp(\mu)$ the integral above can be rewritten as
\begin{align*}
    \int_{\omega(\mu*h)}w(x)\,d(\mu*h)(x) &= \int \biggl(\int_{\omega(\mu*h)} w(x)h_\delta(x-y)\,dx\biggr)\,d\mu(y)\\
    &= \int \biggl(\frac{1}{c_{n,k,\delta}^{\rm MVT}}\int_{B_\delta(y)} w(x)\,dx\biggr)\,d\mu(y).
\end{align*}
Since $B_\delta(y)\subset \omega(\mu*h)$ for all $y\in \supp(\mu)$ the function $w(x)$ is a sub-solution of the Helmholtz equation on $B_\delta(y)$ for all $y \in \supp(\mu)$. Therefore, the mean value inequality for sub-solutions of the Helmholtz equation implies that the expression within the parenthesis is greater than $w(y)$ for each $y \in \supp(\mu)$. Since $w \in L^1(\mu)$ and $\mu$ is non-negative this proves~\eqref{eq:SL-analogue-repeat}.

\medskip

\textbf{Step 3: Regularity of $\partial \omega(\mu*h)$}. We now show that $\omega(\mu*h)$ has real-analytic boundary $\partial \omega(\mu*h)$ by using the moving plane technique as in \cite[Theorem~5.4]{Gus04LecturesBalayage}. 
Set $u := U_{k}^{\mu* h} - U_{k}^{\omega(\mu*h)} \in \bigcap_{0<\alpha<1} C^{1,\alpha}(\overline{B_{R(n,k)}})$. By Proposition~\ref{prop:partial-balayage-obstacle}, we know that $u$ is the smallest among all $w \in H_{0}^{1}(B_{R(n,k)})$ satisfying 
\begin{equation}
w \ge 0 \quad \text{and} \quad (\Delta + k^{2})w \le - \mu*h + 1 \quad \text{in }B_{R(n,k)}. \label{eq:u-smallest-element}
\end{equation}
Moreover, Lemma~\ref{lem:equivalent-saturated-set1} and the fact that $\Bal(\mu*h)=\chi_{\omega(\mu*h)}\dm$ implies
\begin{equation}
u = 0 \quad \text{in } B_{R(n,k)} \setminus \omega(\mu*h). \label{eq:u-zero-ouside}
\end{equation}
Given $x_{0} \in \partial \omega(\mu*h)$, there by~\eqref{eq: support inclusion} exists a hyperplane that separates $\supp(\mu*h)$ and $x_{0}$. Since Laplacian is translation and rotation invariant, without loss of generality, we may assume that the hyperplane is 
\begin{equation}
\begin{Bmatrix}\begin{array}{l|l} x\in\mathbb{R}^{n} & x_{n}=0 \end{array}\end{Bmatrix}, \label{eq:hyperplane-coordinate}
\end{equation}
and that $\supp(\mu*h) \subset \begin{Bmatrix}\begin{array}{l|l} x\in\mathbb{R}^{n} & x_{n}<0 \end{array}\end{Bmatrix}$. We define
\[
(\omega(\mu*h))_{\rm loc} = \omega(\mu*h) \cap \begin{Bmatrix}\begin{array}{l|l} x\in\mathbb{R}^{n} & x_{n}>0 \end{array}\end{Bmatrix}
\]
and
\[
(\partial \omega(\mu*h))_{\rm loc} = \partial \omega(\mu*h) \cap \begin{Bmatrix}\begin{array}{l|l} x\in\mathbb{R}^{n} & x_{n}>0 \end{array}\end{Bmatrix} .
\]
Let $u^{*}$ be the reflection of $u$ with respect to the hyperplane \eqref{eq:hyperplane-coordinate}, that is, $u^{*}(x',x_{n}) = u(x',-x_{n})$. We now define 
\[
v := u - \inf\{u,u^{*}\} = (u-u^{*})_{+}.
\] 
Since $(\Delta + k^{2})u \le - \mu*h + 1 \le 1$ in $B_{R(n,k)}$, we have $(\Delta + k^{2})u^{*} \le 1$ in $B_{R(n,k)}$. Since there exists a unique $\phi \in H_{0}^{1}(B_{R(n,k)})$ such that $(\Delta + k^{2})\phi = 1$ in $B_{R(n,k)}$, using Proposition~\ref{prop: Min principle}, we know that 
\[
(\Delta + k^{2}) \inf\{u,u^{*}\} \le 1 \quad \text{in }B_{R(n,k)}.
\]
From $(\Delta + k^{2})u = 1$ in $(\omega(\mu*h))_{\rm loc}$, we have
\[
(\Delta + k^{2})v \ge 0 \quad \text{in }(\omega(\mu*h))_{\rm loc}.
\]
The boundary condition $v = 0$ on $\partial ((\omega(\mu*h))_{\rm loc})$ and using maximum principle in Proposition~\ref{prop: Max principle} yield $v \le 0$ in $(\omega(\mu*h))_{\rm loc}$, and hence 
\[
v = 0 \quad \text{in }(\omega(\mu*h))_{\rm loc},
\]
because $v \ge 0$ by its definition. Thus we have 
\begin{equation}
\frac{\partial u}{\partial x_{n}} \le 0 \quad \text{on } 
\begin{Bmatrix}\begin{array}{l|l} x\in\mathbb{R}^{n} & x_{n}=0 \end{array}\end{Bmatrix}. \label{eq:1-reflection}
\end{equation}
From \eqref{eq:u-zero-ouside}, we know that $\frac{\partial u}{\partial x_{n}} = 0$ on $(\partial \omega(\mu*h))_{\rm loc}$. On the other hand, we know that
\[
(\Delta + k^{2})u = 1 \quad \text{in }(\omega(\mu*h))_{\rm loc}.
\]
Hence $\frac{\partial u}{\partial x_{n}} \in C^{\infty}((\omega(\mu*h))_{\rm loc}) \cap \bigcap_{0<\alpha<1}C^{0,\alpha}(\overline{(\omega(\mu*h))_{\rm loc}})$ and it satisfies 
\[
(\Delta + k^{2}) \frac{\partial u}{\partial x_{n}} = \frac{\partial}{\partial x_{n}}(\Delta + k^{2})u = 0 \quad \text{in }(\omega(\mu*h))_{\rm loc}.
\]
Applying the strong maximum principle in Proposition~\ref{prop: Max principle} on $\frac{\partial u}{\partial x_{n}}$, we obtain $ \frac{\partial u}{\partial x_{n}} < 0$ in $(\omega(\mu*h))_{\rm loc}$ (because $ \frac{\partial u}{\partial x_{n}} \not\equiv 0$ in $\omega(\mu*h)$).

Since $x_0$ can be separated from $\supp(\mu*h)$ by hyperplanes whose normals form an open convex cone, this argument implies that in a neighbourhood of $x_0$ the function $u$ is decreasing in a cone of directions. We deduce that in a neighbourhood of $x_0$ the free boundary $\partial \omega(\mu*h)$ is the graph of a Lipschitz function. Since the choice of $x_0 \in \partial \omega(\mu*h)$ was arbitrary, we conclude that the free boundary $\partial \omega(\mu*h)$ is locally a Lipschitz graph. Using \cite{Caf77RegularityFreeBoundary,Caf80CompactnessFreeBoundary,Caf98ObstacleProblemRevisit}, we know that $\partial \omega(\mu*h)$ is $C^{1}$, and then from \cite{KN77FreeBoundary} we conclude that $\partial \omega(\mu*h)$ is real-analytic. We also refer to the monograph \cite{Fri88VariationalPrinciples} for the general regularity theory for free boundaries. 
\end{proof}

\appendix

\section{\label{appen:GreenFunction}Auxiliary propositions}


\addtocontents{toc}{\SkipTocEntry}
\subsection{A real-valued fundamental solution}

In this section we give an exact expression for a real-valued radial fundamental solution to the Helmholtz equation. This solution is positive in a ball with suitable radius, which is crucial for our construction of $k$-quadrature domains.

\begin{prop} \label{prop:GreenFunction}
Fix $k > 0$ and $n\geq 2$. For any $R>0$, let $\tilde{\Phi}_{k,R}$ be given by 
\begin{align}
\tilde{\Phi}_{k,R}(x)= 
\frac{ k^{\frac{n-2}{2}}}{4(2\pi)^{\frac{n-2}{2}}J_{\frac{n-2}{2}}(kR)}
|x|^{-\frac{n-2}{2}}\Bigl(Y_{\frac{n-2}{2}}(kR)J_{\frac{n-2}{2}}(k|x|)- J_{\frac{n-2}{2}}(kR)Y_{\frac{n-2}{2}}(k|x|)\Bigr). \label{tildephi_def}
\end{align}
Then the distribution $\tilde{\Phi}_{k,R} \in L_{\rm loc}^{1}(\mathbb{R}^{N})$ is radial, smooth outside the origin and satisfies
\begin{equation}
 \begin{cases}
 (\Delta +k^2) \tilde{\Phi}_{k,R} = -\delta_{0} \quad & \mbox{in } \mathcal{D}'(\mathbb{R}^n),\\
 \tilde{\Phi}_{k,R}(x)=0 \quad & \mbox{for }x\in\partial B_R(0).
 \end{cases}
\end{equation}
Furthermore, in the case when $0 < R < j_{\frac{n-2}{2}}k^{-1}$, the distribution $\tilde{\Phi}_{k,R}$ is positive in $B_{R}(0)$.
\end{prop}


\begin{rem} \label{rem:000-derivative-fundamental-solution}
Since for each $0 < t < j_{\frac{n-2}{2},1}k^{-1}$ and $0 < |x| < j_{\frac{n-2}{2},1}k^{-1}$ we have 
\begin{equation*}
\begin{aligned}
\frac{\partial}{\partial t} \tilde{\Phi}_{k,t}(x) & = 
\frac{ k^{\frac{n-2}{2}}}{4(2\pi)^{\frac{n-2}{2}}} |x|^{-\frac{n-2}{2}} J_{\frac{n-2}{2}}(k|x|)  \frac{\partial}{\partial t} \bigg( \frac{Y_{\frac{n-2}{2}}(kt)}{J_{\frac{n-2}{2}}(kt)} \bigg) \\ 
&= \frac{ k^{\frac{n-2}{2}}}{2(2\pi)^{\frac{n-2}{2}} \pi} |x|^{-\frac{n-2}{2}}  \frac{J_{\frac{n-2}{2}}(k|x|)}{t J_{\frac{n-2}{2}}(kt)^{2} } > 0,
\end{aligned}
\end{equation*}
then for each $0 < s < r < R < j_{\frac{n-2}{2},1}k^{-1}$ we have 
\begin{equation}
0 \le \tilde{\Phi}_{k,s} < \tilde{\Phi}_{k,r} \text{ in } B_{s} \quad \text{and} \quad \tilde{\Phi}_{r} \ge 0 \text{ in }B_{r}. \label{eq:monotonicity-fundamental-solution}
\end{equation}
\end{rem}

\begin{proof}
Clearly $\tilde{\Phi}_{k,R}$ is smooth outside the origin and vanishes on $\partial B_R(0)$. A standard calculation verifies that in the distributional sense $(\Delta+k^2)\tilde{\Phi}_{k,R} = -\delta_{0}$. 

It remains to prove that $\tilde{\Phi}_{k,R}$ is positive in $B_R(0)$. By the asymptotic behaviour of Bessel functions at $0$ one readily checks that $\lim_{|x|\to 0} \tilde{\Phi}_{k, R}(x)=+\infty$. Similarly, one computes 
\begin{equation}
\begin{aligned}
 &\quad \frac{x}{|x|}\cdot \nabla \tilde{\Phi}_{k,R}(x) \\
 &= 
 -\frac{k^{\frac{n}{2}}}{4(2\pi)^{\frac{n-2}{2}}J_{\frac{n-2}{2}}(kR))}
 |x|^{-\frac{n-2}{2}}\Bigl(Y_{\frac{n-2}{2}}(kR)J_{\frac{n}{2}}(k|x|)- J_{\frac{n-2}{2}}(kR)Y_{\frac{n}{2}}(k|x|)\Bigr). 
\end{aligned} \label{eq:radial-derivative-Uk1}
\end{equation}
When $|x|=R$ the right-hand side simplifies to
\begin{equation}
  -\frac{ k^{\frac{n-2}{2}}}{(2\pi)^{\frac{n-2}{2}} J_{\frac{n-2}{2}}(kR)} R^{-\frac{n}{2}},\label{eq:radial-derivative-Uk2}
\end{equation}
which is negative since $J_{\frac{n-2}{2}}$ is positive up to its first zero. Thus in order to show positivity of $\tilde{\Phi}_{k,R}$ in $B_R(0)$, it is enough to show that there is at most one $r \in (0,R)$ so that the radial derivative of $\tilde{\Phi}_{k,R}$ vanishes when $\abs{x}=r$.
 
Finally, if $r\in (0, R)$ is such that the radial derivative vanishes when $|x|=r$ the expression for the derivative implies that
\begin{equation}\label{eq: vanishing radial derivative}
\frac{Y_{\frac{n}{2}}(kr)}{J_{\frac{n}{2}}(kr)}=\frac{Y_{\frac{n-2}{2}}(kR)}{J_{\frac{n-2}{2}}(kR)}.
\end{equation}
Since
\begin{equation}
\frac{d}{dr}\frac{Y_{\frac{n}{2}}(kr)}{J_{\frac{n}{2}}(kr)} = \frac{2}{\pi r J_{\frac{n}{2}}(k r)^2}
\end{equation}
is positive on $(0, R)$ we conclude that there is at most one solution of~\eqref{eq: vanishing radial derivative} in $(0, R)$. This concludes the proof.
\end{proof}


\addtocontents{toc}{\SkipTocEntry}
\subsection{The mean value theorem}

\begin{prop} \label{prop:MVT-metaharmonic}
Let $n\ge2$ be an integer, and let $R>0$ be any constant. If $u \in L^1(B_R(x_0))$ is a solution to 
\[
(\Delta+k^{2})u=0\quad\text{in}\;\;B_{R}(x_{0}),
\]
then\labeltext{$c_{n,k,r}^{{\rm MVT}}$ constant related to mean value theorem}{index:ConstantMVTRepeat}
\begin{equation}
\begin{aligned}
\int_{B_{R}(x_0)}u(x)\,dx &= c_{n,k,R}^{{\rm MVT}}u(x_{0})\quad\text{with}\quad c_{n,k,R}^{{\rm MVT}} = 
(2\pi)^{n/2} \frac{R^{\frac{n}{2}}J_{\frac{n}{2}}(kR)}{k^{\frac{n}{2}}}.
\end{aligned}\label{eq:MVT-Helmholtz}
\end{equation}
In addition, if we assume that $0 < R < j_{\frac{n-2}{2},1}k^{-1}$ and $u \in L^{1}(B_{R}(x_{0}))$ is a sub-solution of the Helmholtz equation,
\[
(\Delta + k^{2}) u \geq 0 \quad \text{in }B_{R}(x_{0}),
\]
then 
\begin{equation}
\int_{B_{R}(x_{0})}u(x)\,d x\geq c_{n,k,R}^{{\rm MVT}}u(x_{0})\label{eq:MVT-super-Helmholtz}
\end{equation}
with equality if and only if $(\Delta+ k^2)u =0$ in $B_R(x_0)$.
In addition,
the mapping 
\[
r \mapsto \frac{1}{c_{n,k,r}^{{\rm MVT}}} \int_{B_{r}(x_{0})}u(x)\,d x
\]
is monotone increasing on $(0, R)$ unless there exists an $0<R'\leq R$ such that $(\Delta+k^2)u =0$ in $B_{R'}(x_0)$ in which case the mapping is constant on $(0, R')$ and increasing on $(R', R)$. 
\end{prop}

\begin{rem}
In particular, 
\begin{align*}
\text{when }n=2, & \quad c_{2,k,R}^{{\rm MVT}}=\frac{2\pi RJ_{1}(kR)}{k},\\
\text{when }n=3, & \quad c_{3,k,R}^{{\rm MVT}}=\frac{4\pi(\sin(kR)-kR\cos(kR))}{k^{3}}.
\end{align*}
Unlike the mean value theorem for harmonic functions, there are radii for which
$c^{{\rm MVT}}_{n,k,R}$ is zero or even negative.
\end{rem}


\begin{proof} [Proof of Proposition~{\rm \ref{prop:MVT-metaharmonic}}]
In what follows we shall assume additionally that $u \in C^\infty(B_R)$. For general sub-solutions the statement can be deduced by a standard mollification argument as in the proof of Proposition~\ref{prop:prop_runge_lone-subsolution}.

Without loss of generality, we may assume that $x_{0}=0$. For each $0<r<R$, using Proposition~\ref{prop:GreenFunction} as well as \eqref{eq:radial-derivative-Uk1} and \eqref{eq:radial-derivative-Uk2}, we have 
\begin{equation*}
\begin{aligned}
u(0) & =-\int_{B_{r}}(\Delta+k^{2})\tilde{\Phi}_{k,r}(x)u(x)\,dx \\ &
=-\int_{B_{r}}\bigg[\Delta \tilde{\Phi}_{k,r}(x)u(x)-\tilde{\Phi}_{k,r}(x)\Delta u(x)\bigg]\,dx  - \int_{B_{r}} \tilde{\Phi}_{k,r}(x) (\Delta + k^{2}) u (x) \,dx \\
 & =-\int_{\partial B_{r}}\bigg[\partial_{|x|}\tilde{\Phi}_{k,r}(x)u(x)-\overbrace{\tilde{\Phi}_{k,r}(x)}^{=\,0}\partial_{|x|}u(x)\bigg] dS  - \int_{B_{r}} \tilde{\Phi}_{k,r}(x) (\Delta + k^{2}) u (x) \,dx \\
 & =\frac{k^{\frac{n-2}{2}}}{(2\pi)^{\frac{n}{2}} J_{\frac{n-2}{2}}(kr)} r^{-\frac{n}{2}}\int_{\partial B_{r}}u(x)\, dS  - \int_{B_{r}} \tilde{\Phi}_{k,r}(x) (\Delta + k^{2}) u (x) \,dx , 
\end{aligned} 
\end{equation*}
that is, 
\begin{equation}
\begin{aligned}
&\quad \frac{k^{\frac{n-2}{2}}}{(2\pi)^{\frac{n}{2}}}\int_{\partial B_{r}}u(x)\,dS \\
&= J_{\frac{n-2}{2}}(kr)r^{\frac{n}{2}}u(0) + J_{\frac{n-2}{2}}(kr)r^{\frac{n}{2}} \int_{B_{r}} \tilde{\Phi}_{k,r}(x) (\Delta + k^{2}) u (x) \,dx.
\end{aligned} \label{eq:MVT-boundary}
\end{equation}
If $(\Delta + k^{2})u \geq 0$ in $B_{R}$ with $0<R<j_{\frac{n-2}2,1}k^{-1}$, then since $\tilde \Phi_r > 0$ in $B_r$ and $J_{\frac{n}{2}}(kr)>0$ for $r<j_{\frac{n-2}2,1}k^{-1}$, we conclude that 
\[
\frac{k^{\frac{n-2}{2}}}{(2\pi)^{\frac{n}{2}}}\int_{\partial B_{r}}u(x)\,dS \geq J_{\frac{n-2}{2}}(kr)r^{\frac{n}{2}}u(0),
\]
with equality if and only if $(\Delta+k^2)u=0$ in $B_r$. Therefore,
\[
\frac{k^{\frac{n-2}{2}}}{(2\pi)^{\frac{n}{2}}} \int_{B_{R}}u(x)\,dx \geq \bigg(\int_{0}^{R}J_{\frac{n-2}{2}}(kr)r^{\frac{n}{2}}\,dr\bigg)u(0)=\frac{R^{\frac{n}{2}}J_{\frac{n}{2}}(kR)}{k}u(0),
\]
with equality if and only if $(\Delta+k^2)u=0$ in $B_R$. This proves \eqref{eq:MVT-Helmholtz} and \eqref{eq:MVT-super-Helmholtz} for $R < j_{\frac{n-2}2,1}k^{-1}$, however for larger $R$ the argument can be followed step-by-step and all inequalities become equalities.

If $(\Delta + k^{2})u \ge 0$ in $B_{R}$, then by \eqref{eq:monotonicity-fundamental-solution} we have, for all $0 < s < r < R < j_{\frac{n-2}{2},1}k^{-1}$,  
\[
\quad \int_{B_{r}} \tilde{\Phi}_{k,r}(x) (\Delta + k^{2}) u (x) \,dx 
 \ge \int_{B_{s}} \tilde{\Phi}_{k,r}(x) (\Delta + k^{2}) u (x) \,dx \ge \int_{B_{s}} \tilde{\Phi}_{k,s}(x) (\Delta + k^{2}) u (x) \,dx,
\]
with equality if and only if $(\Delta+k^2)u=0$ in $B_r$. By~\eqref{eq:MVT-boundary} it follows that
\begin{equation}
\frac{1}{r^{\frac{n}{2}} J_{\frac{n-2}{2}}(kr)} \int_{\partial B_{r}}u(x)\, dS \ge \frac{1}{s^{\frac{n}{2}} J_{\frac{n-2}{2}}(ks)} \int_{\partial B_{s}}u(x)\, dS \label{eq:monotonicity-boundary-integral-1}
\end{equation}
with equality if and only if $(\Delta+k^2)u=0$ in $B_r$.
Note that 
\begin{equation}
\begin{aligned}
& \quad \frac{\partial}{\partial r} \bigg( \frac{1}{c_{n,k,r}^{{\rm MVT}}} \int_{B_{r}}u(x)\,dx \bigg) \\
&= (2\pi)^{n/2} k^{\frac{n}{2}} \bigg[ \frac{\partial}{\partial r}  \bigg( \frac{1}{r^{\frac{n}{2}}J_{\frac{n}{2}}(kr)} \bigg) \int_{B_{r}}u(x)\,dx + \frac{1}{r^{\frac{n}{2}}J_{\frac{n}{2}}(kr)} \int_{\partial B_{r}}u(x)\,dS \bigg] \\
&= (2\pi)^{n/2} k^{\frac{n}{2}} \bigg[ -\frac{k J_{\frac{n-2}{2}}(kr)}{r^{\frac{n}{2}} J_{\frac{n}{2}}(kr)^{2}} \int_{0}^{r} s^{\frac{n}{2}} J_{\frac{n-2}{2}}(ks) \bigg( \frac{1}{s^{\frac{n}{2}} J_{\frac{n-2}{2}}(ks)} \int_{\partial B_{s}}u(x)\,dS \bigg) \,ds \\
& \qquad + \frac{1}{r^{\frac{n}{2}}J_{\frac{n}{2}}(kr)} \int_{\partial B_{r}}u(x)\,dS \bigg]. 
\end{aligned}\label{eq:monotonicity-boundary-integral-2}
\end{equation}
From \eqref{eq:monotonicity-boundary-integral-1} and \eqref{eq:monotonicity-boundary-integral-2} we have
\begin{equation*}
\begin{aligned}
& \quad \frac{\partial}{\partial r} \bigg( \frac{1}{c_{n,k,r}^{{\rm MVT}}} \int_{B_{r}}u(x)\,dx \bigg) \\
& \ge \bigg( (2\pi)^{n/2} k^{\frac{n}{2}} r^{-\frac{n}{2}}\int_{\partial B_{r}}u(x)\,dS \bigg) \bigg( -\frac{1}{r^{\frac{n}{2}} J_{\frac{n}{2}}(kr)^{2}} \int_{0}^{r} k s^{\frac{n}{2}} J_{\frac{n-2}{2}}(ks) \,ds + \frac{1}{J_{\frac{n}{2}}(kr)} \bigg) \\
& = \bigg( (2\pi)^{n/2} k^{\frac{n}{2}} r^{-\frac{n}{2}}\int_{\partial B_{r}}u(x)\,dS \bigg) \bigg( -\frac{1}{r^{\frac{n}{2}} J_{\frac{n}{2}}(kr)^{2}} r^{\frac{n}{2}}J_{\frac{n}{2}}(kr) + \frac{1}{J_{\frac{n}{2}}(kr)} \bigg) \\
& = 0
\end{aligned}
\end{equation*}
for all $0 < r < R < j_{\frac{n-2}{2},1}k^{-1}$ and equality holds if and only if $(\Delta+k^2)u =0$ in $B_r$. This completes the proof of Proposition~\ref{prop:MVT-metaharmonic}.
\end{proof}


\addtocontents{toc}{\SkipTocEntry}
\subsection{Maximum principle}

We will need the following (generalized) maximum principle and properties of sub/super-solutions in small domains. 

\begin{prop}\label{prop: Max principle}
Fix $n \ge 2$, let $U\subset \mathbb{R}^n$ be a bounded open set, and let $\lambda_1(U)$ denote the first eigenvalue of the Dirichlet Laplacian on $U$, that is
\begin{equation}
\lambda_{1}(U) := \inf_{u \in H_{0}^{1}(U)} \frac{\| \nabla u \|_{L^{2}(U)}^{2}}{\| u \|_{L^{2}(U)}^2}. \label{eq:fundamental-tone-U}
\end{equation}
Given any $0<k^{2} < \lambda_{1}(U)$. If $w \in H^{1}(U)$ satisfies $w|_{\partial U} \le 0$ {\rm (}i.e. $w_{+} := \max \{ w,0 \} \in H_{0}^{1}(U)${\rm )} and $(\Delta + k^{2})w \ge 0$ in the sense of $H^{-1}(U)$, then $w \le 0$ in $U$. If we additionally assume that $w \in C(U)$, then in each connected component of $U$ we have either $w < 0$ or $w \equiv 0$.  
\end{prop}

\begin{prop}\label{prop: Min principle}
Fix $n \ge 2$, $k>0$, and let $U\subset \mathbb{R}^n$ be a bounded open set. If $w_1, w_2 \in H^1(U)$ satisfy $(\Delta+k^2)w_j \le 0$ in the sense of $H^{-1}(U)$ for $j=1$ and $2$, then the same is true for $w=\min\{w_1, w_2\}$.
\end{prop}

\begin{rem}
Note that $\lambda_{1}(B_{R}) = j_{\frac{n-2}{2},1}^{2}R^{-2}$. Therefore,  the condition $k^{2} < \lambda_{1}(U)$ is satisfied if  $U \subset B_{R}$ with $0 < R < j_{\frac{n-2}{2},1}k^{-1}$. 
\end{rem}


\begin{proof}[Proof of Proposition~{\rm \ref{prop: Max principle}}]
We observe that $\int_{U} \nabla w_{-} \cdot \nabla w_{+} \,dx = 0$ and test the weak formulation of the equation $(\Delta + k^{2})w \ge 0$ in $H^{-1}(U)$ with $w_{+} \in H_{0}^{1}(U)$. Consequently
\[
\| \nabla w_{+} \|_{L^{2}(U)}^{2} = \int_{U} \nabla w \cdot \nabla w_{+} \,dx = - \langle \Delta w,w_{+} \rangle \le k^{2} \| w_{+} \|_{L^{2}(U)}^{2}.
\]
By the Poincar\'{e} inequality 
\[
\| w_{+} \|_{L^{2}(U)}^{2} \le \frac{1}{\lambda_{1}(U)} \| \nabla w_{+} \|_{L^{2}(U)}^{2} \le \frac{k^{2}}{\lambda_{1}(U)} \| w_{+} \|_{L^{2}(U)}^{2}.
\]
Since $\frac{k^{2}}{\lambda_{1}(U)}<1$ it follows that $w_{+} \equiv 0$, and therefore $w \le 0$ in $U$.

Let us consider the case when $w \in C(U)$. Let $G$ be a connected component of $U$. If $w \equiv 0$, then we have nothing to prove. We consider the case when $w \not\equiv 0$ and we aim to prove that $w<0$ in $G$. Suppose the contrary, that there exists $x_{0} \in G$ such that $w(x_{0})=0$. Then by the mean value theorem (Proposition~{\rm \ref{prop:MVT-metaharmonic}}), 
\[
\int_{B_{\epsilon}(x_{0})} w(x) \,dx \ge 0
\]
for all $\epsilon>0$ with $\overline{B_{\epsilon}(x_{0})} \subset G$. However, since $w \le 0$ in $U$ and $w \in C(U)$ this implies that $w=0$ in $B_{\epsilon}(x_{0})$. Since $x_0$ was an arbitrary point where $w=0$ we have shown that the set $\{x\in G: w(x)=0\}$ is relatively open in $G$. Since $w$ is continuous this set is also relatively closed, and consequently either $w\equiv 0$ in $G$ or $w<0$ in $G$. This completes the proof of Proposition~\ref{prop: Max principle}.
\end{proof}

\begin{proof}[Proof of Proposition~{\rm \ref{prop: Min principle}}]
By a standard partition of unity argument it suffices to prove that there exists $r>0$ so that for any $x_0$ the conclusion is valid when restricting the distributions to $U \cap B_r(x_0)$.

Choose $r$ so small that $k^2< \lambda_1(B_r)$, i.e.\ $r<j_{\frac{n-2}2,1}k^{-1}$. Without loss of generality we may assume that $x_0=0$. By the choice of $r$ there exists a unique $u_{0} \in H^{1}(B_r)$ such that
\[
(\Delta + k^{2})u_{0}=0 \text{ in $B_r$},\quad u_{0}|_{\partial B_r}=1.
\]
By elliptic regularity $u_{0} \in C^{2}(\overline{B_r})$ and by the maximum principle in Proposition~{\rm \ref{prop: Max principle}} we know that $u_{0} > 0$ in $\overline{B_r}$. A direct computation yields that $(\Delta + k^{2})w_{j} \ge 0$ in $U\cap B_r$ if and only if $v_{j} = u_{0}^{-1}w_{j}$ satisfies 
\[
u_{0}^{-1} \nabla \cdot u_{0}^{2} \nabla v_{j} \ge 0 \quad \text{for $j=1,2$ in the sense of $H^{-1}(U\cap B_r)$.}
\]
By~\cite[Theorem~II.6.6]{KS00IntroductionVariationalInequalities} the same property holds for $v =\min\{v_1, v_2\}$. Since $u_{0}>0$ in $\overline{B_r}$ we find $w=\min\{w_1, w_2\} = u_{0} \min\{v_1, v_2\} = u_{0} v$. By the same computation as before (but in the opposite direction) one deduces that $(\Delta+k^2)w \le 0$. This concludes the proof of Proposition~\ref{prop: Min principle}.
\end{proof}



\addtocontents{toc}{\SkipTocEntry}
\subsection{Laplacian under analytic maps}
\label{appen:helmholtz_complex}

The following classical fact was used for seeing how the Helmholtz equation changes under an an analytic change of coordinates.

\begin{lem} \label{lemma_varphi}
Let $\varphi$ be analytic in an open set $U \subset \mathbb{C}$ and let $u$ be $C^{\infty}$ in $\varphi(U)$. Then 
\[
\Delta(u \circ \varphi)(z) = \Delta u(\varphi(z)) \abs{\varphi'(z)}^2.
\]
\end{lem}
\begin{proof}
We introduce the complex (Wirtinger) derivatives 
\[
\partial=\frac{1}{2}(\partial_{1}-i\partial_{2}),\quad\overline{\partial}=\frac{1}{2}(\partial_{1}+i\partial_{2}).
\]
Define $\tilde{u}(z):=u(\varphi(z))$ for $z\in U$. Note that $\varphi$ satisfies the Cauchy-Riemann equation $\overline{\partial}\varphi=0$, and write $\varphi'=\partial\varphi$. Using the chain rule (see e.g.\ \cite[equation~(2.48)]{AIM09PDEQuasiconformal}), we see that 
\begin{equation}
\overline{\partial}\tilde{u}(z)=\partial u\bigg|_{\varphi(z)}\overline{\partial}\varphi(z)+\overline{\partial}u\bigg|_{\varphi(z)}\overline{\partial\varphi(z)}=\overline{\partial}u\bigg|_{\varphi(z)}\overline{\partial\varphi(z)} \label{eq:first-der}
\end{equation}
and 
\begin{align*}
\frac{1}{4}\Delta\tilde{u}(z) & =\partial\overline{\partial}\tilde{u}(z) =\partial\overline{\partial}u\bigg|_{\varphi(z)}\partial\varphi(z)\overline{\partial\varphi(z)}+\overline{\partial}^{2}u\bigg|_{\varphi(z)}\partial\overline{\varphi}(z)\overline{\partial\varphi(z)} \\
& =\partial\overline{\partial}u\bigg|_{\varphi(z)}|\varphi'(z)|^{2} =\frac{1}{4}\Delta u\bigg|_{\varphi(z)}|\varphi'(z)|^{2}. \qedhere
\end{align*}
\end{proof}

\addtocontents{toc}{\SkipTocEntry}
\subsection{On the boundary of conformal images of the disk}
\label{appen: Riemann mapping regularity}

In this appendix we prove a number of results concerning the structure of the boundary of the two-dimensional $k$-quadrature domains constructed in Theorem~\ref{thm_quadrature_complex}. The results we shall prove are certainly well-known to experts in the field but since we have been unsuccessful in finding the precise statements in the literature we choose to include the proof.

Our first aim is to prove the following proposition which is essentially a restatement of Remark~\ref{rem:cusp_discussion}.
\begin{prop}\label{prop_conformal_boundary_regularity}
Let $\varphi$ be an analytic function in a neighbourhood of $\overline{\dd}$ which is injective in $\dd$ and set $D = \varphi(\dd)$. Then each $z\in \partial D$ falls into one of the following three categories.
\begin{enumerate}
\renewcommand{\labelenumi}{\theenumi}
\renewcommand{\theenumi}{(\Roman{enumi})}
 \item\label{itm_smoothpoints} (Smooth points) There is an $r>0$ so that $\partial D \cap B_r(z)$ is an analytic curve.
 
 \item\label{itm_cusppoints}
 (Inward cusp points) There is an $r>0$ so that $D\cap B_r(z)$ contains a semidisk and $\partial D \cap B_r(z)$ consists of two analytic curves ending and tangent to each other at $z$ and having no other common points.
 
 \item\label{itm_doublepoints}
 (Double points) There is an $r>0$ so that $\partial D \cap B_r(z)$ is the union of two analytic curves which are tangent to each other at $z$ and with no other common points. The set $D \cap B_r(z)$ consists of the two components into which the common normals to the analytic curves at $z$ point.
\end{enumerate}
Furthermore, there are only finitely many non-smooth points.
\end{prop}

The characterization of the different points in Proposition~\ref{prop_conformal_boundary_regularity} is classical. The sketch of proof that we provide focuses on proving the statement that the set of all non-smooth points is finite.

\begin{proof}[Sketch of the proof of Proposition~{\rm \ref{prop_conformal_boundary_regularity}}]
Recall that we in the proof of Theorem~\ref{thm_quadrature_complex} proved that $\varphi(\partial\mathbb{D})=\partial D$. We also note that since $\varphi$ is analytic in a neighbourhood of $\overline{\mathbb{D}}$ the set $D$ is bounded.
 
Let $z\in \partial D$ be the image of $\zeta \in \partial \mathbb{D}$. If $\varphi'(\zeta)\neq 0$ then the inverse function theorem for analytic mappings implies that $z$ satisfies the conclusion of~\ref{itm_smoothpoints}. Since $\varphi'$ is analytic in a neighbourhood of $\overline{\mathbb{D}}$ and $\partial \mathbb{D}$ is compact we find that $\varphi'$ either has finitely many zeroes on $\partial\mathbb{D}$ or $\varphi'\equiv 0$ in $\mathbb{D}$. Since $\varphi$ is injective in $\mathbb{D}$ we know that $\varphi' \not\equiv 0$. The injectivity of $\varphi$ in the interior implies that each zero of $\varphi'$ on $\partial \mathbb{D}$ is simple, see \ref{itm_cusppoints_varphi} in Remark~\ref{rem:cusp_discussion}. It is well-known that the zeroes of $\varphi'$ on $\partial \mathbb{D}$ correspond to $\partial D$ having inward cusps. 
 
Let $\{\zeta_1, \ldots, \zeta_k\}$ be the (finitely many) zeroes of $\varphi'$ on $\partial \mathbb{D}$ counted in counter-clockwise manner starting from the positive real axis. Let for $j=1, \ldots, k-1$ define $\gamma_j$ to be the open sub-arc of $\partial\mathbb{D}$ between the consecutive zeroes $\zeta_j$ and $\zeta_{j+1}$, and $\gamma_k$ be the open sub-arc from $\zeta_k$ to $\zeta_1$.
 
Finally, we want to show there are at most finitely many double points. We argue by contradiction. Suppose the contrary, that there are infinitely many double points. We define $\mathscr{D} := \{(\zeta^{-}_{\alpha}, \zeta^{+}_{\alpha})\}_{\alpha \in \Lambda}$ with $(\zeta^{-}_{\alpha}, \zeta^{+}_{\alpha}) \in \partial\mathbb{D}\times \partial\mathbb{D}$ such that $\zeta^{-}_{\alpha} \neq \zeta^{+}_{\alpha}$ and $\varphi(\zeta^{+}_{\alpha})=\varphi(\zeta^{-}_{\alpha})$. By compactness of $\partial\mathbb{D}\times \partial\mathbb{D}$, we can find a countable sequence of distinct pairs $\{(\zeta^{-}_{j}, \zeta^{+}_{j})\}_{j=1}^\infty$ in $\mathscr{D}$ such that 
\[
\lim_{j \rightarrow \infty} \zeta_{j}^{\pm} = \zeta^{\pm}
\]
for some $(\zeta^{-},\zeta^{+}) \in \mathscr{D}$.
We now divide our discussions into three cases:
\begin{enumerate}
 \item neither $\zeta^-$ nor $\zeta^+$ are zeros of $\varphi'$;
 \item one of the points $\zeta^-, \zeta^+$ is a zero of $\varphi'$ and the other is not;
 \item both $\zeta^{-}$ and $\zeta^{+}$ are zeros of $\varphi'$.
\end{enumerate}

\medskip

\textbf{Case 1: neither $\zeta^-$ nor $\zeta^+$ are zeros of $\varphi'$.} In this case, since $\varphi'$ is analytic and non-zero near both points $\zeta^{\pm} \in \partial \mathbb{D}$, then there exists two (sufficiently short) open arcs $\gamma^{\pm} \subset \partial \mathbb{D}$, passing through $\zeta^{\pm}$, with $\gamma^{-} \subset \gamma_{j}$ and $\gamma^{+} \subset \gamma_{j'}$, as well as $\varphi$ is analytic and invertible near $\gamma^{\pm}$. Without loss of generality, we may assume that $\zeta_{j}^{\pm} \in \gamma^{\pm}$ for $j=1,2,\cdots$. In particular, the arcs $\varphi(\gamma^{-})$ and $\varphi(\gamma^{+})$ intersect at infinitely many points.
Let $\varphi_{\rm loc}^{-1}$ denotes the local inverse of $\varphi$ near $\gamma^{+}$, and let $M_{\pm}$ be M\"{o}bius transforms with $M_{\pm}((0,1))=\gamma^{\pm}$ such that 
\[
G := M_{+}^{-1}\circ \varphi_{\rm loc}^{-1} \circ \varphi\circ M_{-}
\]
is analytic and bijective in a neighborhood of $[0, 1]$. By assuming $\gamma^{-}$ is sufficiently short, we can further assuming $\Re G((0,1)) \subset (0,1)$, where $\Re G$ is the real part of $G$.

For each $x_{j} := M_{-}^{-1}(\zeta_{j}^{-}) \in (0,1)$, we see that 
\[
\varphi_{\rm loc}^{-1} \circ \varphi \circ M_{-}(x_{j}) = \varphi_{\rm loc}^{-1} \circ \varphi(\zeta_{j}^{-}) = \zeta_{j}^{+} \in \gamma^{+},
\]
and hence $G(x_{j}) \in (0,1)$. This implies $\Im G(x_{j}) = 0$, where $\Im G$ is the imaginary part of $G$. Since $\zeta_{j}^{\pm}$ accumulate at $\zeta^{\pm} \in \gamma^{\pm}$, then $\Im G$ has an infinite number of zeros in $(0,1)$ and they accumulate at $x_{0} := M_{-}^{-1}(\zeta^{-}) \in (0,1)$. Since $\Im G$ is real-analytic, then $\Im G \equiv 0$. Hence we know that $G((0,1)) \subset (0,1)$, and consequently 
\[
\varphi(\gamma^{-}) \subset \varphi(\gamma^{+}),
\]
that is, $\varphi(\gamma^{-})$ is a sub-arc of $\varphi(\gamma^{+})$, and hence all points in $\varphi(\gamma^{-})$ are double points.

Let $p$ be the endpoint of $\varphi(\gamma^{-})$. If $p$ is not an endpoint of $\varphi(\gamma_{j})$ or $\varphi(\gamma_{j'})$ then by the arguments above there exists a neighborhood of $p$ in $\varphi(\gamma_{j})$ which is an sub-arc of $\varphi(\gamma_{j'})$. By continuity we conclude that $\varphi(\gamma_j)\cap \varphi(\gamma_{j'})$ contains an arc the endpoints of which are also endpoints of either $\varphi(\gamma_j)$ or $\varphi(\gamma_{j'})$. Since the endpoints of $\gamma_j, \gamma_{j'}$ are zeros of $\varphi'$ we have found an accumulation point of double-points which is the image of a zero of $\varphi'$, that is, we are in the setting of Cases 2 and 3. After analysing Cases 2 and 3 we shall conclude that $p$ is actually an endpoint of both $\varphi(\gamma_j)$ and $\varphi(\gamma_{j'})$, and hence $\varphi(\gamma_j)=\varphi(\gamma_{j'})$.

\medskip

\textbf{Case 2: one of the points $\zeta^-, \zeta^+$ is a zero of $\varphi'$ and the other is not.}
Without loss of generality, we may assume that $\varphi'(\zeta^-)=0$ and $\varphi'(\zeta^+)\neq 0$. The image of $\partial \mathbb{D}$ contained in a small neighbourhood of $\zeta^+$ is an analytic curve passing through $\varphi(\zeta^+)=\varphi(\zeta^-)$ while the image of a small neighbourhood of $\zeta^-$ is an inward cusp to $D$ with singular point at $\varphi(\zeta^+)=\varphi(\zeta^-)$. It is easy to deduce that the sets $\varphi(B_r(\zeta^+)\cap \mathbb{D})$ and $\varphi(B_r(\zeta^-)\cap \mathbb{D})$ must intersect for every $r>0$, which contradicts the bijectivity of $\varphi$. 
 
\medskip

\textbf{Case 3: both $\zeta^{-}$ and $\zeta^{+}$ are zeros of $\varphi'$.}
Note that both $\varphi(\zeta^{-})$ and $\varphi(\zeta^{+})$ are vertices of cusps. If $\zeta^{-} \neq \zeta^{+}$, then there are two different inward cusps sharing the same vertex which can be shown to violate that $\varphi$ maps $\mathbb{D}$ bijectively and continuously to $D$.  Thus we can assume that $\zeta^-=\zeta^+$. By rotation we can further assume that $\zeta^- = \zeta^+ = 1$. 

We define $\phi(\theta) = \varphi(e^{i\theta})$, and we are interested the behavior near $\theta = 0$. There exists $\epsilon > 0$ such that
\[
\phi(\theta) = \sum_{k=0}^{\infty} a_{k} \theta^{k} \quad \text{for all }\theta \in (-\epsilon,\epsilon).
\]
We now want to show that $\phi$ is even near 0, that is, $a_{k}=0$ for all odd integer $k$. If this is the case then $\phi((-\epsilon,0)) = \phi((0,\epsilon))$, that is two of the arcs $\{\varphi(\gamma_j)\}_j$ coincide along a sub-arc and in particular we have accumulation of double points in the interior of the arcs, that is we are in the setting of Case 1.

If $\phi$ is not even then there exists a smallest odd integer $k_{0}$ such that $a_{k_{0}} \neq 0$. By assumption $\varphi'(1) = 0$ which implies that $\phi'(0)=0$. This shows that $k_{0} \ge 3$. Using that
\[
\phi(\theta) = \sum_{k=0}^{k_{0}} a_{k} \theta^{k} +O(\theta^{k_0+1}),
\]
it is not difficult to see that there exists a sufficiently small $0 < \epsilon' < \epsilon$ such that there are only finitely many pairs of points $(\theta^{-}, \theta^{+}) \in (-\epsilon',\epsilon') \times (-\epsilon',\epsilon')$ with
\[
\varphi(e^{i\theta^-}) = \phi(\theta^{-}) = \phi(\theta^{+}) = \varphi(e^{i\theta^-}) \quad \text{and} \quad \theta^{-} < \theta^{+}.
\]
This contradicts that $\zeta^{-} = \zeta^{+}$ was an accumulation point the set of double points.

\medskip
 
\textbf{Conclusion.} By combining these three cases we conclude that if there exist an infinite number of double points then there exists $j\neq j'$ so that $\varphi(\gamma_j)= \varphi(\gamma_{j'})$. From the injectivity of $\varphi$ on $\mathbb{D}$ and the construction of the $\gamma_j$ one deduces that $D$ is the complement of the arc $\varphi(\gamma_j)$, but this is impossible under the assumptions on $\varphi$ since they imply that $D$ is bounded.
\end{proof}

For the readers' convenience we next consider the case of inward cusp points in more detail. Suppose that $z_0 = e^{i t_0} \in \p B_1$ and $\varphi'(z_0) = 0$. We let $\gamma(t) = e^{i(t + t_0)}$ be the boundary curve for $\p B_1$ with $\gamma(0) = z_0$. Then 
\[
\gamma^{(k)}(t) = i^k \gamma(t).
\]
Let $\eta(t) = \varphi(\gamma(t))$ be the corresponding boundary curve for $\p D$ with $\eta(0) = \varphi(z_0)$. Then $\eta$ is a smooth curve. We compute the derivatives of $\eta$:
\begin{align*}
\eta' &= \varphi'(\gamma) \gamma' = i \varphi'(\gamma) \gamma \\
\eta'' &= \varphi''(\gamma) (\gamma')^2 + \varphi'(\gamma) \gamma'' = i^2 (\varphi''(\gamma) \gamma^2 + \varphi'(\gamma) \gamma), \\
\eta''' &= i^3 (\varphi'''(\gamma) \gamma^3 + 3 \varphi''(\gamma) \gamma^2 + \varphi'(\gamma) \gamma).
\end{align*}
Evaluating at $t=0$ and using $\varphi'(z_0) = 0$ we obtain 
\begin{align*}
\eta(0) &= \varphi(z_0), \\
\eta'(0) &= 0, \\
\eta''(0) &= i^2 \varphi''(z_0) z_0^2, \\
\eta'''(0) &= i \eta''(0) \left( \frac{\varphi'''(z_0) z_0}{\varphi''(z_0)} + 3 \right).
\end{align*}

We now choose orthonormal coordinates $(x_1, x_2)$ so that the origin corresponds to $\varphi(z_0)$ and the vector $\eta''(0)$ corresponds to $\lambda e_1$ for some $\lambda > 0$. It follows that 
\[
\eta_1(t) = t^2 g(t)
\]
where $g$ is real-analytic near $0$ with $g(0) > 0$. We choose a new time variable $s(t)$ near $t=0$ as 
\[
s(t) = t \sqrt{g(t)}.
\]
This is a valid change of coordinates since $s$ is real-analytic near $0$ and satisfies $s'(0) > 0$. Writing $\tilde{\eta}(s) = \eta(t(s))$, we have $\tilde{\eta}_1(s) = s^2$. Moreover, since $\eta'(0) = 0$, we have 
\[
\tilde{\eta}'''(0) = (t'(0))^3 \eta'''(0) + 3 t'(0) t''(0) \eta''(0) = \eta''(0) \left( t'(0)^3 i \left( \frac{\varphi'''(z_0) z_0}{\varphi''(z_0)} + 3 \right) + 3 t'(0) t''(0) \right).
\]
Here $t'(0) > 0$, $t''(0) \in \mathbb{R}$ and $\eta''(0) = \lambda > 0$. If we assume additionally that
\begin{equation} \label{cusp_order_two_condition}
\frac{\varphi'''(z_0) z_0}{\varphi''(z_0)} + 3 \notin i \mathbb{R},
\end{equation}
then we have $\tilde{\eta}_2'''(0) \neq 0$.

We have proved that assuming \eqref{cusp_order_two_condition}, the boundary curve of $\p D$ near $\varphi(z_0)$ takes the form 
\begin{align*}
\tilde{\eta}_1(s) &= s^2, \\
\tilde{\eta}_2(s) &= h(s)
\end{align*}
where $h(0) = h'(0) = h''(0) = 0$ but $h'''(0) \neq 0$. This is the simplest possible (ordinary) cusp. In particular, $D$ has $C^0$ boundary near such a point.

If \eqref{cusp_order_two_condition} does not hold whenever $z_0 \in \p B_1$ and $\varphi'(z_0) = 0$, then $h(s)$ vanishes to order $\geq 4$ at $0$. If it vanishes to infinite order, then $h \equiv 0$ and the domain would look like the slit disk, but this is excluded by Proposition~\ref{prop_conformal_boundary_regularity}. 
If the vanishing order of $h$ is $k+1$ for some $k \geq 2$, then we have a singularity of type $A_k$. If $k$ is even, the boundary is $C^0$ near such a point, but if $k$ is odd then a curved cusp may occur, see Example~\ref{exa:curve-cusps} as well as Figure~\ref{fig:curve-cusps}. 


\bibliographystyle{custom}
\bibliography{ref}

\end{sloppypar}

\end{document}